\documentclass[12pt]{amsart}
\usepackage{amsmath,amsthm,amsfonts,amssymb,xcolor,enumitem,mathscinet}
\usepackage{graphicx}
\usepackage{mathrsfs}
\usepackage{etoolbox}
\apptocmd{\sloppy}{\hbadness 10000\relax}{}{}
\apptocmd{\sloppy}{\vbadness 10000\relax}{}{}
\usepackage[pdfpagelabels]{hyperref}
\usepackage[letterpaper,margin=1.1in]{geometry}

\newcommand{\N}{\mathbb{N}}
\newcommand{\R}{\mathbb{R}}

\newcommand{\E}{\mathbb{E}}

\newcommand{\ZZ}{\mathbb{Z}}
\newcommand{\RR}{\mathbb{R}}
\renewcommand{\P}{\mathbb{P}}

\newcommand{\var}{\mathop\mathrm{var}\nolimits}

\renewcommand{\phi}{\varphi}
\renewcommand{\hat}{\widehat}
\renewcommand{\tilde}{\widetilde}

\newcommand{\cK}{\mathcal{K}}
\newcommand{\cH}{\mathcal{H}}

\newcommand{\cE}{\mathcal{E}}

\DeclareMathOperator{\spt}{spt}
\newcommand{\tb}{\widetilde{\beta}}
\newcommand{\tJ}{\widetilde{J}}

\newcommand{\two}{{I\!I}}
\newcommand{\three}{{I\!I\!I}}

\newcommand{\bridges}{\mathsf{Bridges}}
\newcommand{\edges}{\mathsf{Edges}}
\newcommand{\phan}{\mathsf{Phantom}}
\newcommand{\leaves}{\mathsf{Leaves}}
\newcommand{\Child}{\Delta_1}
\newcommand{\Top}{\mathop\mathsf{Top}}

\newcommand{\near}{\mathsf{Near}}

\newcommand{\lD}[1]{\underline{D}^{#1}}

\newtheorem{theorem}{Theorem}[section]
\newtheorem{proposition}[theorem]{Proposition}
\newtheorem{lemma}[theorem]{Lemma}
\newtheorem{corollary}[theorem]{Corollary}

\theoremstyle{definition}
\newtheorem{definition}[theorem]{Definition}
\newtheorem{remark}[theorem]{Remark}

\numberwithin{equation}{section}
\numberwithin{figure}{section}

\newcommand{\rad}{\mathop\mathrm{rad}\nolimits}
\newcommand{\diam}{\mathop\mathrm{diam}\nolimits}
\newcommand{\dist}{\mathop\mathrm{dist}\nolimits}
\newcommand{\side}{\mathop\mathrm{side}\nolimits}

\newcommand{\res}{\hbox{ {\vrule height .22cm}{\leaders\hrule\hskip.2cm} }}

\newcommand{\Haus}{\mathcal{H}}
\newcommand{\Pack}{\mathcal{P}}

\newcommand{\rect}{\mathrm{rect}}
\newcommand{\pu}{\mathrm{pu}}

\allowdisplaybreaks

\begin{document}

\title{Identifying 1-rectifiable measures in Carnot groups}

\author{Matthew Badger \and Sean Li \and Scott Zimmerman}
\thanks{M.~Badger was partially supported by NSF DMS grants 1650546 and 2154047. S.~Li was partially supported by NSF DMS grant 1812879.}
\date{August 30, 2023}

\address{Department of Mathematics\\ University of Connecticut\\ Storrs, CT 06269-1009}
\email{matthew.badger@uconn.edu}
\address{Department of Mathematics\\ University of Connecticut\\ Storrs, CT 06269-1009}
\email{sean.li@uconn.edu}
\address{Department of Mathematics\\ The Ohio State University at Marion\\ Marion, OH 43302-5695}
\email{zimmerman.416@osu.edu}

\begin{abstract}We continue to develop a program in geometric measure theory that seeks to identify how measures in a space interact with canonical families of sets in the space. In particular, extending a theorem of the first author and R.~Schul in Euclidean space, for an arbitrary \emph{locally finite Borel measure} in an arbitrary \emph{Carnot group}, we develop tests that identify the part of the measure that is carried by rectifiable curves and the part of the measure that is singular to rectifiable curves. Our main result is entwined with an extension of the \emph{Analyst's Traveling Salesman Theorem}, which characterizes subsets of rectifiable curves in $\RR^2$ (P.~Jones, 1990), in $\RR^n$ (K.~Okikolu, 1992), or in an arbitrary Carnot group (the second author) in terms of local geometric least squares data called \emph{Jones' $\beta$-numbers}. In a secondary result, we implement the Garnett-Killip-Schul construction of a doubling measure in $\RR^n$ that charges a rectifiable curve in an arbitrary \emph{complete, doubling, locally quasiconvex metric space}.\end{abstract}

\maketitle

\tableofcontents

\renewcommand{\thepart}{\Roman{part}}

\section{Introduction} Rectifiability is an important concept in geometric measure theory that supplies a finer notion of regularity of a set or measure than does dimension \cite{rectifiability-survey,measure-dimension}. There is not a single definition of rectifiability, but rather a number of variations that may be encoded using a common framework \cite{ident}. For recent work on rectifiable sets and absolutely continuous measures in Carnot groups and in general metric spaces, we refer the reader to the papers \cite{AM3,AM2,AM1,Bate-unrectifiable,Bate-tangents,Bate-Li-rectifiable,Chousionis-Li-Young} and the references within. A current challenge that we address in this paper is to find characterizations of rectifiability of locally finite measures without imposing the traditional background hypothesis of absolute continuity. In other words, we are interested in detecting how a measure interacts with a prescribed family of sets, but we do not want to make \emph{a priori} assumptions about the null sets of the measure. Building on recent progress on this problem in Euclidean space \cite{badger-naples,badger-schul}, we give the first characterization of a class of rectifiable measures inside the collection of locally finite Borel measures in a non-Euclidean setting.

Following the convention in Morse-Randolph \cite{MR44}, Federer \cite{Federer}, and Badger-Schul \cite{BS1}, we say that a Borel measure $\mu$ on a metric space $X$ is \emph{1-rectifiable} if there exists a sequence $\Gamma_1,\Gamma_2,\dots$ of rectifiable curves in $X$ such that $\mu(X\setminus \bigcup_1^\infty \Gamma_i)=0$; at the other extreme, we say that $\mu$ is \emph{purely 1-unrectifiable} if $\mu(\Gamma)=0$ for every rectifiable curve $\Gamma$ in $X$. We reemphasize that unlike some treatments \cite{DeLellis, Mattila}, we do not impose the simplifying assumption that a 1-rectifiable measure is absolutely continuous with respect to $1$-dimensional Hausdorff measure $\Haus^1$. See \S\ref{ss:rect-curves} for a primer on rectifiable curves and Figure \ref{fig:basic-examples} for some simple examples of rectifiable and purely unrectifiable measures in $\RR^2$. An arbitrary measure is not necessarily rectifiable or purely unrectifiable. Nevertheless, by a general form of the Lebesgue decomposition theorem (see Lemma \ref{l:decomp}), every $\sigma$-finite Borel measure $\mu$ on a metric space $X$ can be written uniquely as \begin{equation}\label{mu-decomp} \mu=\mu_\rect+\mu_\pu,\end{equation} where $\mu_\rect$ is 1-rectifiable and $\mu_\pu$ is purely 1-rectifiable. Unfortunately, the proof that the decomposition \eqref{mu-decomp} exists is abstract and does not indicate how to find the component measures. In our main result (see Theorem \ref{t-main}), we identify the 1-rectifiable and purely 1-unrectifiable parts of an arbitrary locally finite measure on an arbitrary Carnot group equipped with a Hebisch-Sikora norm (see \S\ref{ss:Carnot}).

\begin{figure}\begin{center}\includegraphics[width=.5\textwidth]{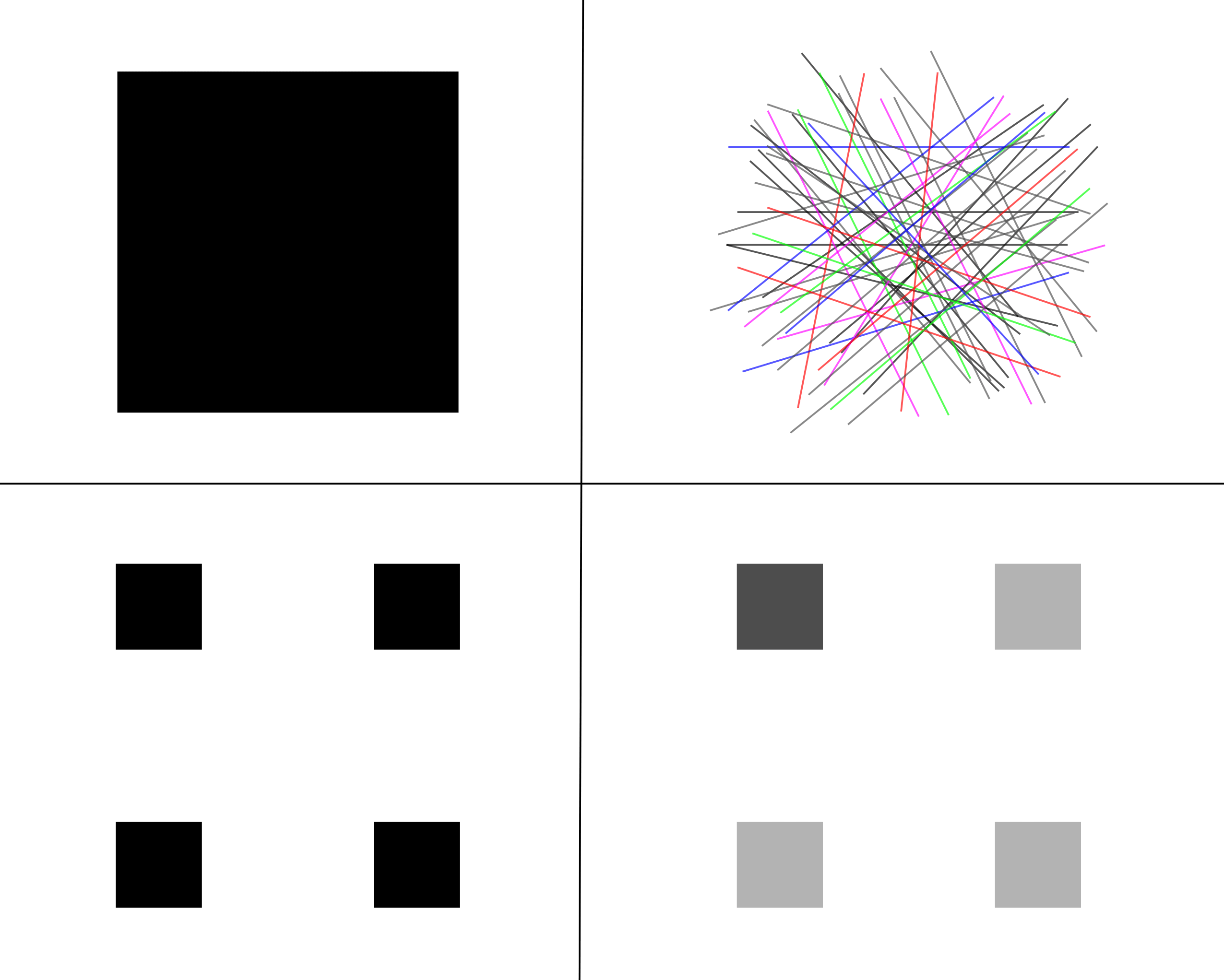}\end{center}\caption{The 2-dimensional Lebesgue measure $\mathcal{L}^2$ in the plane $\RR^2$ is purely 1-unrectifiable (top left). A countable sum $\mu=\sum_{i}^\infty \mu_i$ of measures $\mu_i$ supported on line segments $I_i$ is a 1-rectifiable measure with support $\overline{\bigcup_i I_i}$  (top right); in particular, there are examples of this type with $\spt\mu=\RR^2$.
The natural Hausdorff measure $\Haus^s\res E_s$ restricted to a Cantor set $E_s$ of Hausdorff dimension $s$ is 1-rectifiable when $s<1$ and purely 1-unrectifiable when $s\geq 1$ (bottom left). A self-similar measure $\mu$ supported on the set $E_1$ is 1-rectifiable when the generating sets for $E_1$ have unbalanced weights (bottom right). This illustrates that it is possible for a rectifiable measure to have purely unrectifiable support.}\label{fig:basic-examples}\end{figure}

\begin{theorem} \label{t-main} For every Carnot group $G$ and every locally finite Borel measure $\mu$ on $G$, there exist (explicitly defined) Borel functions $\underline{D}^1(\mu,\cdot)$ and $J^*(\mu,\cdot)$ from $G$ to $[0,\infty]$ such that the 1-rectifiable and purely 1-unrectifiable parts of a given locally finite measure $\mu$ are identified by the pointwise behavior of the functions:
\begin{align}
 \label{rect-part}   \mu_\rect &= \mu \res \left\{ x \in G \, : \, \underline{D}^1(\mu,x)>0 \text{ and } J^*(\mu,x) < \infty \right\},\\
\label{pu-part} \mu_\pu &= \mu \res \left\{ x \in G \, : \, \underline{D}^1(\mu,x) = 0 \text{ or } J^*(\mu,x) = \infty \right\}.
\end{align}
\end{theorem}

The following consequence is immediate.

\begin{corollary}\label{c-main} A locally finite Borel measure $\mu$ on $G$ is 1-rectifiable if and only if $\lD1(\mu,x)>0$ and $J^*(\mu,x)<\infty$ at $\mu$-a.e.~$x\in G$. \end{corollary}

The ``identifying functions'' $\lD1(\mu,\cdot)$ and $J^*(\mu,\cdot)$ play distinct roles in the main theorem. Roughly speaking, the first function $\lD1(\mu,x)$ detects metric dimension, while the second $J^*(\mu,x)$ detects Carnot geometry. Let us now describe them in more detail.

For every locally finite Borel measure $\mu$ on an arbitrary metric space $X$, the \emph{lower 1-density} $\lD1(\mu,\cdot):X\rightarrow[0,\infty]$ is defined by the rule
\begin{equation}
\label{d:lowerdensity}
\lD1(\mu,x)=\liminf_{r\downarrow 0} \frac{\mu(B(x,r))}{2r}\quad\text{for all }x\in X,\end{equation} where $B(x,r)$ is the closed ball with center $x\in X$ and radius $r>0$. The fact that in any metric space the lower 1-density is positive on the 1-rectifiable part of a locally finite measure follows from Cutler's theorem relating the lower density and packing measures (see Theorem \ref{t:cutler}). More specifically, the pointwise behavior of the lower 1-density can be used to identify the unique parts of a locally finite measure that are carried by or singular to Borel sets of finite 1-dimensional packing measure $\mathcal{P}^1$. Thus, since every rectifiable curve in a metric space has finite $\mathcal{P}^1$ measure, we obtain \begin{equation}\mu_\rect\leq \mu\res\{x\in X:\lD1(\mu,x)>0\}\quad\text{and}\quad \mu\res\{x\in X:\lD1(\mu,x)=0\}\leq \mu_\pu\end{equation} for any locally finite measure $\mu$ on $X$. For an in depth discussion, see \S\S\ref{ss:ident}--\ref{ss:rect-curves}, especially Corollary \ref{c-packing} and Remark \ref{r:lower-density}.

The \emph{density-normalized Jones function} $J^*(\mu,\cdot):G\rightarrow[0,\infty]$ connects the local geometry of a locally finite Borel measure $\mu$ on a Carnot group $G$ with the asymptotic geometry of rectifiable curves in $G$. When $G$ has step $s$, the function is defined by the rule \begin{equation}\label{Jstar} J^*(\mu,x):=\sum_{\substack{Q\in\Delta\\ \side Q\leq 1}} \beta^*(\mu,Q)^{2s}\diam Q\,\frac{\chi_Q(x)}{\mu(Q)}\quad\text{for all }x\in G,\end{equation} where $\Delta$ is a fixed system of ``dyadic cubes'' for $G$ (see \S\ref{ss:dyadic}) and $\beta^*(\mu,Q)$ is a certain \emph{anisotropic} measurement of the deviation of $\mu$ in a neighborhood of $Q$ from being a measure supported on a horizontal line in $G$. The definition of $\beta^*(\mu,Q)$ is based on the \emph{stratified $\beta$ numbers} of \cite{Li-TSP}. Roughly speaking, $J^*(\mu,x)$ is finite at some $x$ in the support of $\mu$ whenever the local dimension of $\mu$ at $x$ is less than 1 or $\mu$ has a measure-theoretic weak tangent at $x$. For a discussion of the underlying etymology and history of similar Jones-type geometric square functions in $\RR^n$, see \cite{BS1,BS-proc,badger-schul}. We postpone the precise definition of $\beta^*(\mu,Q)$ to \S\ref{ss:newbetas}. For now, let us simply remark that horizontal lines are the tangents to rectifiable curves in Carnot groups (see \cite[Theorem 2]{pansu}) and the definition of $\beta^*(\mu,Q)$ involves the step of the group. For example, when $G=\RR^n$ is a Euclidean space, the step $s=1$ and the horizontal lines are precisely the 1-dimensional affine subspaces of $\RR^n$. When $G$ is a Heisenberg group, the step $s=2$ and there is a 2-dimensional space of horizontal lines passing through each point in $G$ (see Figure \ref{fig:heisenberg}).

\begin{figure}\begin{center}\includegraphics[width=.8\textwidth]{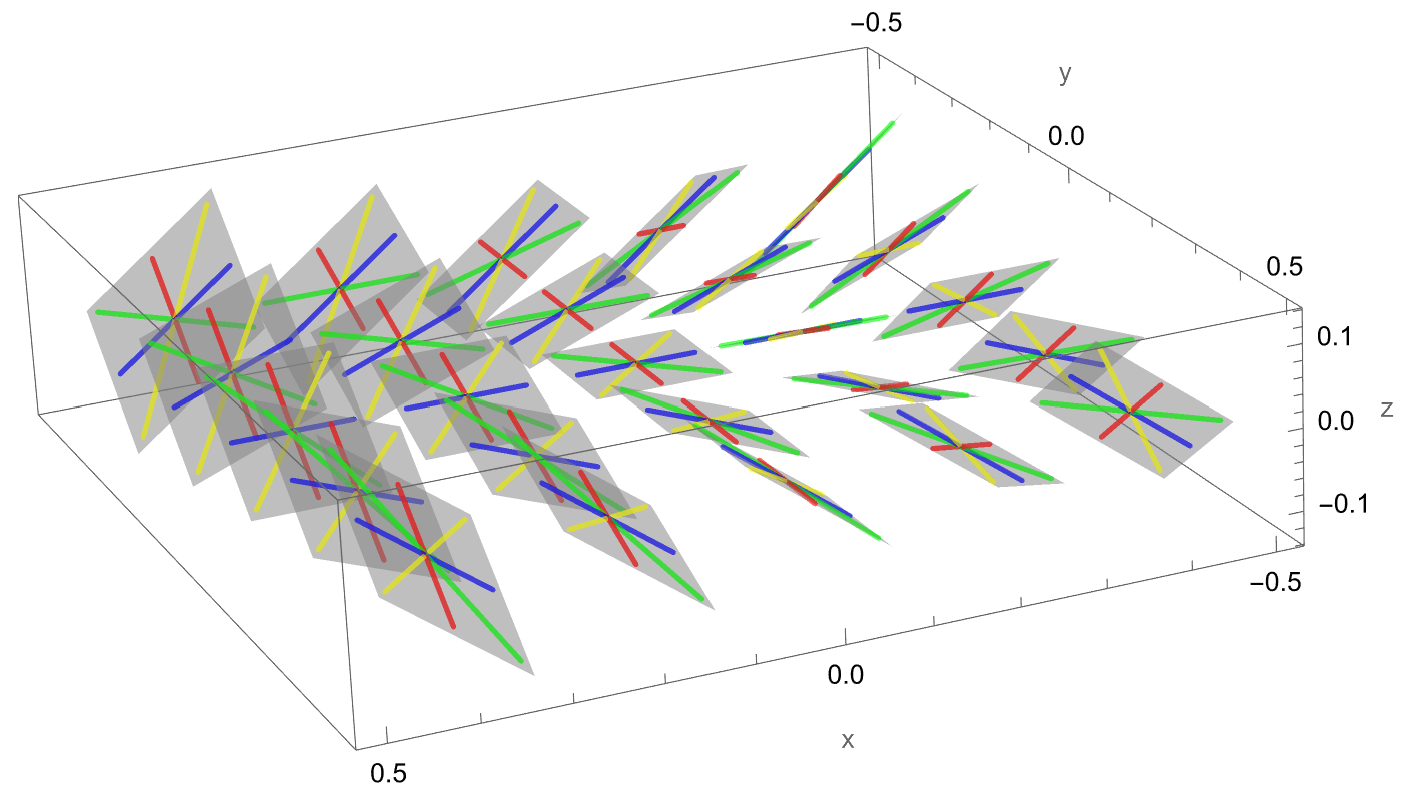}\end{center}\caption{The simplest example of a nonabelian Carnot group is the first Heisenberg group $H^1$, which is topologically equivalent to $\RR^3$ but is equipped with a metric so that $H^1$ has Hausdorff dimension 4. The step of $H^1$ is 2. In the illustration, we show 4 horizontal line segments at 25 points located in the $xy$-plane inside of $H^1$. }\label{fig:heisenberg}\end{figure}

\begin{remark}\label{r:metric} On any metric space $X$, the collections of rectifiable curves and 1-rectifiable measures on $X$ are invariant under bi-Lipschitz changes of metric on $X$. (Of course, the length of any given curve depends on the choice of metric.) In Theorem \ref{t-main}, there is partial flexibility in the choice of metric used to define the two identifying functions. The lower 1-density $\lD1(\mu,\cdot)$ may be defined with respect to any metric on $G$ that is bi-Lipschitz equivalent to a metric associated to a Hebisch-Sikora norm (e.g.~the Carnot-Carath\'eodory metric). However, the definition of the Jones function $J^*(\mu,\cdot)$ (in particular, that of $\beta^*(\mu,Q)$) is more sensitive and presently seems to require the use of metrics that are associated to Hebisch-Sikora norms on $G$ in order to use Lemma \ref{l:taylor} and Proposition \ref{p:proj-prop} in the proof of Proposition \ref{p:goal}.\end{remark}

Underpinning the main theorem is a characterization of subsets of rectifiable curves
with estimates on the length of the shortest curve containing a given set, usually called
the \emph{Analyst’s Traveling Salesman Theorem}. First established in $\RR^n$ by Jones \cite{Jones-TSP} when $n=2$ and by Okikiolu \cite{Ok-TST} when $n\geq 3$, the ATST was recently extended to arbitrary Carnot groups by the second author \cite{Li-TSP} (for earlier necessary or sufficient conditions, see \cite{CLZ-TSP, FFP, Juillet, Li-Schul1,Li-Schul2}). A key insight in \cite{Li-TSP} is that, to obtain a full characterization of subsets of rectifiable curves with effective estimates on length, the local deviation of the set from a horizontal line should incorporate distance in each layer of the Carnot group. Let us now state the theorem. Following \cite{Li-TSP}, for any nonempty set $E\subset G$ and ball $B(x,r)$, define the \emph{stratified $\beta$ number} for $E\cap B(x,r)$ by setting \begin{equation}\label{sean-beta} \beta_E(x,r)^{2s}:=\inf_L \sum_{i=1}^s \sup_{z\in E\cap B(x,r)} \left(\frac{d_i(\pi_i(z),\pi_i(L))}{r}\right)^{2i},\end{equation} where $L$ ranges over all horizontal lines in $G$, $\pi_i:G\rightarrow G_i$ is the projection of $G$ onto a layer $G_i=G/G^{(i+1)}$ of $G$, and $d_i(x,A)=\inf\{d_i(x,a):a\in A\}$ for some choice of metric $d_i$ associated to a Hebisch-Sikora norm on $G_i$ (see \S\ref{ss:Carnot}). When $G=\RR^n$, the step $s=1$, horizontal lines are 1-dimensional affine subspaces, $\pi_1$ is the identity, and the stratified $\beta$ number reduces to the usual Jones $\beta$ number.

\begin{theorem}[ATST in Carnot groups {\cite[Theorem 1.5]{Li-TSP}}] \label{t:sean} Let $G$ be a Carnot group with step $s$ and Hausdorff dimension $q$. For any set $E\subset G$, define the quantity \begin{equation}\beta(E):=\int_0^\infty \int_G \beta_E(x,r)^{2s} \diam B(x,r)\,\frac{dx}{r^q}\frac{dr}{r}.\end{equation} Then $E$ lies in a rectifiable curve if and only if $E$ is bounded and $\beta(E)<\infty$. Moreover, there exists $C>1$ depending only on $G$ and its underlying metrics $d_i$ so that: \begin{enumerate}
\item if $\Gamma$ is any curve containing $E$, then $\diam E+\beta(E)\leq C\Haus^1(\Gamma)$; and,
\item if $\diam E+\beta(E)<\infty$, then there exists a curve $\Gamma$ containing $E$ for which $\Haus^1(\Gamma)\leq C(\diam E +\beta(E))$.
\end{enumerate}
\end{theorem}

To promote Theorem \ref{t:sean} to a characterization of 1-rectifiable measures on $G$, we need to first extend the algorithm for constructing a rectifiable curve through $E$ when $\beta(E)<\infty$ to an algorithm for drawing a curve through the Hausdorff limit of a sequence $(X_k)$ of ``point clouds''.
This algorithm has its origins in \cite{Jones-TSP} when $G=\RR^n$ and \cite{FFP} when $G$ is the (first) Heisenberg group.
In the original setting of the Analyst's Traveling Salesman Theorem, one can simply take $(X_k)$ to be a \emph{nested} sequence of $2^{-k}$-nets for $E$. However, in the setting of the main theorem, when trying to build a rectifiable curve charged by $\mu$, we only know how to locate families $X_k$ of $2^{-k}$-separated points that are nearby, but not necessarily on a set with positive measure (see Lemma \ref{l-centerofmass}) and we must allow $X_k$ to ``float'' as $k\rightarrow\infty$. This issue was resolved when $G=\RR^n$ by the first author and Schul \cite{badger-schul} by introducing ``extensions'' to ``bridges'' and reproving Jones' traveling salesman algorithm from first principles. In \S\ref{s:construction}, we integrate ideas from \cite{badger-schul} and \cite{Li-TSP} to establish a flexible traveling salesman algorithm in arbitrary Carnot groups (see Proposition \ref{p:goal}). There are additional technical challenges along the way. To name just one, the numbers $\beta^*(\mu,Q)$ appearing in Theorem \ref{t-main} are designed so that we can extract enough data points lying nearby a set with positive measure to which we can apply the traveling salesman algorithm. In \cite{badger-schul}, the extraction process involves a nice idea of Lerman \cite{Lerman}: convexity of the distance of a point to a Euclidean line $L$ and Jensen's inequality control the distance of the $\mu$-center-of-mass $z_Q$ in a bounded window $Q$ to the line $L$. Unfortunately, we cannot use this observation in higher step Carnot groups. To overcome this, in \S\ref{s-suff}, we must reorder steps in the  proof from \cite[\S5]{badger-schul} and employ an indirect argument using the Chebyshev inequality.

Interesting examples of 1-rectifiable and purely 1-unrectifiable Borel measures that are singular with respect to $\Haus^1$ and have compact support can be found in \cite{upper-conical}, \cite{MO-curves}, and \cite{MM1988}. Garnett, Killip, and Schul \cite{GKS} produced a family of 1-rectifiable measures $\mu$ on $\RR^n$ that are doubling in the strong sense that \begin{equation}\label{doubling-mu} 0<\mu(B(x,2r))\leq C\mu(B(x,r))<\infty\quad\text{for all $x\in\RR^n$ and all $r>0$}.\end{equation} Not only are their measures singular with respect to the Hausdorff measure $\Haus^1$, but they also satisfy the stronger condition $\lD1(\mu,x)=\infty$ $\mu$-a.e.~(see \cite[Example 1.15]{BS1}). In arbitrary metric spaces, Azzam and Morgoglou \cite{AM-metric} characterize 1-rectifiable doubling measures with \emph{$\sigma$-compact connected supports} in terms of $\lD1(\mu,x)$ alone, but leave open the question of existence of such measures. To close the paper, we extend the Garnett-Killip-Schul existence theorem to a broad class of metric spaces, including Carnot groups and Riemannian manifolds. While the construction of the measures in \cite{GKS} leaned on the product structure of Euclidean space, we show that this is not essential for the proof.

\begin{theorem}\label{t-gks} If $X$ is a complete, doubling, locally quasiconvex metric space, then there exists a doubling measure $\nu$ on $X$ with $\spt\nu = X$ such that $\nu$ is 1-rectifiable.\end{theorem}

It is still an open problem to characterize subsets of rectifiable curves in an arbitrary Banach or metric space. See \cite{Badger-McCurdy-1,Badger-McCurdy-2, DS-metric, ENV-Banach, Hah05} for some partial results and discussion of the main difficulties. On the other hand, Schul \cite{Schul-Hilbert} successfully reformulated the Analyst's Traveling Salesman Theorem so that it holds in an arbitrary (finite or infinite-dimensional) Hilbert space with dimension-independent constants. Gaps in the proof of the theorem in \cite{Schul-Hilbert} were recently discovered, but these have now been filled-in; see  \cite{Badger-McCurdy-1}, \cite{Badger-McCurdy-2}, and \cite{Krandel-Hilbert}. Naples \cite{Naples-TST} has implemented a version of Theorem \ref{t-main} for pointwise doubling measures on infinite-dimensional Hilbert spaces. Progress on traveling salesman type theorems for various higher-dimensional objects has been made in  \cite{AS-TST, BNV, Balogh-Zust, Hyde-TST,Villa-TST}.

The rest of the paper is arranged as follows. In \S\ref{s-prelim}, we collect background results in geometric measure theory and metric geometry, including definitions of Hausdorff and packing measures, rectifiable curves, Carnot groups, and metric cubes. A version of the Analyst's Traveling Salesman Theorem for floating point clouds in a Carnot group is the topic of \S\ref{s:construction}. In \S\ref{ss:newbetas}, we define the anisotropic, stratified beta numbers $\beta^*(\mu,Q)$. In \S\ref{s-suff}, we show how positivity of the lower density $\lD1(\mu,x)$ and finiteness of the Jones function $J^*(\mu,x)$ for $x\in A$ yield rectifiability of $\mu\res A$. In \S\ref{s-necc}, we show that $J^*$ is locally integrable on any rectifiable curve, which yields necessary conditions for 1-rectifiability. The proof of Theorem \ref{t-main} is recorded in \S\ref{s-proofs}, using results from \S\S \ref{s-suff} and \ref{s-necc}. The proof of Theorem \ref{t-gks}, in \S\ref{s-GKS}, may be read independently of \S\S \ref{s:construction}--\ref{s-proofs}.

\subsection*{Acknowledgements} The authors are grateful to an anonymous referee for bringing an error in the statement and proof of \cite[Lemma 5.6]{badger-schul} to their attention (see Remark \ref{r-localize}). To correct the mistake, one should replace the set $A:=\{x\in\Top(\mathcal{T}):S_{\mathcal{T},b}(\mu,x)\leq N\}$ with the set $A:=\{x\in\leaves(\mathcal{T}):S_{\mathcal{T},b}(\mu,x)\leq N\}$. We implement this change in Lemma \ref{l-localize} and illustrate its correct use in the proof of Theorem \ref{t-sec5main}. The corrected lemma should also be applied in the proofs of \cite[Theorems 5.1 and 7.4]{badger-schul} and \cite[Theorem 2.5]{badger-naples}.

\section{Preliminaries}
\label{s-prelim}

\subsection{Implicit constants} When working on a fixed metric space $X$ (on a Carnot group $G$ in \S\S\ref{s:construction}--\ref{s-proofs} and on a complete, doubling, locally convex metric space $X$ in \S\ref{s-GKS}), we may write $a \lesssim b$ to indicate that $a \leq Cb$ for some positive and finite constant $C$ that may depend on $X$, including its metric and dimensions, but (without further qualification) is otherwise independent of a choices of particular sets or measures on $X$. We write $a\sim b$ if $a\lesssim b$ and $b\lesssim a$. We may specify the dependence of implicit constants on additional parameters $c,d,\dots$ by writing $a\lesssim_{c,d,\dots} b$ and $a\sim_{c,d,\dots} b$.

\subsection{Measures and the identification problem}\label{ss:ident} To set our conventions, we recall that a \emph{measurable space} $(X,\mathcal{M})$ is a nonempty set $X$ paired with a $\sigma$-algebra $\mathcal{M}$ on $X$, i.e.~a nonempty collection of subsets of $X$ that is closed under complements and countable unions; a \emph{measure} on $(X,\mathcal{M})$ is a function $\mu:\mathcal{M}\rightarrow [0,\infty]$ such that $\mu(\emptyset)=0$ and $\mu(\bigcup_1^\infty A_i)=\sum_1^\infty \mu(A_i)$ whenever $A_1,A_2,\dots\in\mathcal{M}$ are pairwise disjoint. In particular, a \emph{Borel measure} $\mu$ on a metric space $X$ is a measure defined on some measurable space $(X,\mathcal{M})$ such that $\mathcal{M}$ contains every Borel set in $X$. Given a measure $\mu$ on $(X,\mathcal{M})$ and a set $E\in\mathcal{M}$, the \emph{restriction} of $\mu$ to $E$ is the measure $\mu\res E$ defined by the rule $\mu\res E(A)=\mu(A\cap E)$ for all $A\in\mathcal{M}$.

Given a measure $\mu$ on $(X,\mathcal{M})$ and a non-empty family $\mathcal{N}$ of sets in $\mathcal{M}$, we say that $\mu$ is \emph{carried by} $\mathcal{N}$ if $\mu\left(X\setminus\bigcup_1^\infty N_i\right)=0\quad\text{for some sequence $N_1,N_2,\dots\in\mathcal{N}$}.$ At the other extreme, we say that $\mu$ is \emph{singular to} $\mathcal{N}$ if $\mu(N)=0$ for every $N\in\mathcal{N}$. For example, when $\mathcal{N}$ is the set of rectifiable curves in a metric space $X$, we recover the definition of 1-rectifiable and purely 1-unrectifiable measures recorded in the introduction. We have the following convenient form of the Lebesgue decomposition theorem; a detailed proof is written in the appendix of \cite{BV}.

\begin{lemma} \label{l:decomp} Let $(X,\mathcal{M})$ be a measurable space and let $\mathcal{N}$ be a nonempty collection of sets in $\mathcal{M}$. For every $\sigma$-finite measure $\mu$ on $(X,\mathcal{M})$, there is a unique decomposition $\mu=\mu_\mathcal{N}+\mu_\mathcal{N}^\perp$ as a sum of measures on $(X,\mathcal{M})$ such that $\mu_\mathcal{N}$ is carried by $\mathcal{N}$ and $\mu_\mathcal{N}^\perp$ is singular to $\mathcal{N}$.
\end{lemma}

\begin{remark} The proof of Lemma \ref{l:decomp} is abstract and does not provide any concrete method to produce sets $N_1,N_2,\dots\in\mathcal{N}$ such that $\mu_\mathcal{N}(X\setminus\bigcup_1^\infty N_i)=0$. The \emph{identification problem} (see \cite{ident}) is to find pointwise defined properties $P(\mu,x)$ and $Q(\mu,x)$ such that $$\mu_\mathcal{N} = \mu\res\{x\in X:P(\mu,x)\text{ holds}\}\quad\text{and}\quad \mu_\mathcal{N}^\perp = \mu \res\{x\in X:Q(\mu,x)\text{ holds}\}$$ for every (locally) finite measure $\mu$ on $X$. An ideal solution should involve the geometry of the space $X$ and the sets in $\mathcal{N}$.
\end{remark}

On a metric space $X$, we let $U(x,r)$ and $B(x,r)$ denote the open and closed balls with center $x\in X$ and radius $r>0$, respectively. Let $E\subset X$ and let $\delta>0$. A \emph{$\delta$-cover} of $E$ is a finite or infinite sequence of sets $E_1,E_2,\dots\subset X$ such that $E\subset\bigcup_i E_i$ and $\diam E_i\leq \delta$ for all $i$, where $\diam A$ denotes the diameter of a set $A$. A \emph{$\delta$-packing} in $E$ is a finite or infinite sequence $B_1,B_2,\dots$ of pairwise disjoint closed balls centered in $E$ such that $2\rad B_i\leq \delta$ for all $i$, where $\rad B$ denotes the radius of a ball $B$. For any $E\subset X$, $s\geq 0$, and $\delta>0$, we define $$\Haus^s_\delta(E)= \inf\left\{\sum_i (\diam E_i)^s: E_1,E_2,\dots\text{ is a $\delta$-cover of $E$}\right\},$$ $$ \Haus^s(E)=\lim_{\delta\downarrow 0} \Haus^s_\delta(E)=\sup_{\delta>0} \Haus^s_\delta(E),$$ $$P^s_\delta(E)=\sup\left\{\sum_i(2\rad B_i)^s:B_1,B_2,\dots\text{ is a $\delta$-packing in $E$}\right\},$$ $$P^s(E)=\lim_{\delta\downarrow 0} P_\delta^s(E)=\inf_{\delta>0} P_\delta^s(E),$$ $$\mathcal{P}^s(E)=\inf\left\{\sum_i P^s(E_i):E\subset\bigcup_{i=1}^\infty E_i\right\}.$$ We call $\Haus^s$ the \emph{$s$-dimensional Hausdorff measure} and call $\Pack^s$ the \emph{$s$-dimensional packing measure}; both $\Haus^s$ and $\Pack^s$ are Borel regular metric outer measures on $X$, and in particular, $\Haus^s$ and $\Pack^s$ are measures when restricted to the $\sigma$-algebra of Borel sets. The auxiliary quantity $P^s$ is called the \emph{$s$-dimensional packing premeasure}. We caution the reader that the premeasure $P^s$ is generally not an outer measure---it is monotone, but is not countably subadditive. Note that we have adopted the ``radius'' definition of the packing measure instead of the ``diameter'' definition. The next estimate (valid on any metric space!) is a special case of \cite[Theorem 3.16]{Cutler-density-theorem}.

\begin{theorem}[see {Cutler \cite{Cutler-density-theorem}}] \label{t:cutler} Let $\mu$ be a finite Borel measure on a metric space $X$, let $E\subset X$ be Borel, and let $s>0$. If $a\leq \liminf_{r\downarrow 0} (2r)^{-s} \mu(B(x,r))\leq b$ for all $x\in E$, then $$a\Pack^s(E)\leq \mu(E) \leq 2^s b\Pack^s(E),$$ where we take the left hand side to be $0$ if $a=0$ or $\Pack^s(E)=0$ and take the right hand side to be $\infty$ if $b=\infty$ or $\Pack^s(E)=\infty$.
\end{theorem}

We can now use Cutler's theorem to solve an instance of the identification problem.

\begin{corollary}\label{c-packing} Let $X$ be a metric space, let $s>0$, and let $\mathcal{N}$ be the collection of all Borel sets $E\subset X$ with $\Pack^s(E)<\infty$. For every Borel measure $\mu$ on $X$ such that $\mu$ is finite on bounded sets, the parts $\mu_\mathcal{N}$ carried by $\mathcal{N}$ and $\mu_\mathcal{N}^\perp$ singular to $\mathcal{N}$ (see Lemma \ref{l:decomp}) are identified pointwise by the positivity of the lower $s$-density: $$\mu_\mathcal{N}=\mu\res\{x\in X: \textstyle\liminf_{r\downarrow 0} (2r)^{-s}\mu(B(x,r))>0\},$$ $$\mu_\mathcal{N}^\perp=\mu\res\{x\in X:\textstyle\liminf_{r\downarrow 0} (2r)^{-s}\mu(B(x,r))=0\}.$$
\end{corollary}

\begin{proof} Fix any base point $x_0\in X$. The set $A=\{x\in X:\liminf_{r\downarrow 0} (2r)^{-s}\mu(B(x,r))>0\}$ can be written as a countable union of sets $$A_{k,l}=\{x\in B(x_0,l):\textstyle\liminf_{r\downarrow 0} (2r)^{-s}\mu(B(x,r))>1/k\},$$ where $k$ and $l$ range over all positive integers. Using Cutler's theorem, we determine that $\Pack^s(A_{k,l}) \leq k\,\mu(A_{k,l})\leq k\,\mu(B(x_0,l))<\infty$ for each $k$ and $l$. Therefore, $\mu\res A$ is carried by sets of finite $\Pack^s$ measure. Similarly, let $B=\{x\in X:\liminf_{r\downarrow 0} (2r)^{-s}\mu(B(x,r))=0\}$ and suppose $\Pack^s(E)<\infty$. Then $$\mu\res B(E) = \lim_{l\rightarrow\infty} \mu\res (B\cap U(x_0,l))(E)\leq 2^s\cdot 0\cdot \Pack^s(E)=0,$$ by continuity from below and the upper bound in Cutler's theorem with $b=0$. Thus, $\mu\res B$ is singular to sets of finite $\Pack^s$ measure. Clearly $\mu=\mu\res A + \mu \res B$. By uniqueness of the decomposition in Lemma \ref{l:decomp}, we confirm that $\mu_\mathcal{N}=\mu\res A$ and $\mu_\mathcal{N}^\perp=\mu\res B$.\end{proof}

\begin{corollary}\label{c-ac} Let $X$ be a metric space, let $s>0$, and let $\mathcal{N}$ be the collection of all Borel sets $E\subset X$ with $\Pack^s(E)=0$. For every Borel measure $\mu$ on $X$ such that $\mu$ is finite on bounded sets, the parts $\mu_\mathcal{N}$ carried by $\mathcal{N}$ and $\mu_\mathcal{N}^\perp$ singular to $\mathcal{N}$ (see Lemma \ref{l:decomp}) are identified pointwise by the magnitude of the lower $s$-density: $$\mu_\mathcal{N}=\mu\res\{x\in X: \textstyle\liminf_{r\downarrow 0} (2r)^{-s}\mu(B(x,r))=\infty\},$$ $$\mu_\mathcal{N}^\perp=\mu\res\{x\in X:\textstyle\liminf_{r\downarrow 0} (2r)^{-s}\mu(B(x,r))<\infty\}.$$ In particular, $\mu\ll\Pack^s$ if and only if $\liminf_{r\downarrow 0} (2r)^{-s}\mu(B(x,r))<\infty$ $\mu$-a.e.\end{corollary}

\begin{proof} We leave the proof that $\mu_\mathcal{N}$ and $\mu_\mathcal{N}^\perp$ are identified by the given formulas to the reader. (Just mimic the proof of Corollary \ref{c-packing}.) For the last remark, notice that $\mu\ll\Pack^s$ if and only if $\mu(E)=0$ whenever $\Pack^s(E)=0$. Thus, $\mu\ll\Pack^s$ if and only if $\mu$ is singular to sets of zero $\Pack^s$ measure.\end{proof}

\begin{remark}\label{r:packing-to-hausdorff} Analogous results hold with the \emph{Hausdorff measures} replacing the packing measures and \emph{upper densities} defined using $\limsup$ replacing lower densities defined using $\liminf$. The proof of Theorem \ref{t:cutler} for Hausdorff measures is considerably easier and can be proved using Vitali's $5r$-covering lemma (see \cite{Mattila} or \cite{Juha-book}) and the definition of $\Haus^s$.\end{remark}

\subsection{Rectifiable curves}\label{ss:rect-curves} The length of a curve in a metric space can be defined either \emph{intrinsically} in terms of the variation of a parameterization of the curve or \emph{extrinsically} using the 1-dimensional Hausdorff measure of the trace of the curve. It is well known that a curve has finite extrinsic length if and only if it admits a parameterization with finite intrinsic length; for a detailed explanation, see \cite{AO-curves}. The following theorem originated in the 1920s (see \cite{AO-curves} for a reference).

\begin{theorem}[Wa\.zewski's Theorem] \label{t:wazewski} Let $X$ be a metric space. For any nonempty set $\Gamma\subset X$, the following are equivalent: \begin{enumerate}
\item $\Gamma$ is compact and connected, and $\Haus^1(\Gamma)<\infty$;
\item $\Gamma=f([0,1])$ for some continuous map $f:[0,1]\rightarrow X$ such that $\var(f)=\sup_{t_0<t_1<\cdots<t_n} \sum_1^n \dist(f(t_{i-1}),f(t_i))<\infty$;
\item $\Gamma=f([0,1])$ for some Lipschitz continuous map $f:[0,1]\rightarrow X$.
\end{enumerate} Moreover, any set $\Gamma$ satisfying (1), (2), or (3) is the image of a Lipschitz continuous map $f:[0,1]\rightarrow X$ with $|f(t)-f(s)|\leq L|t-s|$ for all $s,t\in[0,1]$, where $f$ is essentially 2-to-1 and $L=\var(f)=2\Haus^1(\Gamma)$.\end{theorem}

A \emph{rectifiable curve} $\Gamma$ in a metric space $X$ is any nonempty set satisfying one of the three conditions in Wa\.zewski's theorem. To test whether a given set $\Gamma$ is a rectifiable curve it is usually easiest to check (1).
In fact, according to the following lemma, a weaker assumption suffices in complete metric spaces.

A set $Y\subset X$ is said to be \emph{$r$-separated} if $\dist(y,z)\geq r$ for all $y,z\in Y$. If, in addition, $\dist(x,Y)<r$ for all $x\in X$, then we call $Y$ an \emph{$r$-net} for $X$. Recall also that $B\subset X$ is \emph{totally bounded} if for every $r>0$, the set $B$ can be covered by a finite number of balls of radius $r$. It is well-known that a metric space $X$ is compact if and only if $X$ is complete and totally bounded.

\begin{lemma}
\label{l-waz-2}
  Let $X$ be a complete metric space. If a nonempty set $\Gamma \subset X$ is closed, connected, and $\Haus^1(\Gamma)<\infty$, then $\Gamma$ is compact, and thus,  $\Gamma$ is a rectifiable curve.
\end{lemma}
\begin{proof} Equipped with the subspace topology, $\Gamma$ is complete since it is a closed subset of a complete metric space.
Suppose that $\Gamma$ is not compact. Then it cannot be totally bounded. Hence there exists an infinite $r$-net $Y \subset \Gamma$ for some $r\in(0,\diam \Gamma)$. By the triangle inequality, the collection $\mathcal{B} := \left\{B\left(y,r/3\right) \right\}_{y \in Y}$ is pairwise disjoint. Because $\Gamma$ is connected, $\Haus^1(\Gamma\cap B)\geq r/3$ for all $B\in\mathcal{B}$. Therefore,
\[
\mathcal{H}^1(\Gamma) \geq \sum_{B \in \mathcal{B}} \mathcal{H}^1(\Gamma \cap B) \geq \sum_{B \in \mathcal{B}} r/3.
\]
Since the collection $\mathcal{B}$ is infinite, this implies that $\mathcal{H}^1(\Gamma) = \infty$, which is a contradiction.
Therefore, $\Gamma$ must be compact, and by Theorem~\ref{t:wazewski}, $\Gamma$ is a rectifiable curve.
\end{proof}

\begin{remark}\label{r:lower-density} Since every rectifiable curve $\Gamma$ admits a Lipschitz parameterization, it follows that $\Pack^1(\Gamma)\lesssim_L \Pack^1([0,1])<\infty$ (e.g.~see \cite[Lemma 2.8]{BS1}). Hence every 1-rectifiable measure $\mu$ on $X$ is carried by sets of finite $\mathcal{P}^1$ measure. Thus, if $\mu$ is a Borel measure on $X$ that is finite on bounded sets, then the 1-rectifiable part of $\mu$ (cf.~Theorem \ref{t-main}) satisfies \begin{equation}\mu_\rect \leq \mu\res\{x\in X: \textstyle\liminf_{r\downarrow 0} (2r)^{-1}\mu(B(x,r))>0\}\end{equation}  by Corollary \ref{c-packing}. In particular, if $\mu$ is a 1-rectifiable measure on a metric space and $\mu$ is finite on bounded sets, then the lower 1-density $\lD1(\mu,x)=\liminf_{r\downarrow 0} (2r)^{-1}\mu(B(x,r))>0$ at $\mu$-a.e.~$x\in X$. This observation significantly generalizes \cite[Theorem 7.9]{Mattila}, which says that $\lD1(\Haus^1\res \Gamma,x)>0$ at $\Haus^1$-a.e.~$x\in \Gamma$ for any rectifiable curve $\Gamma$ in $\RR^n$. \end{remark}

\subsection{Carnot groups}\label{ss:Carnot}

A connected, simply connected Lie group $G$ is called a \emph{step $s$ Carnot group}
if its associated Lie algebra $\mathfrak{g}$ satisfies
$$
\mathfrak{g} = V_1 \oplus \cdots \oplus V_s,
\quad [V_1,V_i] = V_{i+1}  \text{ for } i=1,\dots,s-1,
\quad [V_1,V_s] =\{0\},
$$
where $V_1,\dots,V_s$ are non-zero subspaces of $\mathfrak{g}$.
We call this a {\em stratification} of the Lie algebra $\mathfrak{g}$.
Choose a basis $\{X_1,\dots,X_N\}$ of $\mathfrak{g}$
so that
$$
\left\{X_{\sum_{j=1}^{i-1} (\dim V_j) +1},\dots,X_{\sum_{j=1}^{i} (\dim V_j)} \right\}
\text{ is a basis of } V_i
\text{ for each } i \in \{1,\dots,s\}.
$$
For any $x \in G$,
we may use
the exponential map $\text{exp}:\mathfrak{g} \to G$
to uniquely write $x = \text{exp}(x_1X_1 + \cdots + x_N X_N)$ for some $(x_1,\dots,x_N) \in \mathbb{R}^N$.
In other words, we can identify $G$ with $\mathbb{R}^N$ via the relationship
$x \leftrightarrow (x_1,\dots,x_N)$.
These are called the {\it exponential coordinates} of $G$.
We will actually group coordinates by the layer that the corresponding basis elements are in.  Thus, we will actually write
$$
x = (x_1,\dots ,x_s),
$$
where $x_i \in \R^{n_i}$ and $n_i = \dim V_i$.
Under this identification, we have $p^{-1} = -p$ for any $p \in G$.
Denote by $|\cdot|$ the Euclidean norm in $G = \mathbb{R}^N$ relative to the above choice of basis.

For each $r \in \{2,\dots,s\}$, we also define the normal subgroups
$$G^{(r)} = \exp(V_r \oplus \cdots \oplus V_s).$$
In terms of exponential coordinates, these are the subspaces of $\R^N$ spanned by the coordinates corresponding to vectors $X_i \in V_r \oplus \cdots \oplus V_s$. For a general discussion of Carnot groups, see \cite{Italians}.

We can express group multiplication in $G$ on the level of the Lie algebra using the Baker-Campbell-Hausdorff (BCH) formula:
\begin{align}\label{e-BCH}
  \log(\exp(X)\exp(Y)) = \sum_{k > 0} \frac{(-1)^{k-1}}{k} \underset{\substack{r_i+s_i > 0,\\r_i,s_i \geq 0,\\1 \leq i \leq k}}{\sum} a(r_1,s_1,\dots,r_k,s_k) [X^{r_1} Y^{s_1} \cdots X^{r_k} Y^{s_k}]
\end{align}
Here the bracket term denotes iterated Lie brackets
\begin{align*}
  [ X^{r_1} Y^{s_1} \dotsm X^{r_n} Y^{s_n} ] = [\underbrace{X,[X,\dotsm[X}_{r_1} ,[ \underbrace{Y,[Y,\dotsm[Y}_{s_1} ,\,\dotsm\, [ \underbrace{X,[X,\dotsm[X}_{r_n} ,[ \underbrace{Y,[Y,\dotsm Y}_{s_n}] \cdots ].
\end{align*}
We have explicit formulas for group multiplication in terms of exponential coordinates:
\begin{align*}
  (x_1,\dots,x_s) \cdot (y_1,\dots,y_s) = (x_1 + y_1, x_2 + y_2 + P_2, \dots, x_s + y_s + P_s).
\end{align*}
Here each $P_i$ is a polynomial of $(x_1,\dots,x_{i-1})$ and $(y_1,\dots,y_{i-1})$, where $x_i$ and $y_i$ are vectors in $\R^{n_i}$.  We call the $P_i$'s the \emph{BCH polynomials}. We use the following lemma in \S\ref{s:construction}.

\begin{lemma}[{\cite[Lemma 4.1]{Li-TSP}}] \label{l:BCH-bound}
  There exists some constant $C > 0$ depending only on $G$ so that if $|y_i| \leq \eta$ and $|x_i| \leq 1$ for all $i \in \{1,\dots,k-1\}$ and any $\eta \in (0,1)$, then
  \begin{align*}
    |P_k(x_1,\dots,x_{k-1},y_1,\dots,y_{k-1})| \leq C\eta.
  \end{align*}
\end{lemma}

There is a natural family of automorphisms known as \emph{dilations} on $G$ indexed by $t > 0$.
Given $t>0$, we define
$$
\delta_t(x) = \delta_t(x_1,\dots,x_s)  = \left( t x_1,t^2x_2, \dots, t^sx_s \right).
$$
It follows that $\{\delta_t\}_{t>0}$ is a one parameter family,
\emph{i.e.}~$\delta_u \circ \delta_t = \delta_{ut}$.

A \emph{homogeneous norm} $N : G \to [0,\infty)$ is a function satisfying the following properties:
\begin{enumerate}[label=(\arabic*)]
  \item $N(g) = 0 ~\Leftrightarrow ~g = 0$,
  \item $N(g^{-1}) = N(g)$,
  \item $N(gh) \leq N(g) + N(h)$.
  \item $N(\delta_t(g)) = tN(g)$ for all $t > 0, g \in G$.
\end{enumerate}
The first three properties ensure that if we define $d(g,h) = N(g^{-1}h)$, then $d$ is a left-invariant metric on $G$.  The last property ensures that the metric scales with dilations, \emph{i.e.}~for all $t > 0$ and $g,h \in G$ we have
\begin{align*}
  d(\delta_t(g),\delta_t(h)) = td(g,h).
\end{align*}
Thus, we see that dilations and homogeneous norms on Carnot groups behave like scalar multiplication and linear norms. That is to say, Carnot groups may be viewed as nonabelian generalizations of vector spaces.  In fact, the class of abelian Carnot groups are precisely the Euclidean spaces. Finally, we mention that it is well known that any two metrics on a Carnot group $G$ induced by homogeneous norms are bi-Lipschitz equivalent.

We now define a family of homogeneous norms that exist for all Carnot groups.  Given a parameter $\eta > 0$, consider $B_{\R^N}(\eta)$, the Euclidean ball around 0 in $G$ with respect to the Euclidean norm $|\cdot|$.  We then define an associated Minkowski gauge on $G$ by
\begin{align*}
  N_\eta(g) = \inf\{ r > 0 : g \in \delta_r(B_{\R^N}(\eta)) \}.
\end{align*}
It is a theorem of Hebisch and Sikora \cite{Hebisch-Sikora} that, for any Carnot group $G$, there exists $\eta_0 > 0$ such that $N_\eta$ is a homogeneous norm for all $0<\eta<\eta_0$.  As Euclidean balls of different radii are not homothetic under the dilations of $G$, we obtain a family of non-isometric norms $\{N_\eta\}_{0<\eta < \eta_0}$.  We call these the \emph{Hebisch-Sikora norms} on $G$.

Define $\pi : G \to \R^{n_1}$
to be the projection of $G$ onto its first layer.
Further, for each $r=1,\dots, s-1$, we let $\pi_r : G \to G_r := G/G^{(r+1)}$.
We endow $G$ with a metric $d$ that arises from a Hebisch-Sikora norm $N$ chosen so that the projected unit ball of $N$ in each $G_r$ also forms the unit ball of a Hebisch-Sikora norm. In particular, this choice ensures that each projection $\pi_r$ is 1-Lipschitz. We note that the norms may be considered ``nested'' in the following sense:
if $N$ and $N'$ are norms of $G_r$ and $G_{r+1}$, then
\begin{align*}
  N(x_1,\dots,x_r) = N'(x_1,\dots,x_r,0)
\end{align*}
by the convexity of balls centered at 0. By abusing notation, we will use $N$ to denote all of these norms. We now record another lemma
which will be important in \S\ref{s:construction}.
\begin{lemma}[{\cite[Lemma 6.6]{Li-TSP}}] \label{l:taylor}
  For any $\alpha \in (0,1)$,
  there exists a constant $C > 0$ so that if $N(x_1,\dots,x_{s-1},0) \in [\alpha,1]$ and $|y| \leq 1/C$, then
  \begin{align*}
    0\leq N(x_1,\dots,x_{s-1},y)-N(x_1,\dots,x_{s-1},0)\leq C |y|^2.
  \end{align*}
\end{lemma}

Finally, a set $L \subset G$ is called a \emph{horizontal line}
if it is a coset of a 1-dimensional subspace in $\R^{n_1} \times \{0\} \subset G$.
In other words,
$$
L = x \cdot \{ (sv,0,\dots,0) \, : \, s \in \R \}
\quad
\text{for some } x \in G, \, v \in \R^{n_1}.
$$
By the definition of the norm on $G$, horizontal lines are isometric copies of $\R$ in $G$.

Using the BCH formulas, one can show that the Jacobian of left translation on $G$ is 1. This tells us that the Lebesgue measure on the underlying manifold $\R^N$ of $G$ is a Haar measure.  By looking at the anisotropic scaling of the dilation $\delta_\lambda$, we see that the Lebesgue measure of balls in $G$ satisfy
\begin{equation} \label{Lebesgue-q}
  |B(x,r)| = c_G r^q\quad\text{for all $x\in G$ and $r>0$},
\end{equation}
where $c_G=|B(0,1)|$ is the Lebesgue measure of the unit ball and $q = \sum_{k=1}^s k \dim V_k$ is the \emph{homogeneous dimension} of $G$.  Therefore, the Lebesgue measure on any Carnot group $G$ is $q$-uniform, Ahlfors $q$-regular, and doubling. Furthermore, it follows from a standard packing argument that any ball in $G$ of radius $r$ may be covered by at most $C(q,\varepsilon)$ balls of radius $\varepsilon r$.

\subsection{Dyadic cubes in ``finite-dimensional'' metric spaces}\label{ss:dyadic} We shall need access to a certain decomposition of an arbitrary Carnot group into a system of ``dyadic cubes'', where cubes of the same ``side length'' are pairwise disjoint. In the harmonic analysis literature, such systems are often called \emph{Christ} or \emph{Christ-David cubes} after constructions by David \cite{David-cubes} and Christ \cite{Christ-cubes} (see e.g.~\cite{what-is-a-cube}), but similar decompositions in a metric space were given earlier by Larman \cite{Larman-dimension}.  Here we quote (a special case of) a recent streamlined construction of cubes by K\"aenm\"aki, Rajala, and Suomala \cite{KRS-cubes}, which can be carried out in any metric space which is ``finite-dimensional'' in the weak sense that every ball $B$ is totally bounded. For simplicity, we record the KRS construction with the scaling parameter $1/2$; see \cite{KRS-cubes} for the general case, which allows for \emph{any} scaling parameter between $0$ and $1$.

Recall that $U(x,r)$ and $B(x,r)$ denote open and closed balls in $X$, respectively.

\begin{theorem}[{\cite[Theorem 2.1, Remark 2.2]{KRS-cubes}}]
\label{t-KRS} Let $X$ be any metric space with totally bounded balls. Suppose that we are given $x_0\in X$ and a family $(X_k)_{k\in\ZZ}$ of $2^{-k}$-nets for $X$ such that $x_0\in X_k\subset X_{k+1}$ for all $k\in\ZZ$. Then there exist a family of collections $\Delta_k = \{ Q_{k,i} \, : \, i \in N_k \subset \mathbb{N} \}_{k\in\ZZ}$
of Borel sets (``cubes'') with the following properties:
\begin{enumerate}
    \item partitioning: $X = \bigcup_i Q_{k,i}$ for every $k \in \mathbb{Z}$,
    \item nesting: $Q_{k,i}\cap Q_{m,j}=\emptyset$ or $Q_{k,i}\subset Q_{m,j}$ whenever $k\geq m$,
    \item centers and roundness: for every $Q_{k,i}$, there is a point $x_{k,i} \in X_k$ such that
    $$
    U(x_{k,i},\tfrac 16 \cdot 2^{-k}) \subset Q_{k,i} \subset B(x_{k,i},\tfrac83 \cdot 2^{-k}),
    $$
    \item inheritance: $\{ x_{k,i} \, : \, i \in N_k \} \subset \{ x_{k+1,i} \, : \, i \in N_{k+1} \}$ for all $k\in \mathbb{Z}$.
    \item origin: for every $k \in \mathbb{Z}$, there exists $Q_{k,i}$ such that
    $$
    U(x_0, \tfrac16 \cdot 2^{-k}) \subset Q_{k,i}.
    $$
\end{enumerate}
\end{theorem}

(To derive Theorem \ref{t-KRS} as stated, invoke the theorem in \cite{KRS-cubes} with $r=1/4$ and duplicate every generation of 4-adic cubes. The resulting cubes are the dyadic cubes.)

Given a fixed system of KRS cubes $(\Delta_k)_{k\in\ZZ}$ and $Q=Q_{k,i}\in \Delta_k$, we let $x_Q=x_{k,i}$ denote its \emph{center} and let $\side Q=2^{-k}$ denote its \emph{side length}. Furthermore, we define $$\lambda U_Q=U(x_Q, \tfrac16\lambda \cdot2^{-k})\quad\text{and}\quad \lambda B_Q=B(x_Q, \tfrac{8}{3}\lambda\cdot2^{-k})$$ for all $\lambda>0$. Given $Q\in\Delta_k$ and $R\in\Delta_{k+1}$, we say that $R$ is a \emph{child} of $Q$ if $R\subset Q$. Let $\Child(Q)$ denote the set of all children of $Q$. Extending this metaphor, we may define \emph{grandchildren}, \emph{descendants}, \emph{parents}, \emph{grandparents}, \emph{ancestors}, and \emph{siblings} in the natural way as convenient. Finally, we assign $\Delta=\bigcup_{k\in\ZZ}\Delta_k$; i.e.~$\Delta$ is the set of all cubes.

\begin{definition}\label{trees-and-leaves} We say that $\mathcal{T}\subset\Delta$ is a \emph{tree of cubes} if $\mathcal{T}$ has a unique maximal element $\Top(\mathcal{T})$ such that if $Q\in\mathcal{T}$, then $P\in\mathcal{T}$ for all $P\in\Delta$ with $Q\subset P\subset \Top(\mathcal{T})$. For each level $l\geq 0$, let $\mathcal{T}_l$ denote the collection of all cubes $Q\in\mathcal{T}$ with $\side Q=2^{-l}\side \Top(\mathcal{T})$. An \emph{infinite branch} of $\mathcal{T}$ is a chain $\Top(\mathcal{T})\equiv Q_0\supset Q_1\supset Q_2\supset\cdots$ with $Q_l\in\mathcal{T}_l$ for all $l\geq 0$. We define the \emph{set of leaves} of $\mathcal{T}$ by $$\leaves(\mathcal{T}):= \bigcup\left\{\bigcap_{l=0}^\infty Q_l:Q_0\supset Q_1\supset Q_2\supset\cdots\text{ is an infinite branch of $\mathcal{T}$}\right\}.$$\end{definition}

\begin{remark} Because $X$ has totally bounded balls, $\#\mathcal{T}_l<\infty$ for all $l\geq 0$. Using K\"onig's lemma (i.e.~in a graph with infinitely many vertices, each of which has finite degree, there exists an infinite path), it can thus be shown that $\leaves(\mathcal{T})=\bigcap_{l=0}^\infty\bigcup\mathcal{T}_l$. In particular, $\leaves(\mathcal{T})$ is a Borel set, since cubes in $\Delta$ are Borel.\end{remark}

\begin{definition}[cf.~{\cite[p.~18]{badger-schul}}] \label{S-functions} For any locally finite Borel measure $\mu$ on $X$, tree of cubes $\mathcal{T}$, and function $b:\mathcal{T}\rightarrow [0,\infty)$, we define the \emph{$\mu$-normalized sum function} $$S_{\mathcal{T},b}(\mu,x):=\sum_{Q\in\mathcal{T}} b(Q)\, \frac{\chi_Q(x)}{\mu(Q)}\in[0,\infty]\quad\text{for all }x\in X,$$ where we interpret $0/0=0$ and $1/0=\infty$.\end{definition}

The following lemma is a slight variation on the Hardy-Littlewood maximal theorem for dyadic cubes in $\RR^n$. The proof works in the metric setting, because the system of cubes $\Delta$ satisfies properties (1) and (2) in Theorem \ref{t-KRS}.

\begin{lemma}[localization, {cf.~\cite[Lemma 5.6]{badger-schul}}]
\label{l-localize}
Let $\mu$ be a locally finite Borel measure on $X$, let $\mathcal{T}$ be a tree of cubes, and let $b:\mathcal{T}\rightarrow[0,\infty)$. Fix $0<N<\infty$ and
define
\begin{equation} \label{A-def}
A := \left\{ x \in \leaves(\mathcal{T}) : S_{\mathcal{T},b}(\mu,x)\leq N\right\}.
\end{equation} If $\mu(A)>0$ and $0<\varepsilon<1$, then
there is a set $\mathcal{G} \subset \mathcal{T}$
such that
\begin{enumerate}
    \item $\mathcal{G}$ is a tree of cubes with $\Top(\mathcal{G}) = \Top(\mathcal{T})$,
    \item $\mu(A \cap \leaves(\mathcal{G})) \geq (1-\varepsilon)\mu(A)$, and
    \item $\sum_{Q \in \mathcal{G}} b(Q) < (N/\varepsilon)\mu(\Top(\mathcal{T}))$.
\end{enumerate}
\end{lemma}

\begin{proof} Suppose that $\mu$, $\mathcal{T}$, $b$, $N$, $A$, and $\varepsilon$ are given as in the statement of the lemma. Note that $\mu(\Top(\mathcal{T}))>0$, because $\mu(A)>0$. Declare a cube $Q\in\mathcal{T}$ to be \emph{bad} if there exists $R\in\mathcal{T}$ such that $Q\subset R$ and \begin{equation}\label{bad-def} \mu(A\cap R)\leq \frac{\varepsilon\mu(A)}{\mu(\Top(\mathcal{T}))}\mu(R).\end{equation} By design, this definition ensures that every child of a bad cube in $\mathcal{T}$ is bad too.

We say that a cube $Q\in\mathcal{T}$ is \emph{good} if $Q$ is not bad. Note that if $R\in\mathcal{T}$ and $\Top(\mathcal{T})\subset R$, then $R=\Top(\mathcal{T})$ and \begin{equation}\mu(A\cap\Top(\mathcal{T})) = \mu(A) > \frac{\varepsilon\mu(A)}{\mu(\Top(\mathcal{T}))}\mu(\Top(\mathcal{T})),\end{equation} since $\varepsilon<1$ and $\mu(A)>0$. Hence $\Top(\mathcal{T})$ is a good cube. Let $\mathcal{G}$ denote the set of all good cubes. Because $\Top(\mathcal{T})$ is in $\mathcal{G}$ and every parent of a good cube in $\mathcal{T}$ is again a good cube, we conclude that $\mathcal{G}$ is a subtree of $\mathcal{T}$ with $\Top(\mathcal{G})=\Top(\mathcal{T})$. This verifies (1).

Next, we check (2). There are two cases. First, if there are no bad cubes, then $\mathcal{G}=\mathcal{T}$ and we trivially have $\mu(A\cap\leaves(\mathcal{G}))=\mu(A)> (1-\varepsilon)\mu(A)$. Otherwise, there is at least one bad cube. Let $\mathcal{B}$ denote the set of all maximal bad cubes, i.e.~ the set of all bad cues that are not properly contained in another bad cube. Note that $\mathcal{B}$ is pairwise disjoint. Thus, using \eqref{bad-def}, we see that \begin{equation}\begin{split}\label{total-bad-mass} \mu(A\setminus(A\cap\leaves(\mathcal{G}))) \leq \sum_{R\in\mathcal{B}} \mu(A\cap R)
&\leq \frac{\varepsilon\mu(A)}{\mu(\Top(\mathcal{T}))} \sum_{R\in\mathcal{B}} \mu(R) \\
&\leq \frac{\varepsilon\mu(A)}{\mu(\Top(\mathcal{T}))}\mu(\Top(\mathcal{T}))
=\varepsilon\mu(A).\end{split}\end{equation} Thus, $\mu(A\cap\leaves(\mathcal{G}))=\mu(A)-\mu(A\setminus(A\cap\leaves(\mathcal{G})))
\geq (1-\varepsilon)\mu(A)$.

Finally, using the definitions of $A$ and $S_{\mathcal{T},b}$, Tonelli's theorem, and the definition of good cubes, we obtain \begin{equation}\begin{split}
N\mu(A) \geq \int_A S_{\mathcal{T},b}(x)\,d\mu(x)
&=\int_A\sum_{Q\in\mathcal{T}}b(Q)\frac{\chi_Q(x)}{\mu(Q)}\,d\mu(x)\\
&=\sum_{Q\in\mathcal{T}} b(Q) \frac{\mu(A\cap Q)}{\mu(Q)}
> \sum_{Q\in\mathcal{G}} b(Q) \frac{\varepsilon\mu(A)}{\mu(\Top(\mathcal{T}))}.
\end{split}\end{equation} Rearranging yields (3).
\end{proof}

\begin{remark}\label{r-localize} As stated, \cite[Lemma 5.6]{badger-schul} is false in general. Let us describe the problem. Instead of using \eqref{A-def}, the set $A$ in \cite{badger-schul} was defined as $A:= \{x\in\Top(\mathcal{T}):S_{\mathcal{T},b}(x)\leq N\}$. It was then asserted without justification that $$A\setminus \bigcup_{\text{bad cubes }Q\in\mathcal{T}} Q= A\cap \leaves(\mathcal{G}),$$ which is not true unless $\leaves(\mathcal{T})=\Top(\mathcal{T})$.
\end{remark}

Mimicking the usual construction of Whitney cubes in $\RR^n$, we may use a system of KRS cubes to build Whitney cubes in the complement of any closed set.

\begin{lemma}\label{l-whitney} If $E\subsetneq X$ is a nonempty closed set, then there exists a family $\mathcal{W}$ of cubes in $\Delta$ with the following properties. \begin{enumerate}
\item partitioning: $X\setminus E=\bigcup_{W\in\mathcal{W}} W$ and $W_1\cap W_2\neq\emptyset$ if and only if $W_1=W_2$;
\item size and location: $\diam W \leq \dist(W,E)$ for all $W\in\mathcal{W}$,\end{enumerate} where $\dist(W,E)=\inf_{w\in W}\inf_{x\in E} d(w,x)$. Moreover, if there exists a constant $c>0$ such that $\diam U(x,r)\geq cr$ whenever $x\in X$, $r>0$, and $U(x,r)\neq X$, then \begin{enumerate}\setcounter{enumi}{2} \item $\dist(W,E)<(128/c)\diam W$ for all $W\in \mathcal{W}$.\end{enumerate}
\end{lemma}

\begin{proof}Given a nonempty closed set $E$ with nonempty complement, take $\mathcal{W}$ to be any maximal family of cubes $W\in\Delta$ such that $\dist(W,E)\geq \diam W$. The partitioning property follows from maximality and properties (1), (2), and (3) of Theorem \ref{t-KRS}. Let $W\in\mathcal{W}$. One the one hand, $\dist(W,E)\geq \diam W$ by definition of the family. On the other hand, let $V$ be the parent of $W$ in $\Delta$. Then $\dist(V,E)<\diam V$ by maximality. Thus, under the extra assumption on the diameters of open balls, \begin{equation*}\dist(W,E)\leq \dist(V,E)+\diam V < 2\diam B_V \leq (128/c)\diam U_W \leq (128/c)\diam W. \qedhere\end{equation*}\end{proof}

\begin{remark}\label{r:halving} Suppose that $X$ is a doubling metric measure space in the sense that there is a Borel measure $\mu$ on $X$ and constant $C>0$ such that \eqref{doubling-mu} holds for all $x\in X$ and $r>0$. By (2) and (3) in Theorem \ref{t-KRS}, for any $Q\in \Delta_k$ and $R\in\Child(Q)$, we have $$Q\subset B(x_R,\diam B_Q)\subset B(x_R, \tfrac{16}{3}\cdot 2^{-k})\text{ and } B(x_R, \tfrac{1}{12} \cdot 2^{-k}-\varepsilon)\subset U(x_R,\tfrac{1}{6}\cdot2^{-(k+1)})\subset U_R$$
for any $0 < \varepsilon < \tfrac{1}{12} \cdot 2^{-k}$.
Doubling of the measure at $x_R$ yields $\mu(Q)\leq C^7\mu(U_R)$ for all $R\in\Child(Q)$. Hence $$\mu(Q)=\sum_{R\in\Child(Q)}\mu(R) \geq \sum_{R\in\Child(Q)}\mu(U_R) \geq C^{-7}\mu(Q)\cdot\#\Child(Q).$$ That is, $\#\Child(Q)\leq  C^7$ for every KRS cube $Q$.\end{remark}

\section{Traveling salesman algorithm in Carnot groups} \label{s:construction}

From here through the end of \S\ref{s-proofs}, let $G$ be a fixed Carnot group that is homeomorphic to $\RR^n$ and has step $s$ and homogeneous dimension $q$. Also, choose metrics $d_i$ associated to a Hebisch-Sikora norm on $G_i=G/G^{(i+1)}$ for all $1\leq i\leq s$.

In this section, our goal is to prove the following traveling salesman type criterion for existence of a rectifiable curve passing through the Hausdorff limit of a sequence of point clouds. Crucially, the weak coherence condition $(V_\two)$ only requires that each cloud lie nearby, but not necessarily on, the rectifiable curve. We will use this flexibility in the proof of Lemma~\ref{l-makeacurve}. In the Euclidean setting, Proposition \ref{p:goal} is due to the first author and Schul \cite{badger-schul}, based in part on earlier constructions in \cite{Jones-TSP} and \cite{Lerman}. There are at least two difficulties in extending this criterion to arbitrary Carnot groups. The first challenge is in the statement of the criterion. The number $\alpha_{k,v}$ is a penalty term that bounds the \emph{stratified distance} to a horizontal line $\ell_{k,v}$ of points $x$ in the clouds $V_{k-1}$ \emph{and} $V_k$ that lie nearby the point $v$ in $V_k$; the correct dependence on the step $s$ in \eqref{e:line-containment} and \eqref{e:Gamma-bound} was only recently identified by the second author \cite{Li-TSP}. Another technical challenge for higher step groups appears in the proof. In the Euclidean case, all length estimates can be stated in terms of the total Hausdorff measure of line segments of approximating curves. However, in the general Carnot setting, we need to employ two notions: edge length of projections of abstract graphs $\Gamma_k$ connecting $V_k$ onto the horizontal layer of $G$ and Hausdorff measure of geometric realizations $\widehat{\Gamma}_k$ of the graphs in the whole space $G$ (see \S\ref{ss:start}).

\begin{proposition}[traveling salesman criterion for point clouds] \label{p:goal}
  Let $x_0 \in G$, let $C^\star \geq 1$, and let $r_0 > 0$.  Suppose that $(V_k)_{k=0}^\infty$ is a sequence of nonempty finite subsets of $B(x_0, C^\star r_0)$ such that
  \begin{itemize}
    \item[$(V_I)$]\label{V1} $d(v,v') \geq 2^{-k}r_0$ for all distinct points $v,v' \in V_k$,
    \item[$(V_\two)$]\label{V2} for all $v_k \in V_k$, there exists $v_{k+1} \in V_{k+1}$ such that $d(v_{k+1},v_k) \leq C^\star2^{-k}r_0$,
    \item[$(V_\three)$]\label{V3} for all $v_k \in V_k$, there exists $v_{k-1} \in V_{k-1}$ such that $d(v_{k-1},v_k) \leq C^\star2^{-k}r_0$.
  \end{itemize}
  Suppose also that, for all $k \geq 1$ and all $v \in V_k$, there is a horizontal line $\ell_{k,v}$ in $G$ and a number $\alpha_{k,v} \geq 0$ such that
    \begin{align}
    x \in \ell_{k,v} \cdot \delta_{2^{-k}r_0}(B_{\R^n}(\alpha_{k,v}^s))\quad\text{for all }x \in (V_{k-1} \cup V_k) \cap B(v, 65C^\star 2^{-k}r_0). \label{e:line-containment}
  \end{align}
  Finally, suppose that
   \begin{align}
    \sum_{k=1}^\infty \sum_{v \in V_k} \alpha_{k,v}^{2s} 2^{-k}r_0 < \infty. \label{e:Gamma-bound}
  \end{align}
  Then the sets $V_k$ converge in the Hausdorff metric to a compact set $V \subset B(x_0, C^\star r_0)$ and there exists a rectifiable curve $\Gamma \subset B(x_0,3C^\star r_0)$ such that $V\subset\Gamma$ and
  \begin{align}
    \cH^1(\Gamma) \lesssim_{G,C^\star} r_0 + \sum_{k=1}^\infty \sum_{v \in V_k} \alpha_{k,v}^{2s} 2^{-k} r_0. \label{e:Gamma-conclusion}
  \end{align}
\end{proposition}

\begin{remark}\label{r:sean-tubular-beta} The motivation for the requirement \eqref{e:line-containment} on $\alpha_{k,v}$ comes from \cite{Li-TSP}. Recall that the stratified $\beta$ number $\beta_E(x,r)$ is defined by \eqref{sean-beta}. By \cite[Proposition 1.6]{Li-TSP}, \begin{equation}\label{tubular-beta} \beta_E(x,r)\sim\inf_L \inf\{\varepsilon>0:E\cap B(x,r) \subset L\cdot \delta_r(B_{\RR^n}(\varepsilon^s))\},\end{equation} where $B(x,r)$ is a ball in $G$, $B_{\RR^n}(\varepsilon^s)$ is a Euclidean ball about the origin of the manifold $\mathbb{R}^n$ underlying $G$, and $\varepsilon$ represents the ``width'' of a tubular neighborhood $L\cdot \delta_r(B_{\RR^n}(\varepsilon^s))$ of the horizontal line $L$ formed using the group multiplication, the group dilation, and the step of the group. The implicit constant in \eqref{tubular-beta} depends on $n$, $s$, and the choice of the metric $d_i$ in each layers $G_i$ of $G$, but is otherwise independent of $E$, $x$, and $r$.\end{remark}

The following auxiliary result captures an essential bi-Lipschitz property of projections near points that are relatively ``flat'', i.e.~close to a horizontal line relative to their scale of separation. It replaces \cite[Lemma 8.3]{badger-schul}, which was an application of the Pythagorean theorem in $\RR^n$.

\begin{proposition}
\label{p:proj-prop}
Assume $G$ is a Carnot group of step $s$,
  and let $\pi:G \to \mathbb{R}^{n_1}$ be the projection to the first layer of $G$.  For any $\alpha > 1$, there exist positive constants $C$ and $\varepsilon_0$ depending on $G$ and $\alpha$ so that if $L \subset G$ is a horizontal line, $P:G \to \pi(L)$ is the composition of $\pi$ with the orthogonal projection in $\mathbb{R}^{n_1}$ onto $\pi(L)$,
  and $a,b \in L \cdot B_{\mathbb{R}^n}(\varepsilon^s)$
  for some $\varepsilon < \varepsilon_0$ so that $d(a,b) \in [1,\alpha]$
  then
  \begin{align*}
    \frac{d(a,b)}{1 + C\varepsilon^{2s}} \leq |P(a) - P(b)| \leq d(a,b).
  \end{align*}
\end{proposition}

\begin{proof}
  The right hand inequality restates the fact that the projections which comprise $P$ are 1-Lipschitz.  We will prove the left hand inequality.
  We may without loss of generality assume that the horizontal line $L$ contains the origin.
  In particular, this means that $L$ has the form $\{(ut,0,\dots,0) : t \in \R\}$ for some $u \in \R^{n_1}$.
  We also suppose that $a \in 0 \cdot B_{\R^n}(\varepsilon^s)$ and $u$ was chosen so that $b \in (u,0,\dots,0) \cdot B_{\R^n}(\varepsilon^s)$.  Hence
  \begin{align}
    \pi(a), \pi(b) \in \pi(L) + B_{\R^{n_1}}(\varepsilon^s). \label{e:close-euc-line}
  \end{align}
  By choosing $\varepsilon_0$ sufficiently small, we can use the triangle inequality to guarantee that $|\pi(b) - \pi(a)| \geq 1/2$, $|P(b) - P(a)| \geq 1/4$, and $|u| \leq 2\alpha$.

To continue, let us prove that there exists a constant $C_0 > 0$ so that $$a^{-1}b = (\pi(b) - \pi(a), \xi_2,\dots,\xi_s)$$ and each $\xi_i \in \R^{n_i}$ has norm $|\xi_i| \leq C_0 \varepsilon^s$.  We will actually prove the statement for $\delta_{1/2\alpha}(a^{-1}b)$ (with the first layer properly rescaled) as it will allow us to use Lemma \ref{l:BCH-bound}. Rescaling back by $\delta_{2\alpha}$ then gives the corresponding statement for $a^{-1}b$. The fact that the coordinate in the first layer of $\delta_{1/2\alpha}(a^{-1}b)$ is $\frac{1}{2\alpha} (\pi(b) - \pi(a))$ is clear by the Baker-Campbell-Hausdorff formula \eqref{e-BCH}.  By our assumptions on $a,b$, we have
  \begin{align*}
    \delta_{1/2\alpha}(a^{-1}b) = (x_1,\dots,x_s) \cdot (u',0,\dots,0) \cdot (y_1,\dots,y_s),
  \end{align*}
  where $|x_i|,|y_i| \leq \varepsilon^s/2\alpha$ and $|u'| = |u|/2\alpha \leq 1$.  Two applications of Lemma \ref{l:BCH-bound} gives the result.

Now, by Lemma \ref{l:taylor}, we have $d(a,b) = N(a^{-1}b) \leq N(\pi(b) - \pi(a), \xi_2,\dots, \xi_{s-1}) + C_1 \varepsilon^{2s}$
  for some constant $C_1>0$.
  Iterating this gives a constant $C_2 > 0$ so that
  \begin{align*}
    d(a,b) \leq N(\pi(b) - \pi(a)) + C_2 \varepsilon^{2s} = |\pi(b) - \pi(a)| + C_2 \varepsilon^{2s}.
  \end{align*}
  Recalling \eqref{e:close-euc-line}, the Pythagorean theorem gives $|\pi(b) - \pi(a)| \leq |P(a) - P(b)| + 10\varepsilon^{2s}$.  Altogether, we get a constant $C_3>0$ such that
  \begin{align*}
    d(a,b) \leq |P(a) - P(b)| + C_3 \varepsilon^{2s}.
  \end{align*}
  Since $|P(a) - P(b)| \geq 1/4$, we have proven the desired inequality.
\end{proof}

\subsection{Start of the proof of Proposition \ref{p:goal}} \label{ss:start-start}

The rest of this section is devoted to the proof of
Proposition~\ref{p:goal}. We follow the general outline of the proof in the Euclidean case (see \cite[\S8.1]{badger-schul}). We shall refer the reader to the original proof for arguments that are essentially metric and highlight the changes that are necessary for the Carnot setting. The details are rather technical. As such, the reader who is willing to assume the veracity of Proposition \ref{p:goal} is encouraged to jump to \S\ref{ss:newbetas}.

Without loss of generality, we can
rescale the metric on $G$ using a dilation
so that $r_0 = 1$.
By (the proof of) Lemma 8.2 of \cite{badger-schul}, the sets $V_k$ converge in the Hausdorff metric to a compact set $V \subset B(x_0,C^\star)$.
Note that, if $\# V_k = 1$ for all $k$, then $V$ is a singleton, and so the result trivially holds.  Assume, therefore, that there is some least $k_0 \geq 0$ so that that $\# V_k \geq 2$ for all $k \geq k_0$.

\subsection{The construction} \label{ss:algo}
We will inductively construct a sequence of \emph{abstract} graphs $\Gamma_k$ on the vertices of $\bigcup_j V_j$.   The abstract edges will simply be unordered pairs of vertices. On occasion, we may refer to connected families of edges as ``curves''. (In the Euclidean case \cite{badger-schul}, the edges in $\Gamma_k$ were realized geometrically as line segments.)

To begin, we will define the {\it extension} of a vertex.  Given $v \in V_k$, we define $E[k,v]$ in the following way.  Let $v_0 = v$.  Once $v_i \in V_{k+i}$ has been defined, choose $v_{i+1}$ to be a closest point in $V_{k+i+1}$ to $v_i$.  The extension $E[k,v]$ is then defined as $E[k,v] = \{ (v_i, v_{i+1}) \}_{i=0}^\infty$.
Given distinct vertices $v,v' \in V_k$, define the {\em bridge}
\begin{align*}
  B[k,v,v'] = E[k,v] \cup \{(v,v')\} \cup E[k,v'].
\end{align*} Bridges will be used to span large ``gaps'' between vertices in $V_k$.

\subsubsection{Initial curve $\Gamma_{k_0}$} We remark that either $k_0=0$ and $V_0\subset B(x_0,C^\star)$ by assumption, or $k_0\geq 1$ and $V_{k_0}\subset B(x,C^\star 2^{-k_0})$ by $(V_\three)$, where $V_{k_0-1}=\{x\}$.
We construct the initial graph $\Gamma_{k_0}$ by including every edge $(v',v'')$ with $v',v'' \in V_{k_0}$. That is,
\begin{equation}\label{e:Gamma0-def}
  \Gamma_{k_0} := \bigcup_{v',v''\in V_{k_0}}(v',v'').
\end{equation}
\subsubsection{Future curves $\Gamma_{k}$}

Suppose that $\Gamma_{k_0},\dots,\Gamma_{k-1}$ have been defined for some $k\geq k_0+1$. In order to define the next set $\Gamma_k$, we first describe the edge set in $\Gamma_k$ locally nearby each vertex $v\in V_k$. We will then declare $\Gamma_k$ to be the union of new parts of the curve together with the bridges from previous generations.
That is,
if $\Gamma_{k,v}$ denotes the new part of $\Gamma_k$ nearby $v$, then
\begin{equation}\label{e:Gamma-def}
  \Gamma_k:=\bigcup_{v\in V_k} \Gamma_{k,v} \cup \bigcup_{j=k_0}^{k-1} \bigcup_{B[j,w',w'']\subset \Gamma_{j}} B[j,w',w''].
\end{equation}

For each $k \geq k_0$ and $v \in V_k$, define $B_{k,v} := B(v,65C^\star 2^{-k})$. According to $(V_I)$, there is some constant $M>0$ such that $\# (V_k \cap B_{k,v}) \leq M$ for all $k \geq k_0$ and every $v \in V_k$. Let $\varepsilon>0$ be a small parameter, depending only on $G$, chosen according to various needs below. In particular, when $\varepsilon>0$ is sufficiently small, we can invoke Proposition \ref{p:proj-prop}.

Fix an arbitrary vertex $v \in V_k$. We will define $\Gamma_{k,v}$ in two cases.

\textbf{Case I:} Suppose $\alpha_{k,\hat{v}}\geq \varepsilon$ for some $\hat{v} \in V_k \cap B_{k,v}$.

To construct $\Gamma_{k,v}$, consider each pair of vertices $v',v''\in V_k\cap B_{k,v}$.
If
$|\pi(v') - \pi(v'')| < 30C^\star 2^{-k}$, include the edge $(v',v'')$ in $\Gamma_{k,v}$. Otherwise, include the bridge $B[k,v',v'']$.
In other words,
$$
\Gamma_{k,v} = \bigcup_{v',v'' \in V_k} \left(
\bigcup_{|\pi(v')-\pi(v'')|< 30C^\star2^{-k}} (v',v'')
\cup
\bigcup_{|\pi(v')-\pi(v'')| \geq 30C^\star2^{-k}} B[k,v',v'']
\right)
$$
This ends the description of $\Gamma_{k,v}$ in \textbf{Case I}.

\textbf{Case II:} Suppose $\alpha_{k,\hat{v}}<\varepsilon$ for every $\hat{v} \in V_k \cap B_{k,v}$.

Identify the projected horizontal line $\pi(\ell_{k,v})$ with $\RR$. (In particular, pick directions ``left'' and ``right''.)
Let $\pi_{k,v}:G \to \mathbb{R}$ denote the projection $P$
defined in Proposition~\ref{p:proj-prop} composed with this identification.
 By \eqref{e:line-containment}, ($V_I$), and Proposition~\ref{p:proj-prop},
 the map $\pi_{k,v}$ is bi-Lipschitz on $(V_k \cup V_{k-1}) \cap B_{k,v}$
 with
 \begin{equation}
     \label{e:flat}
     d(z',z'') \leq (1+C\varepsilon^{2s})|\pi_{k,v}(z') - \pi_{k,v}(z'')|
     \qquad
     \forall z',z'' \in (V_k \cup V_{k-1}) \cap B_{k,v}.
 \end{equation}
 In particular, both $V_k\cap B_{k,v}$ and $V_{k-1}\cap B_{k,v}$ can be arranged linearly along $\ell_{k,v}$.
 That is, if we set $v_0=v\in V_k$, we can write
 $$
 v_{-l},\dots, v_{-1},v_0,v_1,\dots,v_m
 $$
 to denote the vertices in $V_k\cap B_{k,v}$ arranged from left to right according to the relative order of $\pi_{k,v}(v_i)$ in $\RR$, where $l,m\geq 0$.
In other words,
 $$
 \pi_{k,v}(v_{-l}) < \cdots < \pi_{k,v}(v_{-1}) < \pi_{k,v}(v_{0}) < \pi_{k,v}(v_{1}) < \cdots < \pi_{k,v}(v_{m}).
 $$
We start by  describing the ``right half" $\Gamma_{k,v}^R$ of $\Gamma_{k,v}$.
Starting from $v_0$ and working to the right, include each edge $(v_i,v_{i+1})$ in $\Gamma_{k,v}^R$ until $|\pi(v_{i+1}) - \pi(v_i)| \geq 30 C^\star 2^{-k}$, $v_{i+1}\not\in B(v,30 C^\star 2^{-k})$, or $v_{i+1}$ is undefined (because $i=m$). Let $t\geq 0$ denote the number of edges that were included in $\Gamma_{k,v}^R$.

\textbf{Case II-NT:} If $t\geq 1$ (that is, at least one edge was included), then we say that the vertex \emph{$v$ is not terminal to the right}, and we are done describing $\Gamma_{k,v}^R$.

\textbf{Case II-T1 and Case II-T2:} If $t=0$ (that is, no edges were included), then we say that the vertex \emph{$v$ is terminal to the right} and continue our description of $\Gamma_{k,v}^R$, splitting into subcases depending on how $\Gamma_{k-1}$ looks near $v$. Let $w_v$ be a vertex in $V_{k-1}$ that is closest to $v$.
As mentioned above, we may enumerate the vertices in $V_{k-1}\cap B_{k,v}$ starting from $w_v$ and moving right (with respect to the identification of $\ell_{k,v}$ with $\RR$) by $$w_v=w_{v,0},w_{v,1},\dots,w_{v,s}$$
i.e.~$\pi_{k,v}(w_{v,0}) < \cdots < \pi_{k,v}(w_{v,s})$.
Let $w_{v,r}$ denote the rightmost vertex that appears in $V_{k-1}\cap B(v,C^\star 2^{-(k-1)})$. There are two alternatives:
\begin{enumerate}
  \item[\textbf{T1:}] If $r=s$ or if $|\pi(w_{v,r})-\pi(w_{v,r+1})|\geq 30 C^\star 2^{-(k-1)}$, then we set $\Gamma_{k,v}^R=\{v\}$.
  \item[\textbf{T2:}] If $|\pi(w_{v,r})-\pi(w_{v,r+1})|<30 C^\star 2^{-(k-1)}$, then $v_{1}$ exists by ($V_\two$) (and $|\pi(v)-\pi(v_1)|\geq 30 C^\star 2^{-k}$). In this case, we set $\Gamma_{k,v}^R=B[k,v,v_1]$.
\end{enumerate} The first alternative defines \textbf{Case II-T1}. The second alternative defines \textbf{Case II-T2}. This concludes the description of $\Gamma_{k,v}^R$.

We define the ``left half" $\Gamma_{k,v}^L$ of $\Gamma_{k,v}$ symmetrically. Also, define the terminology \emph{$v$ is not terminal to the left} and \emph{$v$ is terminal to the left} by analogy with the corresponding terminology to the right. Having separately defined both the ``left half" $\Gamma_{k,v}^L$ and the ``right half" $\Gamma_{k,v}^R$ of $\Gamma_{k,v}$, we now declare $$\Gamma_{k,v}:=\Gamma_{k,v}^L\cup \Gamma_{k,v}^R.$$ This concludes the construction of $\Gamma_{k,v}$ in \textbf{Case II}.

\subsection{Connectedness}\label{ss:connected} The graph $\Gamma_{k_0}$ is connected as it is the complete graph on $V_{k_0}$. The graphs $\Gamma_k$ are locally connected nearby each vertex in $V_k$ by construction of the $\Gamma_{k,v}$. Together with the fact that $\Gamma_k$ includes all bridges appearing in $\Gamma_{k-1}$ and that bridges include extensions to all future generations, it can be shown that $\Gamma_k$ is globally connected. See \cite[\S8.3]{badger-schul} for sample details.

\subsection{Start of the length estimates}\label{ss:start}

Let $\pi : G \to \R^{n_1}$ be the horizontal projection.
 Given $E$, a nonempty collection of abstract edges of $\bigcup_{k=k_0}^\infty V_k$ (for example $\Gamma_k$), we define its \emph{projected length} $\ell(E)$ by
\begin{align} \label{e:projected-length}
  \ell(E) := \sum_{(u,v) \in E} |\pi(u) - \pi(v)|.
\end{align} (This concept did not appear in \cite{badger-schul}.)
We remark that the projected length may be larger than the length of the curve in $\R^{n_1}$ formed by projecting $\bigcup_{k=k_0}^\infty V_k$ into $\R^{n_1}$ and connecting pairs of points whose vertices in $E$ are contained in an edge. The difference is that the quantity above might over-count the length since the projected line segments could overlap.

Our primary task is to verify the following bound on $\ell(\Gamma_k)$:
\begin{lemma} \label{l:abs-goal}
  There exists a constant $C > 0$ depending only on $G$ and $C^\star$ so that
  \begin{equation}\label{e:abs-goal}
    \ell(\Gamma_k) \leq C\left(2^{-k_0} + \sum_{j={k_0+1}}^k\sum_{v\in V_j} \alpha_{j,v}^{2s} 2^{-j} \right)\quad\text{for all }k\geq k_0+1.
  \end{equation}
\end{lemma}

For convenience, in the rest of this section, we write $a\lesssim b$ to denote $a\lesssim_{G,C^\star} b$. Let us first see how Proposition \ref{p:goal} follows from this lemma.

\begin{proof}[Proof of Proposition \ref{p:goal} given Lemma \ref{l:abs-goal}]
First, assume that for some constant $C_1>0$ depending on at most $G$ and $C^\star$, we know that for all $k \geq k_0+1$,
\begin{align}
    \sum_{(u,v) \in \Gamma_k} d(u,v)
    \leq C_1 \left( \ell(\Gamma_k) + \sum_{j=k_0+1}^k \sum_{v \in V_j} \alpha_{j,v}^{2s} 2^{-j} \right). \label{e:haus-len-bound}
\end{align} Let $\widehat{\Gamma}_k$ be a \emph{geometric realization} of $\Gamma_k$ in $G$ formed by drawing a geodesic in $G$ for each edge $(u,v)\in\Gamma_k$ and taking the closure of the union of these geodesics. Observe that $\widehat{\Gamma}_k\subset B(x_0,3C^\star)$ by the triangle inequality, since $u,v\in B(x_0,C^\star)$ for each $(u,v)\in\Gamma_k$ and $\diam B(x_0,C^\star)=2C^\star$.
Together, \eqref{e:Gamma-bound}, \eqref{e:abs-goal}, and \eqref{e:haus-len-bound} yield
\begin{align}
    \cH^1(\widehat\Gamma_k) \leq C_2 \left( 2^{-k_0} + \sum_{j=k_0+1}^\infty \sum_{v \in V_j} \alpha_{j,v}^{2s} 2^{-j} \right)<\infty \quad\text{for all }k \geq k_0+1, \label{e:haus-corollary}
\end{align} where $C_2$ is a constant depending on at most $G$ and $C^\star$. Let $(\widehat \Gamma_{k_j})_{j=1}^\infty$ be any subsequence of $(\widehat\Gamma_k)_{k=k_0}^\infty$ that converges in the Hausdorff metric, say $\Gamma=\lim_{j\rightarrow\infty} \widehat{\Gamma}_{k_j}$. Then by Go\l{}ab's semicontinuity theorem, which is valid in any metric space (see \cite{AO-curves}), $\Gamma$ is a rectifiable curve and $\Haus^1(\Gamma)\leq \liminf_{j\rightarrow\infty} \Haus^1(\widehat{\Gamma}_{k_j})<\infty$ by \eqref{e:haus-corollary}. That is to say, $\Gamma$ satisfies \eqref{e:Gamma-conclusion}. Also, we know that $\Gamma\subset B(x_0,3C^\star)$, since each $\widehat\Gamma_k\subset B(x_0,3C^\star)$. Back in \S\ref{ss:start-start}, we noted that $V_{k_j}$ converges in the Hausdorff metric to a compact set $V \subset B(x_0,C^\star)$. Since $V_k \subset \widehat\Gamma_k$, it follows that $V \subset \Gamma$, as well. Therefore, we have reduced the proof of Proposition \ref{p:goal}, given Lemma \ref{l:abs-goal}, to verifying \eqref{e:haus-len-bound}.

Suppose first that $(u,v) \in \Gamma_k$ is a pair which is not part of an extension $E[i,z]$ included in $\Gamma_k$.
If this edge was added to $\Gamma_{j,w}$ in \textbf{Case I} above (noting that it is only possible for $j < k$ when $(u,v)$ is the ``central span'' of a bridge $B[j,u,v]$), then $u,v \in V_j \cap B_{j,w}$ and $\alpha_{j,\hat{v}} \geq \varepsilon$
for some $\hat{v} \in V_j \cap B_{j,w}$. Thus,
\begin{equation*}
  d(u,v) \leq \diam B_{j,w}\leq 130C^\star2^{-j} \leq 130C^\star \varepsilon^{-2s}\alpha_{j,\hat{v}}^{2s}2^{-j}.
  \end{equation*}
Since each $B_{j,w}$ contains boundedly many pairs $(u,v)$ depending only on $G$ and $C^\star$, and further, each $\hat{v}$ is selected by a bounded number of points $w$, we may choose $C_1$ large enough so that the sum of $d(u,v)$ over all such pairs $(u,v)$ is bounded from above by
\begin{equation*}
C_1  \varepsilon^{-2s} \sum_{j=k_0+1}^k \sum_{\hat{v} \in V_j} \alpha_{j,\hat{v}}^{2s} 2^{-j}.
  \end{equation*}
If $(u,v)$ was added in \textbf{Case II}, then we get from \eqref{e:flat} that
  \begin{align*}
    d(u,v) \leq (1 + C\varepsilon^{2s}) |\pi(u) - \pi(v)|.
  \end{align*}
  Choosing $C_1 \geq 1 + C\varepsilon^{2s}$ ensures that the sum of $d(u,v)$ over all pairs $(u,v)$ discussed here is bounded from above by
  \begin{equation*}
  C_1 \sum_{(u,v) \in \Gamma_k} |\pi(u) - \pi(v)| = C_1 \ell(\Gamma_k).
  \end{equation*}

We now bound the length of all extensions $E[j,w]$ in $\Gamma_k$.
If $E[j,w]$ was added to $\Gamma_{j,v}$ in \textbf{Case I} for some $v \in V_j$, then there is some $\hat{v} \in V_j \cap B_{j,v}$ so that $\alpha_{j,\hat{v}} \geq \varepsilon$.  We then get
\begin{align}
 \sum_{(u',u'') \in E[j,w]} d(u',u'') \leq C^\star 2^{-j+1} \leq 2C^\star \varepsilon^{-2s} \alpha_{j,\hat{v}}^{2s} 2^{-j}. \label{e:ext-bound}
\end{align}
As each $\Gamma_{j,v}$ can only have boundedly many such extensions and each $V_j \cap B_{j,v}$ has boundedly many elements, we may conclude that the sum of $d(u',u'')$ over all edges $(u',u'')$ in such extensions is bounded by
a constant multiple of
\begin{align*}
    2C^\star \varepsilon^{-2s} \sum_{j=k_0}^k \sum_{v \in V_j} \alpha_{j,v}^{2s}2^{-j}.
\end{align*}
For extensions contained in a bridge $B[j, w, w']$ that were added in \textbf{Case II}, we get a bound as follows:
  \begin{align*}
    \sum_{(u',u'') \in E[j,w]} d(u',u'') + \sum_{(u',u'') \in E[j,w']} d(u',u'') \overset{\eqref{e:ext-bound}}{\leq} 4C^\star 2^{-j} \leq \frac{4}{30}|\pi(w) - \pi(w')|.
  \end{align*}
Thus, by increasing the lower bound $C_1\geq 1+C\varepsilon^{2s}$ to $C_1\geq 2+C\varepsilon^{2s}$, we can account for all such extensions. This completes the proof of \eqref{e:haus-len-bound}.
\end{proof}

The rest of this section is now dedicated to proving Lemma \ref{l:abs-goal}. Roughly speaking, we would like to bound the length of $\Gamma_{k_0}$ by $C 2^{-k_0}$ and to bound $\ell(\Gamma_k)$ by $\ell(\Gamma_{k-1})+C \sum_{v\in V_k} \alpha_{k,v}^{2s} 2^{-k}$ for all $k\geq k_0$ and some $C$ independent of $k$.
At each step, we will ``pay" for the length of $\Gamma_k$ with the length of $\Gamma_{k-1}$ plus some extra accumulation $C\sum_{v\in V_k}\alpha_{k,v}^{2s} 2^{-k}$.
The main difficulty arises when attempt to ``pay'' for an edge $(v',v'')$ in $\Gamma_k$ when
either of its vertices
is close to a terminal vertex from \textbf{Case II} of the construction.
This is because, in this case, the old curve may not be long enough to ``pay'' for
$|\pi(v')-\pi(v'')|$.
To address this issue, we will take advantage of a ``prepayment'' technique called phantom length originating in Jones' original traveling salesman construction \cite{Jones-TSP} (also see \cite{Lerman}).

\subsection{Phantom length}\label{ss:phantom} Below, it will be convenient to have notation to refer to the vertices appearing in a bridge. For each extension $E[k,v]=\bigcup_{i=0}^\infty (v_i,v_{i+1})$, we define the corresponding \emph{extension index set} $I[k,v]$ by $$I[k,v]=\{(k+i,v_i):i\geq 0\}.$$ For each bridge $B[k,v',v'']$, we define the corresponding \emph{bridge index set} $I[k,v',v'']$ by $$I[k,v',v'']=I[k,v']\cup I[k,v''].$$

Following \cite{badger-schul}, for all $k\geq k_0$ and $v\in V_k$, we define the \emph{phantom length associated with the pair} $(k,v)$ as $p_{k,v} := 3 C^\star 2^{-k}$.
If $B[k,v',v'']$ is a bridge between vertices $v',v''\in V_k$, then the totality $p_{k,v',v''}$ of phantom length associated to pairs in $I[k,v',v'']$ is given by $$p_{k,v',v''}:=3C^\star\left(2^{-k}+2^{-(k+1)}+\cdots\right)+3C^\star\left(2^{-k}+2^{-(k+1)}+\cdots\right) =12 C^\star 2^{-k}.$$

During the proof, we will track phantom length at certain pairs $(k,v)$ with $v\in V_k$ as we now describe. For the initial generation, define the \emph{index set} $\phan(k_0)$ by $$\phan(k_0):= \{(k_0,v):v\in V_{k_0}\}.$$ Suppose that $\phan(k_0),\dots,\phan(k-1)$ have been defined for some $k\geq k_0+1$, where the index sets already defined satisfy the following two properties.
\begin{itemize}
\item \emph{Bridge property:} For all $j \in \{ k_0 , \dots, k-1 \}$, if a bridge $B[j,w',w'']$ was introduced in $\Gamma_j$, then $\phan(j)$ contains $I[j,w',w'']$.
\item \emph{Terminal vertex property:} Let $w\in V_{k-1}$ and suppose $\ell$ is a horizontal line with
$$
y \in \ell \cdot \delta_{2^{-(k-1)}} (B_{\mathbb{R}^n}(\varepsilon^s))
\quad\text{for all }y\in V_{k-1}\cap B(w,30 C^\star 2^{-(k-1)}).
$$
Let $\pi_\ell:G \to \mathbb{R}$ be the composition of $\pi$ with the orthogonal projection in $\R^{n_1}$ onto $\ell$ and the identification of $\ell$ with $\mathbb{R}$ as before. If there does not exist $$w'\in V_{k-1}\cap B(w,30 C^\star 2^{-(k-1)})\quad\text{with } \pi_\ell(w') < \pi_\ell(w)$$ or there does not exist $$w''\in V_{k-1}\cap B(w,30 C^\star 2^{-(k-1)})\quad\text{with }\pi_\ell(w'') > \pi_\ell(w),$$ then $(k-1,w)\in \phan(k-1)$.
\end{itemize}
(Note that $\phan(k_0)$ satisfies both properties trivially since, by definition, $\phan(k_0)$ includes $(k_0,v)$ for every $v\in V_{k_0}$.)
We will form $\phan(k)$ via $\phan(k-1)$ as follows. Initialize the set $\phan(k)$ to be equal to $\phan(k-1)$. Next, delete all pairs $(k-1,w)$ and $(k,z)$ appearing in $\phan(k-1)$ from $\phan(k)$. Lastly, for each vertex $v\in V_k$, include additional pairs in $\phan(k)$ according to the following rules: \begin{itemize}
\item \textbf{Case I:} Suppose that $v\in V_k$ and $\alpha_{k,w}\geq \varepsilon$ for some $w \in V_k \cap B_{k,v}$. Include $(k,v')$ in $\phan(k)$ for all vertices $v'\in V_k\cap B_{k,v}$ and include $I[k,v',v'']$ as a subset of $\phan(k)$ for every bridge $B[k,v',v'']$ in $\Gamma_{k,v}$.
\item \textbf{Case II:} Suppose that $v\in V_k$ and $\alpha_{k,w}<\varepsilon$ for all $w \in V_k \cap B_{k,v}$.
\begin{itemize}
    \item \textbf{Case II-NT:} Suppose $\Gamma_{k,v}^R$ or $\Gamma_{k,v}^L$ is defined by \textbf{Case II-NT}. Do nothing.
\item \textbf{Case II-T1:} Suppose $\Gamma_{k,v}^R$ or $\Gamma_{k,v}^L$ is defined by \textbf{Case II-T1}. Include $(k,v)\in \phan(k)$.
\item \textbf{Case II-T2:} Suppose $\Gamma_{k,v}^R$ or $\Gamma_{k,v}^L$ is defined by \textbf{Case II-T2}.
When $\Gamma_{k,v}^R$ is defined by \textbf{Case II-T2}, include $I[k,v,v_1]$ as a subset of $\phan(k)$. When $\Gamma_{k,v}^L$ is defined by \textbf{Case II-T2}, include $I[k,v_{-1},v]$ as a subset of $\phan(k)$. In particular, note that $(k,v)$ is included in $\phan(k)$.
\end{itemize}
\end{itemize}
The phantom length associated to deleted pairs will be available to pay for the length of edges in $\Gamma_k$ near terminal vertices in $V_k$ and to pay for the phantom length of pairs in $\phan(k)\setminus\phan(k-1)$. Verification that $\phan(k)$ satisfies the bridge and terminal vertex properties is the same as the Euclidean case. See \cite[p.~30]{badger-schul} for details.

\subsection{Proof of \eqref{e:abs-goal} given \eqref{e:sum-c}} The projected length of a set of edges is defined in \eqref{e:projected-length}. Suppose that there exists $C=C(G,C^\star)$ such that for all $k\geq k_0+1$,
\begin{equation}\begin{split}
&\ell(\edges(k)) +\ell(\bridges(k)) + \sum_{(j,u) \in \phan(k)} p_{j,u} \\
 &\quad\leq \ell(\edges(k-1)) + \sum_{(j,u) \in \phan(k-1)} p_{j,u} +C\sum_{v\in V_k} \alpha_{k,v}^{2s} 2^{-k}+ \frac{5}{6}\,\ell(\bridges(k)),\label{e:sum-c}\end{split}\end{equation} where $\edges(k)$ denotes the set of all pairs $(v',v'')$ included in $\Gamma_{k}$ that are not part of a bridge $B[j,w',w'']$ included in $\Gamma_k$, $\bridges(k)$ denotes the union of all bridges $B[k,v',v'']$ included in $\Gamma_{k}$, and $\phan(k)$ is defined in \S \ref{ss:phantom}. Recall the definition of $\Gamma_k$ in \eqref{e:Gamma-def} and also that $\Gamma_{k_0}$ contains no bridges. Applying \eqref{e:sum-c} telescopically $k-k_0$ times yields
\begin{equation*}\begin{split}
  &\ell(\Gamma_k) = \ell(\edges(k))+\sum_{j=k_0+1}^k \ell(\bridges(j))\\
  &\quad \leq \underbrace{\ell(\edges(k_0)) + \sum_{(j,u)\in \phan(k_0)} p_{j,u}}_I + C \sum_{j=k_0+1}^k \sum_{v \in V_j} \alpha_{j,v}^{2s} 2^{-j} + \underbrace{\frac{5}{6} \sum_{j=k_0+1}^k \ell(\bridges(j))}_{\two}.
\end{split}\end{equation*} Since $V_{k_0}\subset B(x,C^\star 2^{-k_0})$ for some $x$ and $V_{k_0}$ is $2^{-k_0}$-separated, the number of points in $V_{k_0}$ is bounded, depending only on $G$ and $C^\star$. It follows that $I\lesssim_{G,C^\star} 2^{-k_0}$. Also, since $\Gamma_k$ includes all bridges introduced in $\Gamma_{k_0+1},\dots, \Gamma_{k}$, we have $\two\leq \frac{5}{6}\ell(\Gamma_k)$. Thus,
\begin{align*}
  \frac{1}{6} \ell(\Gamma_k) \lesssim_{G,C^*} 2^{-k_0} + \sum_{j=k_0+1}^k \sum_{v \in V_j} \alpha_{j,v}^{2s} 2^{-j}.
\end{align*}
This proves \eqref{e:abs-goal} given \eqref{e:sum-c}.

\subsection{Proof of \eqref{e:sum-c}} \label{ss:sum-b} This section corresponds to \cite[\S9.4]{badger-schul}. Fix $k \geq k_0+1$. Our goal is to prove \eqref{e:sum-c}. As the projection $\pi: G \to \R^{n_1}$ is 1-Lipschitz, we have from \eqref{e:line-containment} that
\begin{align}
  \sup_{x \in (V_k \cup V_{k-1}) \cap B_{k,v}} \dist_{\R^{n_1}}(\pi(x), \pi(\ell_{k,v})) \leq \alpha_{k,v}^s 2^{-k}. \label{e:proj-alpha}
\end{align}
By an abuse of notation, we will refer to the projected line $\pi(\ell_{k,v})$ in $\R^{n_1}$ as $\ell_{k,v}$. It should always be clear from context to which line we are referring. Moreover, we will write $\pi_{k,v} : \R^{n_1} \to \R$ to denote orthogonal projection onto $\pi(\ell_{k,v})$ composed with identification of the line with $\mathbb{R}$. By \eqref{e:proj-alpha}, the sets $\pi(V_k)$ satisfy \cite[(8.1)]{badger-schul} with ``error'' $\alpha_{k,v}^s$. Thus, the estimate \eqref{e:sum-c} is almost a direct application of the proof of \cite[Proposition 8.1]{badger-schul}, except for the fact that $\pi(V_k)$ is not necessarily $2^{-k}$ separated. In \cite{badger-schul}, the separation condition is primarily used to get a bound on $\#\pi(V_k)$, but in our context this conclusion follows from a bound on $\#V_k$. We sketch some details for the reader's convenience.

It follows from the construction that for all $k\geq k_0$, $$(v',v'') \in \edges(k) \Longrightarrow |\pi(v') - \pi(v'')|<30 C^\star 2^{-k},$$ $$B[k,v',v'']\subset \bridges(k)\Longrightarrow 30 C^\star 2^{-k}\leq |\pi(v') - \pi(v'')| <130 C^\star 2^{-k}.$$ Furthermore, if $B[k,v',v'']\subset\bridges(k)$, then
\begin{equation*}\begin{split}
  \ell(B[k,v',v'']) &=  |\pi(v') - \pi(v'')|+\ell(E[k,v'])+\ell(E[k,v''])\\
  &\overset{\eqref{e:ext-bound}}\leq |\pi(v') - \pi(v'')|+ 4C^\star 2^{-k}<1.14 |\pi(v')-\pi(v'')|,
\end{split}\end{equation*} where, in addition to \eqref{e:ext-bound}, we used the fact that $\pi$ is $1$-Lipschitz.

Each graph $\Gamma_k$ gives rise to a geometric realization of $\pi(\Gamma_k)$ in $\R^{n_1}$ by taking a union of line segments in $\R^{n_1}$ corresponding to abstract edges:
\begin{align*}
  \cE_k := \bigcup_{(u,v) \in \Gamma_k} [\pi(u), \pi(v)].
\end{align*}
Since $\Gamma_k$ is connected, $\cE_k$ is as well. The length of an edge in $\Gamma_k$ agrees with the Hausdorff measure $\Haus^1$ of the corresponding line segment in $\cE_k$.  We will call line segments in $\cE_k$ ``edges'' and unions of line segments with the extensions at their endpoints ``bridges'' using the same classification as in \S\ref{ss:algo}.
Given $v \in V_k$, we let $\cE_{k,v}$ denote the associated line segments from $\Gamma_{k,v}$.

Edges and bridges forming $\cE_k$ and ``new" phantom length associated to pairs in the set $\phan(k)\setminus\phan(k-1)$ may enter the local picture $\cE_{k,v}$ of $\cE_k$ near $\pi(v)$ for several vertices $v\in V_k$, but they each only need to be accounted for once to estimate the left hand side of \eqref{e:sum-c}. Continuing to follow \cite{badger-schul}, we prioritize as follows:
\begin{enumerate}
\item[1.] \textbf{Case I} edges, \textbf{Case I} bridges, \textbf{Case I} phantom length.
\item[2.] \textbf{Case II-T1} phantom length and edges that are near \textbf{Case II-T1} terminal vertices (where here and below \emph{near} means at a distance at most $2C^\star 2^{-k}$);
\item[3.] \textbf{Case II-T2} bridges, \textbf{Case II-T2} phantom length, and (parts of) edges that are near \textbf{Case II-T2} terminal vertices;
\item[4.] remaining (parts of) edges, which are necessarily not near \textbf{Case I} vertices and \textbf{Case II-T1} and \textbf{Case II-T2} terminal vertices.
\end{enumerate}

\textbf{First Estimate (Case I):} This is analogous to the estimates on \cite[p.~33]{badger-schul}. Since $\#(V_k\cap B_{k,v})\lesssim_{G,C^*}1$, we may charge the length of edges, new bridges, and new phantom length appearing in $B_{k,v}$ to $\alpha_{k,u}^{2s}2^{-k}$ for some vertex $u\in B_{k,v}$ with $\alpha_{k,u}\geq \varepsilon$.

\textbf{Second Estimate (Case II-T1):} As long as we choose $\varepsilon$ to be small enough so that $2(1+C\varepsilon^{2s})<2.5$, where $C$ is the constant in Proposition \ref{p:proj-prop}, this estimate is the same as the one on \cite[p.~33]{badger-schul}. Use Proposition \ref{p:proj-prop} in place of \cite[Lemma 8.3]{badger-schul}.

\textbf{Third Estimate (Case II-T2):} This estimate introduces the term $\frac{5}{6}\ell(\bridges(k))$ in \eqref{e:sum-c}. While it is similar to the estimate on \cite[pp.~33--34]{badger-schul}, the proof there uses a notion of the ``core'' of a bridge, which we have not introduced. Thus, we record some details.
Suppose that $\alpha_{k,u} < \varepsilon$ for all $u \in V_k \cap B_{k,v}$
and $v$ is \textbf{T2} terminal to the right. (The case when $v$ is terminal to the left can be handled analogously.)
Let $v_1\in V_k$ and $w_{v,r},w_{v,r+1}\in V_{k-1}$
denote vertices appearing in the definition of $\Gamma_{k,v}^R$. We will pay for $p_{k,v,v_1}$,
the projected length of the bridge $B[k,v,v_1]$,
and the length (Hausdorff measure) of the part of any segments in $\mathcal{E}_k$ inside of
$B_{\R^{n_1}}(\pi(v),2C^\star 2^{-k}) \cup B_{\R^{n_1}}(\pi(v_1), 2C^\star 2^{-k})=:BB$
with at least one endpoint which is the projection of a point in $B(v,2C^\star 2^{-k}) \cup B(v_1,2C^\star 2^{-k}) =: U$.

First, the totality $p_{k,v,v_1}$ of phantom length associated to all vertices in $B[k,v,v_1]$ is $12C^\star 2^{-k}$.
Second,
$$
\ell(B[k,v,v_1])\overset{\eqref{e:ext-bound}}{\leq} 4C^\star 2^{-k}
+
|\pi(v) - \pi(v_1)|
\leq
8C^\star2^{-k}+ |\pi(w_{v,r}) - \pi(w_{v,r+1})|
$$
because $d(v,w_{v,r})<2C^\star 2^{-k}$
and
$d(v_1,w_{v,r+1})< 2C^\star 2^{-k}$.
Finally, by our choice of $\varepsilon$ in the \textbf{Second Estimate}
as before,
since $\alpha_{k,v}<\varepsilon$ and $\alpha_{k,v_1}<\varepsilon$,
the total length of parts of edges inside $BB$ does not exceed $5 C^\star 2^{-k}$. Altogether,
\begin{align*}
  \ell(B[k,v,v_1]) + p_{k,v,v_1} &+ \sum_{\substack{(v',v'')\in\edges(k) \\ \{v',v''\} \cap U \neq \emptyset}} \Haus^1 \left([\pi(v'), \pi(v'')] \cap BB\right) \\
& \leq |\pi(w_{v,r}) - \pi(w_{v,r+1})| + 8 C^\star 2^{-k} + 12 C^\star 2^{-k}  + 5 C^\star 2^{-k} \\
&\leq
|\pi(w_{v,r}) - \pi(w_{v,r+1})| + \frac{25}{30} |\pi(v) - \pi(v_1)|.
\end{align*} In the last inequality, we used $|\pi(v) - \pi(v_1)| \geq 30C^\star 2^{-k}$. In fact, this is the entire rationale for the requiring bridges to have large spans. We remark that $(w_{v,r},w_{v,r+1}) \in \edges(k-1)$ and the assignment $v\mapsto (w_{v,r},w_{v,r+1})$ when $v$ is \textbf{T2} terminal to the right is one-to-one.

We have now paid for all phantom length, all bridges, and those parts of edges that are within a ball of radius $2C^\star 2^{-k}$ from the projection of a \textbf{Case II-T1} and \textbf{Case II-T2} terminal vertex.
The next estimate will pay for all remaining edge lengths.

\textbf{Fourth Estimate (Case II-NT):}
Suppose $(v',v'') \in \edges(k)$ is an edge for which
the length of $[\pi(v'),\pi(v'')]$ has not yet been fully paid,
and fix a point $y \in V_{k-1}$ so that $d(y, v') < C^\star2^{-k}$.
Then $\alpha_{k,v'}<\varepsilon$ and $\alpha_{k,v''}<\varepsilon$,
and there are $u',u'' \in \R^{n_1}$ such that
$[u',u'']$ is the largest closed subinterval of $[\pi(v'),\pi(v'')]$ so that $u'$ and $u''$ lie at distance at least $2C^\star 2^{-k}$ from the projections of \textbf{II-T1} and \textbf{II-T2} terminal vertices of $V_k \cap B_{k,v'}$.
Only $\Haus^1([u',u''])$ remains to be paid for as we have already paid for the rest of the length of $[\pi(v'),\pi(v'')]$ in the \textbf{Second} and \textbf{Third Estimate}.
By Proposition \ref{p:proj-prop} and \eqref{e:proj-alpha},
\begin{align*}
  |u' - u''| &\leq (1+C\alpha_{k,v'}^{2s}) |\pi_{k,v'}(u') - \pi_{k,v'}(u'')|\\
  &\leq \Haus^1([\pi_{k,v'}(u'),\pi_{k,v'}(u'')])+C \alpha_{k,v'}^{2s} |\pi(v')-\pi(v'')|\\
          &\leq \Haus^1([\pi_{k,v'}(u'),\pi_{k,v'}(u'')])+30 C^\star C \alpha_{k,v'}^{2s} 2^{-k}.
\end{align*} This is analogous to the first displayed equation in the Fourth Estimate on \cite[p.~34]{badger-schul}, except that we have replaced $90=3\cdot 30$ with $30C$, where $C$ is from Proposition \ref{p:proj-prop}. The argument on \cite[pp.~34--35]{badger-schul} shows how to efficiently charge $\Haus^1([\pi_{k,v'}(u'),\pi_{k,v'}(u'')])$ to $\ell(\edges(k-1))$ and $\sum_{u\in V_k} \alpha_{k,u}^{2s}2^{-k}$.

Carefully tallying the four estimates above, one obtains \eqref{e:sum-c}.

\section{Stratified \texorpdfstring{$\beta$}{beta} numbers for locally finite measures}\label{ss:newbetas}

We continue to let $G$ denote the Carnot group fixed at the start of \S\ref{s:construction}. Further, from here through the end of \S\ref{s-proofs}, we let $\Delta=\bigcup_{k\in\ZZ}\Delta_k$ be a fixed system of ``dyadic cubes'' on $G$ given by Theorem \ref{t-KRS} with respect to a fixed family of nested $2^{-k}$-nets $(X_k)_{k\in\ZZ}$ for $G$.

 Motivated by \cite{badger-schul} and \cite{Li-TSP}, we wish to design a useful gauge of how close a locally finite measure $\mu$ on $G$ is to being supported on a horizontal line in a neighborhood of a cube $Q\in\Delta$, which both allows for the possibility of non-doubling measures and incorporates distance in each of the layers $G_1,\dots,G_s$ of $G$. The definition of $\beta^*(\mu,Q)$ proceeds in several stages.

\begin{definition} For all $x,y \in G$ and $r > 0$, define
\begin{align*}
  \tilde{\beta}(x,y;r)^{2s} := \sum_{i=1}^s \left( \frac{d_i(\pi_i(x),\pi_i(y))}{r} \right)^{2i}.
\end{align*} Further, define $\tilde{\beta}(x,E;r):=\inf_{y\in E} \tilde{\beta}(x,y;r)$ for all nonempty $E\subset G$.
\end{definition}

\begin{definition}[non-homogeneous stratified $\beta$ numbers] \label{nh-beta} Let $\mu$ be a locally finite Borel measure on $G$. For any Borel set $Q$, with $0<\diam Q<\infty$, and any horizontal line $L$, define
$$
\beta(\mu,Q,L)^{2s}
:=
\int_Q \tilde{\beta}(z,L;\diam Q)^{2s}\, \frac{d\mu(z)}{\mu(Q)}.
$$ Further, define $\beta(\mu,Q):=\inf_{L} \beta(\mu,Q,L)$, where $L$ runs over all horizontal lines in $G$.
\end{definition}

\begin{definition} For $Q \in \Delta_k$, $k\in\ZZ$, we define the family $\near(Q)$ of cubes \emph{near} $Q$ by
$$
\near (Q) := \{ R \in \Delta_{k-1} \cup \Delta_{k} \, : \, 2B_R \cap 588 B_Q \neq \emptyset \},
$$
where $588B_Q=B(x_Q,1568\cdot 2^{-k})$ and $x_Q$ is the center of $Q$.\end{definition}

%\begin{figure}
%\begin{center}\includegraphics[width=3.25in]{anisotropic.png}\end{center}
%\caption{In $G=\RR^2$: Illustration of pattern formed by overlapping balls $2B_R$ with $R\in\near(Q)$ inside of the window $40B_Q$. The central region $2B_Q$ is highlighted.}\label{fig:anisotropic}
%\end{figure}

\begin{definition}[anisotropic stratified $\beta$ numbers] Let $\mu$ be a locally finite Borel measure on $G$. For every $Q\in\Delta$, define
\begin{align*}
  \beta^*(\mu,Q)^{2s} := \inf_L \max_{R \in \near(Q)}  \beta(\mu,2B_R,L)^{2s} \min\left\{1, \frac{\mu(2B_R)}{\diam 2B_R} \right\}
\end{align*}
where the infimum is over the set of all horizontal lines in $G$.\end{definition}

\begin{remark} The numbers $\beta^*(\mu,Q)$ are a rough gauge of how far $\mu\res 588B_Q$ is from a measure supported on a horizontal line. They are \emph{anisotropic} insofar as the normalizations $$\frac{1}{\mu(2B_R)}\min\left\{1, \frac{\mu(2B_R)}{\diam 2B_R}\right\}$$ of the integral of the scale-invariant stratified distance of points in $2B_R$ to a horizontal line $L$ against the measure $\mu$, i.e. $$\sum_{i=1}^s\int_{2B_R} \left(\frac{d_i(\pi_i(z),\pi_i(L))}{\diam 2B_R}\right)^{2i}\,d\mu(z),$$ vary independently in the regions $2B_R$ that emanate in \emph{different directions and distances} from the central region $2B_Q$ inside of the window $588B_Q$.
% See Figure \ref{fig:anisotropic}.
\end{remark}

\begin{remark}\label{J-is-S} Let $x\in G$, let $\mathcal{T}$ denote the \hyperref[trees-and-leaves]{tree of cubes} $Q\in\Delta$ such that $x\in Q$ and $\side Q\leq 1$, and let $b(Q)=\beta^*(\mu,Q)^{2s}\diam Q$ for all $Q\in\mathcal{T}$. Then $J^*(\mu,x)=S_{\mathcal{T},b}(\mu,x)$, where $J^*(\mu,x)$ is given by \eqref{Jstar} and $S_{\mathcal{T},b}(\mu,\cdot)$ is given by Definition \ref{S-functions}.\end{remark}

\begin{remark}\label{near-contained} Let $Q\in\Delta_k$ and let $R\in\near(Q)\cap \Delta_{k-1}$. Then $$U(x_R,\tfrac13 \cdot 2^{-k})=U_R\subset R\subset 2B_R\subset B(x_R,\tfrac{32}{3}\cdot 2^{-k}).$$ Because $2B_R\cap 588B_Q\neq \emptyset$, we conclude that \begin{equation}\label{3R-contain} 2B_R\subset B(x_Q,1568\cdot 2^{-k}+\diam 2B_R)\subset B(x_Q,1592 \cdot 2^{-k})=597B_Q.\end{equation} Further, since cubes in $\near(Q)\cap\Delta_{k-1}$ are pairwise disjoint, a volume doubling argument yields $\# \near(Q)\cap\Delta_{k-1} \lesssim 1$, where the implicit constant depends only on $G$. A similar computation shows that $2B_R\subset 597 B_Q$ for all $R\in\near(Q)\cap \Delta_k$ and $\#\near(Q)\cap\Delta_{k}\lesssim 1$, as well.\end{remark}

\begin{remark}
Midpoint convexity of $x\mapsto x^p$ when $p>1$ gives us a quasitriangle inequality for the stratified distance:
\begin{align}
  \tb(x,y;r)^{2s} \leq 2^{2s-1} \left(\tb(x,z;r)^{2s} + \tb(z,y;r)^{2s} \right). \label{e:quasi-triangle}
\end{align}
We also have a change of scales inequalities:
\begin{align}\label{quasimonotone}
  \tb(x,y;t)\leq \tb(x,y;r) \leq \frac{t}{r}\tb(x,y;t)\quad\text{whenever $t\geq r>0$.}
\end{align}\end{remark}

\section{Rectifiability of sets on which the Jones function is finite}
\label{s-suff}

Suppose that $\mu$ is a locally finite Borel measure on $G$.
For each cutoff $c>0$, we define the truncated beta number $\beta^{*,c}(\mu,Q)$ for $Q\in\Delta$ by ignoring cubes $R\in\near(Q)$ on which $\mu$ has small 1-dimensional density. That is,
\begin{align}
\beta^{*,c}(\mu,Q)^{2s} := \inf_{L} \max \left\{ \beta(\mu,2B_R,L)^{2s} \min\{c,1\} : R \in \near(Q),\ \frac{\mu(2B_R)}{\diam 2B_R} \geq c \right\},
\end{align} where as usual the infimum runs over all horizontal lines in $G$ and $\beta(\mu,2B_R,L)^{2s}$ is defined in Definition \ref{nh-beta}. If there are no $R\in\near(Q)$ with $\mu(2B_R)\geq c\diam 2B_R$, simply assign $\beta^{*,c}(\mu,Q)=0$.
The associated density-normalized Jones function is defined by
\begin{equation}
J^{*,c}(\mu,x) := \sum_{Q \in \Delta_+} \beta^{*,c}(\mu,Q)^{2s} \diam(Q)\, \frac{\chi_Q(x)}{\mu(Q)}\quad\text{for all }x\in G,\end{equation} where $\Delta_+$ is the set of cubes of side length at most 1. It is immediate from the definitions that $\beta^{*,c}(\mu,Q)\leq \beta^*(\mu,Q)$ for all $Q\in\Delta$ and $J^{*,c}(\mu,x)\leq J^*(\mu,x)$ for all $x\in G$.

This section is devoted to the proof of the following theorem.
\begin{theorem}
\label{t-sec5main}
Let $\mu$ be a locally finite Borel measure on $G$. For every $c>0$,
\begin{equation}\label{mu-c}
\mu \res \{ x \in G :  \lD{1}(\mu,x)> 2c \text{ and } J^{*,c}(\mu,x) < \infty \}
\end{equation}
is 1-rectifiable.
\end{theorem}

Our main tool for constructing a rectifiable curve passing through a set of points is Proposition \ref{p:goal}. In order to find (countably many) rectifiable curves covering the set where $\lD1(\mu,x)$ is positive and $J^{*,c}(\mu,x)$ is finite, we need to extract enough data to input to the proposition. In \cite{badger-schul}, the convexity of the Euclidean distance of a point to a line was used to find points $z_Q$ (centers of mass) for each $Q\in\Delta$ for which we could control the distance of $z_Q$ to any line $L$ using $\beta$ numbers. This approach is not available in an arbitrary Carnot group $G$, so we reverse the process. First, we associate a special line $\ell_Q$ to each $Q\in\Delta$. In particular, with $\mu$ and $c>0$ fixed, for each $Q\in\Delta$, choose any horizontal line $\ell_Q$ so that \begin{equation} \label{special-line}
\max \left\{\beta(\mu, 2B_R, \ell_Q)^{2s}\min\{c,1\}:R\in\near(Q),\frac{\mu(2B_R)}{\diam 2B_R}\geq c\right\} \leq 2 \beta^{*,c}(\mu, Q)^{2s}.\end{equation} If there are no $R\in\near(Q)$ such that $\mu(2B_R)\geq c\diam 2B_R$, choose $\ell_Q$ arbitrarily or leave $\ell_Q$ undefined---we will never refer to it. Once we have fixed these lines, we may show that there exist points $\{z_R\}_{R\in\Delta}$ for which we can control the distance of $z_R$ to $\ell_Q$ whenever $R\in\near(Q)$ and $\mu(2B_R)\geq c\diam 2B_R$.

\begin{lemma}
\label{l-centerofmass} There exist points $\{z_R\}_{R\in\Delta}$ such that $z_R\in 2B_R$ for each $R\in\Delta$ and
\begin{equation}\label{z-to-l}
\tb(z_R,\ell_Q;\diam 2B_Q)\lesssim \tb(z_R,\ell_Q;\diam 2B_R) \lesssim \beta(\mu,2B_R,\ell_Q)\end{equation} for each $R$ and $Q$ in $\Delta$ with $R\in\near(Q)$ and $\mu(2B_R)\geq c\diam 2B_R$.\end{lemma}

\begin{proof} Fix $R\in\Delta$. Since $\diam 2B_Q\leq \diam 2B_R\leq 2\diam 2B_Q$ when $R \in \near(Q)$, the first inequality in \eqref{z-to-l} follows from \eqref{quasimonotone}, so it suffices to prove the second inequality. By definition, for any horizontal line $L$, $$\beta(\mu,2B_R,L)^{2s}=\int_{2B_R} \tb(z,L;\diam 2B_R)^{2s} \frac{d\mu(z)}{\mu(2B_R)}.$$ Thus, for each horizontal line $\ell_Q$ associated to some $Q\in\Delta$, Chebyshev's inequality gives
\begin{align*}
  \mu\left(\left\{ z \in 2B_R : \tilde{\beta}(z,\ell_Q;\diam 2B_R)^{2s} \geq C \beta(\mu,2B_R,\ell_Q)^{2s} \right\}\right)
    \leq \frac{\mu(2B_R)}{C}\quad\text{for all }C>1.\end{align*}
By an argument similar to Remark \ref{near-contained}, there exists a constant $N=N(G)<\infty$ such that $\#\{Q\in\Delta:R\in\near(Q)\} \leq N$. Choosing $C=2N>1$, it follows that $$
\mu\left( \bigcup_{\{Q:R\in\near(Q)\}} \{ z \in 2B_R : \tilde{\beta}(z,\ell_Q;\diam 2B_R)^{2s} \geq  2N \beta(\mu,2B_R,\ell_Q)^{2s} \}\right) \leq \frac{1}{2}\mu(2B_R).$$ Therefore, as long as $\mu(2B_R)>0$, there exists $z_R\in 2B_R$ such that \begin{equation}\label{zr-def} \tb(z_R,\ell_Q;\diam 2B_R)^{2s} \leq 2N \beta(\mu,2B_R,\ell_Q)^{2s}\end{equation} for all $Q\in\Delta$ such that $R\in\near(Q)$. Pick one such point for each $R\in\Delta$ such that $\mu(2B_R)>0$. (This includes all cubes $R\in\Delta$ such that $\mu(2B_R)\geq c\diam 2B_R$. For any $R\in\Delta$ with $\mu(2B_R)=0$, choose $z_R=x_R$ if desired.)
\end{proof}

The following lemma describe a scenario in which the whole set of leaves of a tree is contained in a rectifiable curve. Moreover, the length of such a curve can be controlled by the diameter or side length of the top cube and a sum involving $\beta^{*,c}(\mu,Q)^{2s}$.

\begin{lemma}\label{l-makeacurve} Let $\mu$ and $c$ be fixed as above. Suppose that $\mathcal{T}$ is a \hyperref[trees-and-leaves]{tree of cubes} such that
\begin{equation}
    \label{e-lowerreg}
\mu(2B_Q) \geq c \diam(2B_Q) \quad
\text{for all } Q \in \mathcal{T}\qquad\text{and}
\end{equation}
\begin{equation}\label{2s-sum-finite}
S_\mathcal{T}=\sum_{Q \in \mathcal{T}}\beta^{*,c}(\mu,Q)^{2s} \diam(Q) <\infty.
\end{equation} Then there exists a rectifiable curve $\Gamma$ with $\leaves(\mathcal{T}) \subset \Gamma$ such that
\begin{equation}\label{T-H1-bound} \mathcal{H}^1(\Gamma) \lesssim \side \Top(\mathcal{T}) +  \max\{c^{-1}, 1\} S_\mathcal{T}.\end{equation}
\end{lemma}

\begin{proof}If the \hyperref[trees-and-leaves]{set of leaves} of the tree is empty, the conclusion is trivial. Thus, we assume that $\leaves(\mathcal{T})\neq\emptyset$. Without loss of generality, we may further assume that every cube in $\mathcal{T}$ intersects $\leaves(\mathcal{T})$. (Delete any cubes without this property.) Let $\{\ell_Q\}_{Q\in\Delta}$ be given by \eqref{special-line} and let $\{z_R\}_{R\in\Delta}$ be given by Lemma \ref{l-centerofmass}.

We employ a traveling salesman algorithm for constructing rectifiable curves in Carnot groups from \S\ref{s:construction}. In particular, we will apply Proposition~\ref{p:goal} with parameters $$C^\star = 24\quad\text{and}\quad r_0 = \side\Top(\mathcal{T}).$$ To do so, we must identify a sequence $(V_k)_{k\geq 0}$ of point clouds satisfying conditions \hyperref[V1]{$(V_I)$}, \hyperref[V2]{$(V_\two)$}, \hyperref[V3]{$(V_\three)$} of the proposition and sequences $(\ell_{k,v})_{k\geq 0,v\in V_k}$ of lines and $(\alpha_{k,v})_{k\geq 0,v\in V_k}$ of linear approximation errors satisfying \eqref{e:line-containment} and \eqref{e:Gamma-bound}.

\emph{Point Clouds.} For each $k \geq 0$, define $Z_k:=\{z_Q:Q\in\mathcal{T}\text{ and }\side Q=2^{-k}r_0\}$ and choose $V_k$ to be a maximal $2^{-k}r_0$-separated subset of $Z_k$. By definition, $V_k$ satisfies $(V_I)$.

Suppose that $v_k \in V_k$ for some $k\geq 0$. Then $v_k=z_Q$ for some $Q \in \mathcal{T}$ with $\side Q=2^{-k}r_0$. Because every cube in $\mathcal{T}$ is part of an infinite chain, there exists $R\in \mathcal{T}$ with $\side R = (1/2)\side Q$ and $R \subset Q$. By maximality of $V_{k+1}$ in $Z_{k+1}$, there is $S\in\mathcal{T}$ with $\side S=\side R$ such that $z_S\in V_{k+1}$ and $d(z_S, z_R) \leq 2^{-(k+1)}r_0$. Hence $v_{k+1}:=z_S$ satisfies
\begin{align*}
d(v_k,v_{k+1}) =d(z_Q,z_S)&\leq d(z_Q, x_Q)+d(x_Q,x_R) + d(x_R,z_R) + d(z_R,z_S) \\ &\leq \left(\tfrac{16}{3}+\tfrac{8}{3}+\tfrac{8}{3}+\tfrac12\right)\cdot 2^{-k}r_0 < 12 \cdot2^{-k}r_0.
\end{align*}
Therefore, ($V_\two$) holds.

Similarly, suppose that $v_k \in V_k$ for some $k \geq 1$, again say that $v_k=z_Q$ for some $Q\in\mathcal{T}$ with $\side Q = 2^{-k}r_0$. Let $P\in\mathcal{T}$ be the parent of $Q$, which satisfies $\side P=2\side Q$ and $Q\subset P$. By maximality of $V_{k-1}$ in $Z_{k-1}$, there is $O\in\mathcal{T}$ with $\side O=\side P$ such that $z_O\in V_{k-1}$ and $d(z_O,z_P)\leq 2^{-(k-1)}r_0$. Hence $v_{k-1}:=z_O$ satisfies
\begin{align*}
d(v_k,v_{k-1}) = d(z_Q,z_O) &\leq d(z_Q,x_Q)+d(x_Q,x_P)+d(x_P,z_P)+d(z_P,z_O) \\
 &\leq  \left(\tfrac{16}{3} + \tfrac{16}{3} + \tfrac{32}{3} + 2\right)\cdot 2^{-k}r_0 < 24 \cdot 2^{-k}r_0
\end{align*} Therefore, ($V_\three$) holds.

\emph{Horizontal Lines and Linear Approximation Errors.} Next, we will describe how to choose the horizontal lines $\ell_{k,v}$ and errors
$\alpha_{k,v}$ for use in Proposition~\ref{p:goal}. For each $k \geq 0$ and $v \in V_k$, let $Q_{k,v}$ denote the cube $Q \in \mathcal{T}$ such that $\side Q=2^{-k}r_0$ and $v = z_Q$. Then let $\ell_{k,v}=\ell_{Q_{k,v}}$ be the horizontal line chosen just before Lemma \ref{l-centerofmass} to satisfy \eqref{special-line}.

Suppose that $k\geq 1$, let $v \in V_k$, let $Q=Q_{k,v}$, and let $$x\in (V_{k-1}\cup V_k)\cap B(v,65C^\star 2^{-k}r_0) = (V_{k-1}\cup V_k)\cap B(v,1560\cdot 2^{-k}r_0).$$ We must bound the distance of $x$ to $\ell_{k,v}$. Since $x\in V_{k-1}\cup V_k$, we can express $x=z_R$ for some $R=R_x\in\mathcal{T}$ with $\side Q\leq \side R\leq 2\side Q$. Note that $x\in 2B_R$ and $$d(x,x_Q)\leq d(x,v)+d(v,x_Q)\leq
1560\cdot 2^{-k}r_0 + \tfrac{16}{3}\cdot 2^{-k}r_0<1568\cdot 2^{-k}r_0.$$ Thus, $x\in 2B_R\cap 588 B_Q$, whence $R\in\near(Q)$. By Lemma \ref{l-centerofmass} and \eqref{quasimonotone}, we obtain $$\tb(x,\ell_{k,v};2^{-k}r_0)^{2s}\sim\tb(x,\ell_{k,v};\diam 2B_Q)^{2s}=\tb(z_R,\ell_{Q};\diam 2B_Q)^{2s} \lesssim \beta(\mu,2B_R,\ell_{Q})^{2s}.$$ Taking the maximum over all admissible $x$ and invoking \eqref{special-line} and \eqref{e-lowerreg}, we obtain $$\sup_{x\in (V_{k-1},V_k)\cap B(v,65C^\star2^{-k}r_0)}
\tb(x,\ell_{k,v};2^{-k}r_0)^{2s} \lesssim \beta^{*,c}(\mu,Q)^{2s} \max\{c^{-1},1\}.$$ By \cite[Proposition 1.6]{Li-TSP} or \cite[Lemma 6.2]{Li-TSP}, it follows that there exists $\alpha_{k,v}$ such that $\alpha_{k,v}^{2s}\lesssim \beta^{*,c}(\mu,Q)^{2s}\max\{c^{-1},1\}$ and $$x\in \ell_{k,v}\cdot \delta_{2^{-k}r_0}(B_{\RR^n}(\alpha_{k,v}^s))\quad\text{for all }x\in (V_{k-1},V_k)\cap B(v,65C^\star2^{-k}r_0).$$ In other words, the errors $\alpha_{k,v}$ satisfy \ref{e:line-containment}.
Moreover,
\begin{align*}
    \sum_{k=1}^\infty \sum_{v \in V_k} \alpha_{k,v}^{2s} 2^{-k}r_0
    \lesssim
    \max\{c^{-1}, 1\} \sum_{Q \in \mathcal{T}}\beta^{*,c}(\mu,Q)^{2s} \diam(Q) \sim \max\{c^{-1},1\}S_\mathcal{T} < \infty
\end{align*} by \eqref{2s-sum-finite}. This verifies \eqref{e:Gamma-bound}.

\emph{The Rectifiable Curve.} Therefore, by Proposition~\ref{p:goal}, there exists a rectifiable curve $\Gamma$ in $G$ such that the Hausdorff distance limit $V=\lim_{k\to \infty} V_k$ is contained in $\Gamma$. Moreover,
\begin{align*}
\mathcal{H}^1(\Gamma) &\lesssim r_0 + \sum_{k=1}^\infty \sum_{v \in V_k} \alpha_{k,v}^{2s} 2^{-k}r_0 \lesssim \side\Top(\mathcal{T}) + \max\{c^{-1},1\} S_\mathcal{T}.
\end{align*} In other words, \eqref{T-H1-bound} holds. It remains to prove that $\leaves(\mathcal{T}) \subset \Gamma$ and suffices to show that $\leaves(\mathcal{T})\subset V$.
Pick $y \in \leaves(\mathcal{T})$ so that $y = \lim_{k \to \infty}y_k$
for some sequence of points $y_k \in Q_{k}$, for some infinite chain
$Q_0 \supset Q_1 \supset Q_2 \supset\cdots$ in $\mathcal{T}$.
By maximality of $V_k$ in $Z_k$, for each $k\geq 0$ we may find $v_k\in V_k$ such that $d(v_k,z_{Q_k})<2^{-k}r_0$. Hence
$$ d(y,V) \leq d(y,y_k) + d(y_k,z_{Q_k}) + d(z_{Q_k},v_k) \leq d(y,y_k) + \diam 2B_{Q_k} + 2^{-k}r_0 \rightarrow 0
$$ as $k\rightarrow\infty$, since $\lim_{k\to\infty} y_k=y$.
Thus,  $y \in V$, and therefore, $\leaves(\mathcal{T})\subset V\subset \Gamma$.
\end{proof}

We are ready to prove the theorem.

\begin{proof}[Proof of Theorem~\ref{t-sec5main}]
Let $\mu$ be a locally finite Borel measure on $G$ and $c>0$ be given. We wish to show that the measure $\mu_c$ defined by \eqref{mu-c} is 1-rectifiable. That is, we wish to find a sequence $\Gamma_1,\Gamma_2,\dots$ of rectifiable curves such that $\mu_c(G\setminus\bigcup_1^\infty \Gamma_i)=0$.

Suppose that $x \in G$ has $\lD1(\mu,x) > 2c$. Then there is some radius $r_x>0$ such that $$
\mu(B(x,r)) > 4cr
\quad
\text{for all } 0<r \leq r_x.
$$ Thus, for any $Q \in \Delta$ containing $x$ with $\frac83\side Q\leq r_x$, we have $B(x,\frac83\side Q)\subset 2B_Q$ and
$$
\mu(2B_Q) \geq \mu(B(x,\tfrac83\side Q)) \geq \tfrac{32}{3} c \side Q =c\diam 2B_Q.
$$
Choose $Q_x \in \Delta$ to be the maximal cube containing $x$ with $\frac83\side Q\leq r_x$ and $\side Q\leq 1$.
Then $x \in \leaves(\mathcal{T}_x)$, where
$$
\mathcal{T}_x := \left\{ Q \in \Delta \, : \, Q \subset Q_x \text{ and } \mu(2B_R) \geq c \diam(2B_R) \text{ for all } R \in \Delta \text{ with } Q \subset R \subset Q_x \right\}.
$$
Note that $\mathcal{T}_x=\mathcal{T}_y$ whenever $Q_x=Q_y$ and the collection $\{ Q_x : \lD1(\mu,x)>2c\}$ of cubes is countable, since it is a subset of the countable family $\Delta$. Thus, we may choose a sequence $\{ x_i \}_{i=1}^\infty$ of points in $G$
such that $\underline{D}^1(\mu,x_i) > 2c$ for each $i\geq 1$ and $$\{ x \in G  : \lD1(\mu,x)>2c\}\subset\bigcup_{i =1}^\infty \leaves(\mathcal{T}_{x_i}).$$
Therefore,
\begin{equation*}
\{ x \in G : \underline{D}^1(\mu,x) > 2c,\, J^{*,c}(\mu,x) < \infty \}
\subset
\bigcup_{i =1}^\infty \bigcup_{M=1}^\infty \{ x \in \leaves(\mathcal{T}_{x_i}) :  J^{*,c}(\mu,x) \leq M \}.\end{equation*}
This shows that to prove the measure $\mu_c$ defined in \eqref{mu-c} is 1-rectifiable, it suffices to prove that each measure $\mu\res \{x\in \leaves(T_{x_i}):J^{*,c}(\mu,x)\leq M\}$ is 1-rectifiable.

Fix $i\geq 1$ and $M\geq 1$. Since $\side Q_{x_i}\leq 1$, the set $\{x\in \leaves(\mathcal{T}_{x_i}): J^{*,c}(\mu,x)\leq M\}$ is contained in
$$
A := \left\{ x \in \leaves(\mathcal{T}_{x_i}) \, : \, \sum_{Q \in \mathcal{T}_{x_i}} \beta^{*,c}(\mu,Q)^{2s} \diam Q\, \frac{\chi_Q(x)}{\mu(Q)} \leq M \right\}.
$$ To complete the proof of the theorem, it is enough to prove that $\mu\res A$ is 1-rectifiable. If $\mu(A)=0$, we are done. Suppose that $\mu(A)>0$. By Lemma \ref{l-localize}, applied with the function $b(Q)\equiv\beta^{*,c}(\mu,Q)^{2s}\diam Q$ and $\varepsilon=1/k$, for each $k\geq 2$, there is a subtree $\mathcal{G}_k$ of $\mathcal{T}_{x_i}$
such that $\mu(A \cap \leaves(\mathcal{G}_k)) \geq (1-1/k)\mu(A)$
and $$\sum_{Q \in \mathcal{G}_k} \beta^{*,c}(\mu,Q)^{2s} \diam(Q) < k M\, \mu(Q_{x_i})<\infty.$$
Since the tree $\mathcal{G}_k$ satisfies \eqref{e-lowerreg} and \eqref{2s-sum-finite},
Lemma~\ref{l-makeacurve} produces
a rectifiable curve $\Gamma_k$ in $G$ such that $\leaves(\mathcal{G}_k)\subset \Gamma_k$ and
$$
\mu(A \setminus \Gamma_k)
=
\mu(A) - \mu(A \cap \Gamma_k)
\leq
\mu(A) - \mu(A \cap \leaves(\mathcal{G}_k))
\leq
(1/k)\cdot \mu(A).
$$ Therefore, $\mu\res A$ is 1-rectifiable: \begin{equation*}
\mu \left(A \setminus \textstyle\bigcup_{k=2}^\infty \Gamma_k \right)
\leq
\inf_{k \geq 2} \mu(A \setminus \Gamma_k) \leq \inf_{k \geq 2}  (1/k)\cdot \mu(A)= 0. \qedhere\end{equation*}
\end{proof}

By repeating the arguments above, making minor changes as necessary, one can obtain the following two variants of Theorem \ref{t-sec5main}. For some sample details, see \cite[Lemmas 5.4 and 7.3]{badger-schul}. For all $Q\in\Delta$, define $\beta^{**}(\mu,Q)=\inf_L \max_{R\in\near(Q)} \beta(\mu,2B_R,L)$, where the infimum is over all horizontal lines in $G$. Also define \begin{equation}J^{**}(\mu,x)=\sum_{Q\in\Delta_+} \beta^{**}(\mu,Q)^{2s}\diam Q\,\frac{\chi_Q(x)}{\mu(Q)}\quad\text{for all }x\in G.\end{equation}
\begin{theorem}\label{t-allcubes} If $\mu$ is a locally finite Borel measure on $G$, then the measure given by $\mu\res\{x\in G:J^{**}(\mu,x)<\infty\}$ is 1-rectifiable.
\end{theorem} With $\beta(\mu,Q)$ as in Definition \ref{nh-beta}, define \begin{equation}\tJ(\mu,x) = \sum_{Q\in\Delta_+} \beta(\mu,2B_Q)^{2s}\diam Q\, \frac{\chi_Q(x)}{\mu(Q)}\quad\text{for all }x\in G.\end{equation}
\begin{theorem}\label{t-doubling} If $\mu$ is a locally finite Borel measure on $G$, then the measure $$\mu\res\left\{x\in G: \limsup_{r\downarrow 0}\frac{\mu(B(x,2r))}{\mu(B(x,r))}<\infty\text{ and }\tJ(\mu,x)<\infty\right\}$$ is 1-rectifiable.\end{theorem}

\section{Finiteness of the Jones function on rectifiable curves}
\label{s-necc}

In this section, we show that finiteness of the Jones function defined in \eqref{Jstar} is necessary for a measure to be carried by rectifiable curves; cf.~\cite[\S4]{badger-schul}.
\begin{theorem}
\label{t-necc}
  If $\mu$ is a locally finite Borel measure on a Carnot group $G$ and $\Gamma$ is a rectifiable curve in $G$, then the function $J^*(\mu,\cdot) \in L^1(\mu\res\Gamma)$.  In particular, $J^*(\mu,x) < \infty$ for $\mu$-a.e.~$x \in \Gamma$.
\end{theorem}

At the core of Theorem \ref{t-necc} is the following computation, which incorporates and extends the necessary half of Theorem \ref{t:sean}. A minor difficulty in the proof of Lemma \ref{p:necc} compared with the proof of the corresponding statement in \cite[\S4]{badger-schul} is that we need to use \eqref{e:quasi-triangle}. Recall that $\Delta_{+}$ is the set of $Q\in\Delta$ with $\side Q\leq 1$.

\begin{lemma}\label{p:necc}
If $\nu$ is a finite Borel measure on $G$ and $\Gamma$ is a rectifiable curve in $G$, then
  \begin{align}\label{necc-goal}
    \sum_{\substack{Q\in\Delta_{+}\\ \nu(Q \cap \Gamma) > 0}} \beta^*(\nu, Q)^{2s} \diam Q \lesssim \Haus^1(\Gamma) + \nu(G \backslash \Gamma).
  \end{align}
\end{lemma}
\begin{proof}
Given two sets $E, U \subset G$, define
\begin{align*}
  \tb_E(U) = \inf_L \sup_{z \in E \cap U} \tb(z,L;\diam U),
\end{align*} where as usual the infimum is over all horizontal lines in $G$. In particular, recalling \eqref{sean-beta}, we have $\tb_E(B(x,r)) \leq \beta_E(x,r)\leq 2\tb_E(B(x,r))$ for all $x\in G$ and $r>0$ by \eqref{quasimonotone}.

By \eqref{3R-contain}, $2B_R\subset 597 B_Q$ for all $R\in\near(Q)$. Fix an absolute constant $A=1200$ (this is an overestimate) and a small constant $\varepsilon > 0$ depending only on the step $s$ of $G$ to be determined later. Partition the set of cubes $Q\in\Delta_+$ that intersect the curve $\Gamma$ in a set of positive measure into two classes:
  \begin{align*}
    \Delta_\Gamma &= \{Q \in \Delta_+ : \nu(\Gamma \cap Q) > 0 \text{ and }(\varepsilon/2A) \beta^*(\nu,Q) \leq \tb_\Gamma(AB_Q)\}, \\
    \Delta_\nu &= \{ Q \in \Delta_+ : \nu(\Gamma \cap Q) > 0 \text{ and }(\varepsilon/2A) \beta^*(\nu,Q) > \tb_\Gamma(AB_Q)\}.
  \end{align*}
Then
  \begin{align*}
    \sum_{\substack{Q\in\Delta_{+}\\ \nu(Q \cap \Gamma) > 0}} \beta^*(\nu, Q)^{2s} \diam Q = \sum_{Q \in \Delta_\Gamma} \beta^*(\nu,Q)^{2s} \diam Q + \sum_{Q \in \Delta_\nu} \beta^*(\nu,Q)^{2s} \diam Q.
  \end{align*}
From the definition of $\Delta_\Gamma$, the Analyst's Traveling Salesman Theorem in Carnot groups (Theorem \ref{t:sean}), and \eqref{Lebesgue-q}, it follows that  \begin{align*}
    \sum_{Q \in \Delta_\Gamma} \beta^*(\nu,Q)^{2s} \diam Q
    &\leq
    \sum_{Q \in \Delta_\Gamma} (\varepsilon/2A)^{-2s} \tb_\Gamma(AB_Q)^{2s} \diam B_Q\\
    &\leq(\varepsilon/2A)^{-2s} \sum_{Q \in \Delta} \beta_\Gamma(x_Q,(8A/3)\side Q)^{2s} \diam B_Q \lesssim \Haus^1(\Gamma).
  \end{align*}
To complete the proof of \eqref{necc-goal}, we will show that $\sum_{Q \in \Delta_\nu} \beta^*(\nu,Q)^{2s} \diam Q \lesssim \nu(G \setminus \Gamma).$

Let $Q \in \Delta_\nu$.  By change of scales \eqref{quasimonotone}, the definition of $\tb_\Gamma(AB_Q)$, and the definition of the family $\Delta_\nu$, we can find a horizontal line $L$ in $G$ so that
  \begin{equation}\begin{split} \sup_{z\in\Gamma\cap AB_Q} \tb(z,L;\diam 2B_Q) \leq A\tb_\Gamma(AB_Q) < (\varepsilon/2) \beta^*(\nu,Q). \label{e:Gamma-ell-dist}
  \end{split}\end{equation}
For the same horizontal line $L$, we have that
  \begin{align*}
    \beta^*(\nu,Q)^{2s} \leq \max_{R \in \near(Q)} \beta(\nu, 2B_R, L)^{2s} \min \left\{ 1, \frac{\nu(2B_R)}{\diam 2B_R} \right\} =: \max_{R \in \near(Q)} \beta(\nu, 2B_R, L)^{2s} m_R.
  \end{align*}
Fix $R\in\near(Q)$ and divide $2B_R$ into two sets:
\begin{align*}
    N_R = \{y \in 2B_R : \tb(y,L;\diam 2B_R) \leq \varepsilon \beta^*(\nu,Q) \},\quad F_R = 2B_R\setminus N_R.
\end{align*}
Note that $F_R\subset G\setminus \Gamma$ by \eqref{e:Gamma-ell-dist}. To proceed, write
\begin{align}
    \beta(\nu, 2B_R,L)^{2s} m_R &= \int_{N_R\cup F_R} \tb(y,L;\diam 2B_R)^{2s} m_R \frac{d\nu(y)}{\nu(2B_R)} \nonumber \\
    &\leq \varepsilon^{2s} \beta^*(\nu,Q)^{2s} + \int_{F_R} \tb(y,L;\diam 2B_R)^{2s} m_R \frac{d\nu(y)}{\nu(2B_R)}. \label{e-ellipses}
\end{align}
Note that, since $Q\in\Delta_\nu$, if $\varepsilon$ is very small, then $\tb_\Gamma(AB_Q)$ is very small relative to $\beta^*(\nu,Q)$. This will allow us to effectively replace the horizontal line $L$ appearing in \eqref{e-ellipses} with $\Gamma$. For any $y\in 2B_R$, the inequalities \eqref{e:quasi-triangle}, \eqref{quasimonotone}, and \eqref{e:Gamma-ell-dist}, the fact that $2B_R\subset 597B_Q$ and $\nu(\Gamma\cap Q)>0$, and the choice $A=1200>2\cdot 597 + (\diam Q)/(\side Q)$ give us
\begin{equation*}\begin{split}
\tb(y,L;\diam 2B_R)^{2s} &\leq
 2^{2s-1}\left(\tb(y,\Gamma \cap AB_Q; \diam 2B_R)^{2s} + \sup_{z \in \Gamma \cap AB_Q} \tb(z, L; \diam 2B_R)^{2s}\right) \\ &<
 2^{2s-1} \tb(y,\Gamma \cap AB_Q; \diam 2B_R)^{2s} + (1/2)\varepsilon^{2s}\beta^*(\nu, Q)^{2s}\\ &=
 2^{2s-1} \tb(y,\Gamma; \diam 2B_R)^{2s} + (1/2)\varepsilon^{2s} \beta^*(\nu, Q)^{2s}.
\end{split}\end{equation*}
Combining the previous two displays and using $m_R\leq \nu(2B_R)/\diam 2B_R$, we have \begin{align*}
\beta(\nu,2B_R,L)^{2s} m_R
  &\leq (3/2)\varepsilon^{2s}\beta^*(\nu,Q)^{2s} + 2^{2s-1} \int_{F_R}  \tb(y, \Gamma; \diam 2B_R)^{2s}m_R \frac{d\nu(y)}{\nu(2B_R)} \\
  &\leq (3/2)\varepsilon^{2s}\beta^*(\nu,Q)^{2s} + 2^{2s-1} \int_{F_R} \tb(y, \Gamma; \diam 2B_R)^{2s} \frac{d\nu(y)}{\diam 2B_R}.
\end{align*}
Taking the maximum over all cubes $R \in \near(Q)$, choosing $\varepsilon$ sufficiently small depending only on $s$, rearranging, and using $\diam Q\leq \diam 2B_R$, we obtain
\begin{align*}
    \beta^*(\nu,Q)^{2s}\diam Q \leq 2^{2s} \max_{R \in \near(Q)} \int_{F_R} \tb(y, \Gamma; \diam 2B_R)^{2s} d\nu(y).
\end{align*} As we already noted, each $F_R\subset G\setminus \Gamma$. Thus, by Remark \ref{near-contained} and \eqref{quasimonotone}, \begin{align}\label{beta-nu-bound} \beta^*(\nu,Q)^{2s}\diam Q \lesssim \int_{597B_Q\setminus \Gamma} \tb(y,\Gamma;\side Q)^{2s}\,d\nu(y)\end{align}
Let $\mathcal{W}$ be a Whitney decomposition of $G\setminus \Gamma$ given by Lemma \ref{l-whitney}. For each $j\in\ZZ$, let $$\mathcal{W}_j=\{W\in\mathcal{W}: 2^{-(j+1)}<\dist(W,\Gamma)\leq 2^{-j}\}.$$ For any set $I$, also define $\mathcal{W}(I)=\{W\in\mathcal{W}:\nu(I\cap W)>0\}$ and $\mathcal{W}_j(I)=\mathcal{W}_j\cap\mathcal{W}(I)$. Then, continuing from \eqref{beta-nu-bound}, \begin{align*}\beta^*(\nu,Q)^{2s}\diam Q &\lesssim \sum_{W\in\mathcal{W}(597B_Q)} \sup_{y\in W} \tb(y,\Gamma,\side Q)^{2s}\,\nu(W\cap 597 B_Q)\\
&\lesssim \sum_{i=1}^s \sum_{W\in\mathcal{W}(597B_Q)}\sup_{y\in W} \left(\frac{d_i(\pi_i(y),\pi_i(\Gamma))}{\side Q}\right)^{2i}\nu(W\cap 597 B_Q).\end{align*} Suppose that $\side Q=2^{-k}$. If $W\in\mathcal{W}_j(597B_Q)$, then by bounding the distance of a point in $W\cap 597B_Q$ to a point in $\Gamma\cap Q$, we have $$2^{-(j+1)} \leq \dist(W,\Gamma)\leq \diam 597B_Q \leq 3184\cdot 2^{-k},$$ which implies that $j\geq k-12$. Also if $W\in\mathcal{W}_j$ and $y\in W$, then $d_i(\pi_i(y),\pi_i(\Gamma))\leq \dist(y,\Gamma)\leq \dist(W,\Gamma)+\diam W \leq 2\dist(W,\Gamma)\leq 2\cdot 2^{-j},$ where the first inequality holds because the projections $\pi_i:G\rightarrow G_i$ are 1-Lipschitz and the penultimate inequality is by property (2) of Lemma \ref{l-whitney}. Therefore, \begin{equation}\label{compare-to-bs1}\beta^*(\nu,Q)^{2s}\diam Q \lesssim \sum_{i=1}^s\sum_{j=-\log_2(\side Q)-12}^\infty \sum_{W\in\mathcal{W}_j(597B_Q)} \left(\frac{2^{-j}}{\side Q}\right)^{2i}\nu(W\cap 597B_Q).\end{equation} This estimate is valid for every $Q\in\Delta_\nu$.

Equation \eqref{compare-to-bs1} is analogous to \cite[(3.8)]{BS1} (with step $s=1$). Because the cubes in $\mathcal{W}$ are pairwise disjoint and each of the families $\{597B_Q:Q\in\Delta\text{ and }\side Q=2^{-k}\}$ have bounded overlap, we may repeat the computation in \cite{BS1} \emph{mutatis mutandis} $s$ times to obtain $\sum_{Q\in\Delta_\nu} \beta^*(\nu,Q)^{2s}\diam Q\lesssim \nu(G\setminus \Gamma).$
\end{proof}

We now apply the lemma to prove that $J^*(\mu,\cdot)$ is integrable on any rectifiable curve.

\begin{proof}[Proof of Theorem~\ref{t-necc}]
Let $\Gamma\subset G$ be a rectifiable curve. Integrating the Jones function,
  \begin{align*}
    \int_\Gamma J^*(\mu,x) ~d\mu(x) &= \sum_{Q \in \Delta_+} \beta^*(\mu,Q)^{2s} \frac{\diam(Q)}{\mu(Q)} \int_\Gamma \chi_Q(x) d\mu(x) \\
    &= \sum_{\substack{Q \in \Delta_+\\ \mu(Q \cap \Gamma) > 0}} \beta^*(\mu,Q)^{2s} \diam(Q) \frac{\mu(Q \cap \Gamma)}{\mu(Q)} \leq \sum_{\substack{Q \in \Delta_+\\ \mu(Q \cap \Gamma) > 0}} \beta^*(\mu,Q)^{2s} \diam(Q).
  \end{align*}
Let $K = \overline{\bigcup \{Q \in \Delta_+ : \mu(Q \cap \Gamma) > 0 \}}$ and put $\nu := \mu\res K$.  Since the set $K$ is compact and $\mu$ is locally finite, we have $\nu(G) = \mu(K) < \infty$.  Furthermore, $\mu\res Q = \nu\res Q$ whenever $Q \in \Delta_+$ and $\mu(Q \cap \Gamma) > 0$. Thus, by Lemma \ref{p:necc},
  \begin{equation*}
    \int_\Gamma J^*(\mu,x) ~d\mu(x) \leq \sum_{\substack{Q \in \Delta_+ \\ \nu(Q \cap \Gamma) > 0}} \beta^*(\nu,Q)^{2s} \diam(Q) \lesssim \cH^1(\Gamma) + \nu(G \setminus \Gamma) < \infty. \qedhere
  \end{equation*}
\end{proof}

\begin{corollary}\label{J-pu} Let $\mu$ be any locally finite Borel measure on $G$. Then the measure $\mu\res\{x\in G: J^*(\mu,x)=\infty\}$ is purely 1-unrectifiable.\end{corollary}

\begin{proof} If $\Gamma$ is a rectifiable curve in $G$, then $J^*(\mu,x)<\infty$ at $\mu$-a.e.~$x\in \Gamma$ by Theorem \ref{t-necc}. That is to say, $\mu(\Gamma\cap\{x\in G:J^*(\mu,x)=\infty\})=0$ for every rectifiable curve $\Gamma$.\end{proof}

\section{Proof of Theorem \ref{t-main}}\label{s-proofs}

Equipped with the results from \S\S \ref{s-suff} and \ref{s-necc}, we are ready to the prove the main theorem.
Let $\mu$ be a locally finite Borel measure on $G$. Both the lower density $\lD1(\mu,\cdot)$ and the Jones function $J^*(\mu,\cdot)$ are Borel measurable. Hence
$$R = \left\{ x \in G  :  \underline{D}^1(\mu,x)>0 \text{ and } J^*(\mu,x) < \infty \right\}\text{ and }$$
$$P = \left\{ x \in G  :  \underline{D}^1(\mu,x) = 0 \text{ or } J^*(\mu,x) = \infty \right\}$$ are Borel sets and $G=R\cup P$. By the uniqueness clause of Lemma \ref{l:decomp}, if we show that $\mu\res R$ is 1-rectifiable and $\mu\res P$ is purely 1-unrectifiable, then $$\mu_\rect=\mu\res R\quad\text{and}\quad\mu_\pu=\mu\res P.$$
On the one hand, $J^{*,c}(\mu,x) \leq J^*(\mu,x)$ for all $x \in G$ and $c > 0$ (see \S\ref{s-suff}). Thus,
\begin{align*}
    R
    &=
    \left\{ x \in G  :  \underline{D}^1(\mu,x)>0 \text{ and } J^*(\mu,x) < \infty \right\}\\
    &\subset
    \bigcup_{n=1}^\infty
    \left\{ x \in G \, : \, \underline{D}^1(\mu,x)>2/n \text{ and } J^{*,1/n}(\mu,x) < \infty \right\}=:\bigcup_{n=1}^\infty R_n.
\end{align*} By Theorem \ref{t-sec5main}, $\mu \res R_n$ is 1-rectifiable for each $n\geq 1$. Therefore, $\mu\res R\leq \sum_{n=1}^\infty \mu\res R_n$ is 1-rectifiable, as well.
On the other hand, we can write
$$
P = \left\{ x \in G  :  J^*(\mu,x) = \infty \right\} \cup \left\{ x \in G  :  \underline{D}^1(\mu,x) = 0 \right\}=:P_1\cup P_2.
$$
The measure $\mu\res P_1$ is purely 1-unrectifiable by Corollary \ref{J-pu} and the measure $\mu\res P_2$ is purely 1-unrectifiable by Corollary \ref{c-packing} and Remark \ref{r:lower-density}. Since $\mu\res P\leq \mu\res P_1+\mu\res P_2$, $\mu\res P$ is also purely 1-unrectifiable. This completes the proof of Theorem \ref{t-main}.

\section{Garnett-Killip-Schul-type measures in metric spaces}
\label{s-GKS}

Towards Theorem \ref{t-gks}, suppose that $(X,d)$ is a complete metric space such that \begin{itemize}
\item $X$ is \emph{doubling}, i.e.~there exists a constant $C_{db}\geq 1$ such that every bounded set of diameter $D$ can be covered by $C_{db}$ or fewer sets of diameter $D/2$; and,
\item $X$ is \emph{locally quasiconvex}, i.e.~for every compact set $E\subset X$, there exists a constant $C_{qc, E}\geq 1$ such that for every $x,y\in E$ with $x\neq y$, there exists a parameterized curve $\gamma:[0,1]\rightarrow X$ such that $\gamma(0)=x$, $\gamma(1)=y$, and $\var(\gamma)\leq C_{qc, E}\, d(x,y)$.
\end{itemize} Because $X$ is complete and doubling, there exists a doubling measure $\mu$ on $X$, i.e.~a measure satisfying \eqref{doubling-mu} for all $x\in X$ and $r>0$; for a proof, see \cite[Theorem 3.1]{KRS-cubes} or \cite[Theorem 13.3]{Juha-book}. Let $C_\mu$ denote the doubling constant of $\mu$. Our goal is to construct a doubling measure $\nu$ on $X$ that is 1-rectifiable. We will explicitly construct $\nu$ and rectifiable curves $\Gamma$ with $\nu(\Gamma)>0$ in a manner similar to \cite{GKS}, which handled the particular case that $X=\RR^n$ and $\mu$ is the Lebesgue measure.

\subsection{Construction of the measure $\nu$}\label{ss:nu}

Fix any system $(\Delta_k)_{k\in\ZZ}$ of dyadic cubes on $X$ given by Theorem \ref{t-KRS}. We freely adopt the notation of \S\ref{ss:dyadic}. In particular, to each $Q\in\Delta := \bigcup_{k\in\ZZ}\Delta_k$, we may refer to the center $x_Q$, side length $\side Q$, inner ball $U_Q$, and outer ball $B_Q$ associated to $Q$. For any $j\geq 1$ and $Q \in \Delta_k$, let $\Delta_j(Q)=\{R\in\Delta_{k+j}:R\subset Q\}$ denote the collection of all $j$-th generation descendents of $Q$. Note that $\mu(Q)\geq \mu(U_Q)>0$ for all $Q\in\Delta$ because $\mu$ is doubling. We proved the following facts in Remark \ref{r:halving}.

\begin{lemma} \label{l:cube-doubling} There exists $C_1>0$ depending only on  $C_\mu$ such that $\mu(R)\geq C_1 \mu(Q)$ for all $R\in\Delta_1(Q)$.
\end{lemma}

\begin{corollary} \label{l:growth} There exists $M\geq 1$ depending only on $C_\mu$ such that $\#\Delta_j(Q) \leq M^j$ for all $Q\in\Delta$ and $j>0$.
\end{corollary}

Next, let us show that each cube in $\Delta$ contains a descendent---within a few generations---that is quantitatively far away from the complement of the cube. A similar claim is proved in the paper \cite{KRS-cubes}.

\begin{lemma} \label{l:center}
  For any $n\in\ZZ$ and $Q \in \Delta_n$, there exists some $R \in \Delta_{7}(Q)$ such that $d(R,Q^c) := \inf_{x\in R}\inf_{y\not\in Q} d(x,y) > \frac{1}{8}\cdot 2^{-n}$.
\end{lemma}

\begin{proof}
Fix $n\in\ZZ$ and $Q \in \Delta_n$.
By property (4) of Theorem \ref{t-KRS}, there exists $R\in\Delta_{n+7}$ such that $x_R=x_Q$. Therefore,
\begin{equation*}
d(R,Q^c) \geq d(B_R,U_Q^c)
\geq d(x_Q,U_Q^c) - \sup_{z\in B_R} d(z,x_Q)
\geq \tfrac16 \cdot 2^{-n}
-\tfrac83\cdot 2^{-(n+7)}=\tfrac{7}{48}\cdot 2^{-n}. \qedhere
\end{equation*}
\end{proof}

It will be convenient to thin $\Delta$ by skipping generations and to restrict to cubes starting from a fixed generation. For each integer $n\geq 0$, define \begin{equation}D_n = \Delta_{7n}\quad\text{and}\quad D = \bigcup_{n=0}^\infty D_n.\end{equation}
For all $Q \in D$ and $k\geq 0$, define $D_k(Q)$ to be the $k$-th generation descendants of $Q$ in $D$,  i.e.~$D_k(Q) := \{ R \in D_{n+k} : R \subset Q\}$.
By Lemma \ref{l:center}, for each $Q \in D_n$, we may choose some cube $R_Q \in D_1(Q)$
such that \begin{equation}\label{RQ-def} d(R_Q,Q^c) > \tfrac18\cdot2^{-7n}=16\cdot 2^{-7(n+1)}.\end{equation}

\begin{figure}
\begin{center}\includegraphics[width=.66\textwidth]{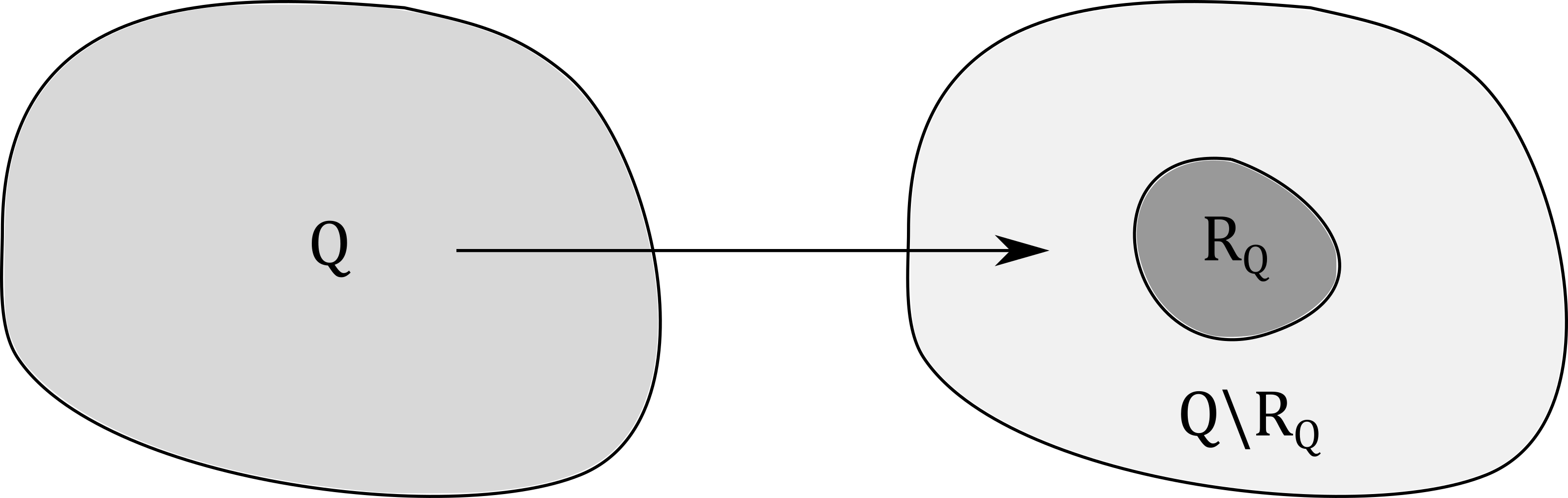}\end{center}
\caption{To define $f_Q\,d\mu$, redistribute the mass $\mu(Q)$ so that more mass is assigned to $R_Q$ and less mass is assigned to $Q\setminus R_Q$.}\label{fig:mu-to-nu}\end{figure}

Let $0<\delta\ll 1$ be a constant whose value will be fixed later; see \eqref{e:delta-bound}. For each $Q\in D$, we define a Borel measure $\nu_Q$ on $X$ that is absolutely continuous with respect to $\mu$ by defining its Radon-Nikodym derivative as a sum of indicator functions:
\begin{equation}\label{fQ-weight}
  f_Q := \frac{d\nu_Q}{d\mu} = a_Q \chi_{R_Q} + \delta \chi_{Q \setminus R_Q},
\end{equation} where $a_Q > 0$ is chosen so that $\nu_Q(Q) = \mu(Q)$. Note that $\nu_Q(Q^c)=0$. See Figure \ref{fig:mu-to-nu}.

\begin{lemma} \label{l:p-bound}  For all $Q \in D$, we have $\nu_Q(R_Q) \geq (1-\delta) \nu_Q(Q)$.
\end{lemma}

\begin{proof} Because $\mu(Q)=\nu_Q(Q)$, we have \begin{align*}
\nu_Q(R_Q)= \nu_Q(Q) - \nu_Q(Q \setminus R_Q)&=\nu_Q(Q) - \delta\, \mu(Q \setminus R_Q)\\
&\geq \nu_Q(Q) - \delta\, \mu(Q)= (1-\delta)\nu_Q(Q).\qedhere\end{align*}
\end{proof}

\begin{lemma} \label{l:a-bound}
  There is a constant $C_2 \geq 1$ depending only on $C_\mu$ such that $\sup_{Q \in D} a_Q \leq C_2$.
\end{lemma}

\begin{proof}
  Since $D_1(Q)=\Delta_{7}(Q)$, iterating Lemma \ref{l:cube-doubling} gives $\mu(R_Q) \geq C_1^{-7} \mu(Q)$ for all $Q \in D$.  We defined $a_Q$ so that  \begin{align*}
    \mu(Q) = \nu_Q(Q) = a_Q\mu(R_Q) + \delta\,\mu(Q \setminus R_Q) = a_Q \mu(R_Q) + \delta\, \mu(Q) - \delta\, \mu(R_Q).
  \end{align*}
  Hence $a_Q = \delta + (1 - \delta)\mu(Q)/\mu(R_Q)\leq 1 + C_1^{\,7}=: C_2$.
\end{proof}

To define the measure $\nu$, we iterate the construction of $f_Q\,d\mu_Q$ and pass to a limit. Formally, for each $k \geq 0$, we define $f_k = \sum_{Q \in D_k} f_Q$.  Using these weights, for each $n\geq 0$, we define a Borel measure $\nu_n$ by setting
\begin{align}\label{fn-weight}
  d\nu_n = \left( \prod_{k=0}^n f_k \right) d\mu.
\end{align}
See Figure \ref{fig:mu-to-nu-2}. Finally, we define the measure $\nu$ to be a weak-$*$ limit of $\nu_n$ (along some subsequence).

\begin{figure}
\begin{center}\includegraphics[width=.66\textwidth]{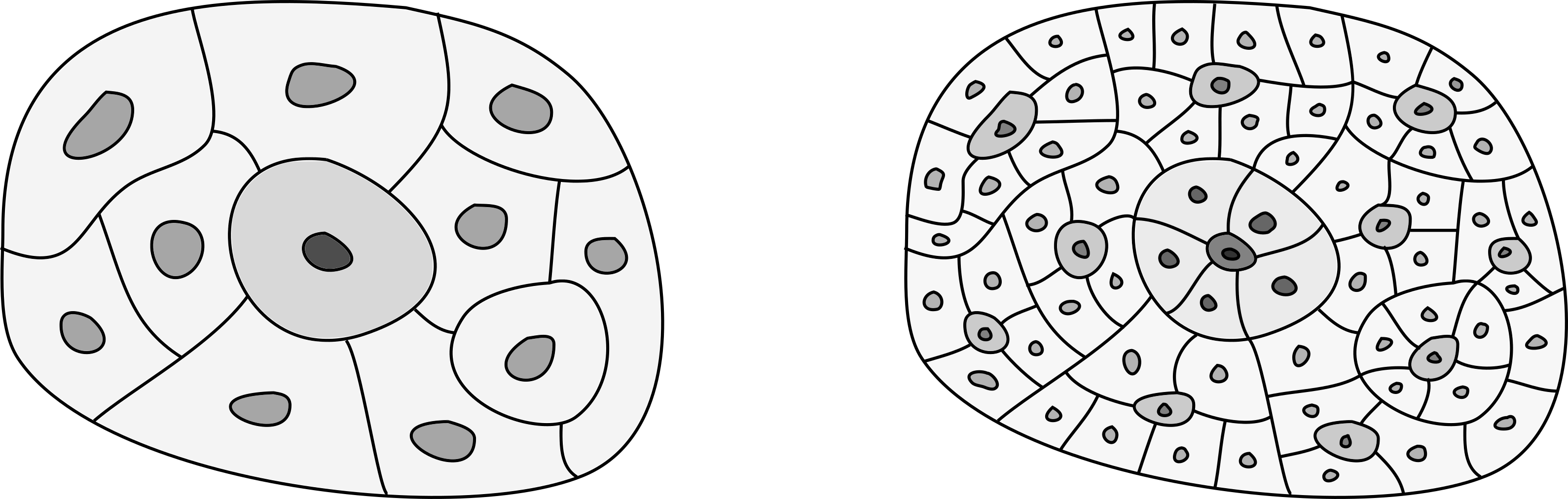}\end{center}
\caption{Possible densities $f_0f_1|_Q$ (left) and  $f_0f_1f_2|_Q$ (right).}\label{fig:mu-to-nu-2}\end{figure}

\begin{lemma} For all $n\geq 0$ and $Q\in D_n$, we have $\nu(\partial Q)=0$ and $\nu(Q)=\nu_{n-1}(Q)$. (When $n=0$, this should be read as $\nu(Q)=\mu(Q)$ for all $Q\in D_0$.)\end{lemma}

\begin{proof} From \eqref{fQ-weight} and \eqref{fn-weight}, it is immediate that $\nu_p(Q)=\nu_{n-1}(Q)$ for all $p\geq n\geq 0$ and $Q\in D_n$. If we can show that $\nu(\partial Q)=0$ for all $Q\in\Delta$, then $\nu(Q)=\lim_{p\rightarrow \infty}\nu_p(Q)=\nu_{n-1}(Q)$ for all $Q\in D_n$ by weak convergence.

Fix $Q\in D_n$ for some $n\geq 0$. To prove that $\nu(\partial Q)=0$, we must find a good cover of the boundary. To that end, let $\mathcal{A}$ denote the family of all $A\in D_n$ such that $\overline{A}\cap \partial Q\neq\emptyset$. Each cube $A\in\mathcal{A}$ is \emph{adjacent} to $Q$. Because $\mu$ is doubling and the sets $\{U_A:A\in\mathcal{A}\}$ are pairwise disjoint and confined to a bounded region of $X$, the collection $\mathcal{A}$ is finite. We will cover $\partial Q$ with certain subsets of the adjacent cubes. Given $A\in\mathcal{A}$ and $k\geq 1$, let $F_{A,k}=A\setminus \bigcup_{j=0}^{k-1} \bigcup_{S\in D_j(A)} R_S$. That is, form $F_{A,k}$ from $A$ by removing any central descendants $R_S$ of $A$ through $k$ generations. By Lemma \ref{l:center}, $\partial Q$ is contained in some open subset $V_k$ of $\bigcup_{A\in\mathcal{A}} F_{A,k}$ for each $k\geq 1$. By weak convergence, monotonicity, and subadditivity of measures, and the fact that $F_{A,k}$ is a union of cubes in $D_{n+k}$, $$\nu(\partial\Omega)\leq \nu(V_k)\leq \liminf_{m\rightarrow\infty} \nu_m(V_k)\leq \liminf_{m\rightarrow\infty}\sum_{A\in\mathcal{A}}\nu_m(F_{A,k})=\sum_{A\in\mathcal{A}} \nu_{n+k-1}(F_{A,k})$$ for all $k\geq 1$. Because $F_{A,k}$ is formed by deleting $k$ generations of central descendants, $\nu_{n+k-1}(F_{A,k}) = \delta^k\nu_{n-1}(A)$ for all $k\geq 1$. Because $\sum_{A\in\mathcal{A}}\nu_{n-1}(A)<\infty$, we conclude that $\nu(\partial\Omega) \leq \lim_{k\rightarrow\infty} \delta^k\sum_{A\in\mathcal{A}}\nu_{n-1}(A)=0$.\end{proof}

\subsection{Doubling of $\nu$}

\begin{lemma} \label{l:adjacent} There is a constant $C_3\geq 1$ depending only on $C_\mu$ and $\delta$ so that if $S\in D_n$ for some $n\geq 0$ and $\mathcal{N}(S)=\{T\in D_n: d(S,T)\leq 2048 \cdot 2^{-7n}\}$, then
  \begin{align} \label{e:adjacent-bounds}
    C_3^{-1} \nu(S) \leq \nu(T) \leq C_3\, \nu(S)\quad\text{for all }T\in\mathcal{N}(S).
  \end{align}
\end{lemma}

\begin{proof} Let $S\in D_n$ and $\mathcal{N}(S)$ be fixed as in the statement.  To proceed, let $T\in\mathcal{N}(S)$. There are two cases.

\emph{Case 1.} Suppose that $S$ and $T$ have a common ancestor in $D$. Let $k \geq 0$ be the largest integer such that
  $S \subset Q_0$ and $T \subset Q_0$ for some $Q_0 \in D_k$. In other words, let $Q_0$ be the first common ancestor of $S$ and $T$. We claim that neither $S$ nor $T$ is contained in $R_Q$ for any $Q \in \bigcup_{j={k+1}}^{n-2} D_j$. Indeed, first suppose to get a contradiction that $S \subset R_Q$ for some $Q \in D_j$ with $k+1 \leq j \leq n-2$. Then $T \cap Q = \emptyset$, since $S\subset Q$ and $Q$ is not a common ancestor of $S$ and $T$. Hence
    $$
    d(R_Q,Q^c) \leq d(S,T) \leq 2048 \cdot 2^{-7n} \leq \frac{1}{8} \cdot 2^{-7j},
    $$
    where we used the fact that $j \leq n-2$.  This violates \eqref{RQ-def}.

  An identical argument
  implies $T \nsubseteq R_Q$
  for any $Q\in\bigcup_{j={k+1}}^{n-2} D_j$.
  The consequence of this is that $f_j(x) = \delta = f_j(y)$ for all $x \in S$ and $y \in T$ when $k+1\leq j\leq n-2$. Also $f_j(x) = f_j(y)$ for all $x \in S$ and $y \in T$ when $0\leq j \leq k-1$, since $Q_0 \in D_k$ is a common ancestor of $S$ and $T$. Hence only $f_k$ and $f_{n-1}$ may have different values for $x$ and $y$.  Thus, Lemma~\ref{l:a-bound} gives
  \begin{align*}
    \frac{\left( \prod_{j=0}^{n-1} f_j(x) \right)}{\left( \prod_{j=0}^{n-1} f_j(y) \right)} = \frac{f_k(x)f_{n-1}(x)}{f_k(y)f_{n-1}(y)} \in \left[\delta^2/C_2^2, C_2^2/\delta^2\right].
  \end{align*}

  \emph{Case 2.} Suppose that $S$ and $T$ do not have a common ancestor in $D$. Repeating the argument above informs us that neither $S$ nor $T$ is contained in $R_Q$ for any $Q \in \bigcup_{j=0}^{n-2} D_j$.  It follows that $f_j(x) = \delta = f_j(y)$ for all $x \in S$ and $y \in T$ when $0 \leq j \leq n-2$. Again, by Lemma~\ref{l:a-bound}, we have
\begin{align*}
  \frac{\left( \prod_{j=0}^{n-1} f_j(x) \right)}{\left( \prod_{j=0}^{n-1} f_j(y) \right)} = \frac{f_{n-1}(x)}{f_{n-1}(y)} \in [\delta/C_2, C_2/\delta].
\end{align*}

In each case, \begin{equation*}
(\delta/C_2)^2\mu(S)\leq \nu(S)\leq (C_2/\delta)^2\mu(S)\quad\text{and}\quad (\delta/C_2)^2\mu(T)\leq \nu(T)\leq (C_2/\delta)^2\mu(T).\end{equation*} The lemma follows, because $\mu$ is a doubling measure and $T \in \mathcal{N}(S)$; cf.~Remark \ref{r:halving}.
\end{proof}

\color{black}

\begin{lemma}
  There is a constant $C_4\geq 1$ depending only on $C_\mu$ and on $\delta$ so that
  \begin{align}
    C_4^{-1} \mu(Q) \leq \nu(Q) \leq C_4\, \mu(Q) \quad \text{for all }Q \in D_1. \label{e:D1-bound}
  \end{align}
\end{lemma}

\begin{proof} If $Q\in D_1$, then either $\nu(Q)=\delta\,\mu(Q)$ or $\nu(Q)=a_P\,\mu(Q)$, where $P\in D_0$ is the parent of $Q$ in $D$. Hence $\delta\, \mu(Q)\leq \nu(Q)\leq C_2\,\mu(Q)$ for all $Q\in D_1$ by Lemma \ref{l:a-bound}. Therefore, we may take $C_4 = \max\{\delta^{-1},C_2\}$.
\end{proof}

\begin{proposition}\label{prop:doubling}
  The measure $\nu$ is doubling.
\end{proposition}

\begin{proof}
  Let $B(x,r)$ be a ball in $X$.

\emph{Case 1.} Assume that $r \leq {\frac{16}{3}}$. Then there exists a unique integer $j\geq 0$ such that $$\tfrac{16}{3} \cdot 2^{-7j}\leq r<\tfrac{16}{3}\cdot 2^{-7(j-1)}.$$ Since $D_j$ partitions $X$, there exists a unique cube $S\in D_j$ such that $x\in S$. On the one hand, since $r\geq \diam S$, we have $B(x,r)\supset S$ and
  \begin{align}
    \nu(B(x,r)) \geq \nu(S). \label{e:lower-doubling}
  \end{align}

Let $\mathcal{T}$ denote all cubes of $D_j$ that intersect $B(x,2r)$.  Thus, $\nu(B(x,2r)) \leq \sum_{T \in \mathcal{T}} \nu(T)$.
If $T\in \mathcal{T}$, then
  \begin{align*}
    d(S,T) \leq d(x,T)\leq 2r < \tfrac{32}{3} \cdot 2^{-7(j-1)} < 1366 \cdot 2^{-7j}
  \end{align*}
  and so $T \in \mathcal{N}(S)$ as defined in Lemma \ref{l:adjacent}. This lemma implies that $\nu(T) \leq C_3 \nu(S)$ for all $T \in \mathcal{T}$. Hence
  \begin{align}
    \nu(B(x,2r)) \leq \sum_{T \in \mathcal{T}} C_3 \nu(S) = \# \mathcal{T} \cdot C_3 \nu(S). \label{e:upper-doubling}
  \end{align}
  The proposition will follow from \eqref{e:lower-doubling} and \eqref{e:upper-doubling} in this case once we have shown that $\# \mathcal{T}$ is uniformly bounded.
  Indeed, for all $T \in \mathcal{T}$, we have
  \begin{align*}
      d(x_S,T) + \diam T \leq \diam S + d(s,T) + 2 \diam T < 1382 \cdot 2^{-7j}.
  \end{align*}
    This implies that $T \subset B(x_S,1382 \cdot 2^{-7j})$ and we also know that $T\supset U_T=U(x_T,\tfrac{1}{6}\cdot 2^{-7j})$. Thus, because $\mu$ is doubling, $\#\mathcal{T}\lesssim_{C_\mu} 1$; cf.~the argument in Remark \ref{r:halving}.

\emph{Case 2.}  Now assume $r > \frac{16}{3}$.  Let
  \begin{align*}
    S_1 = \bigcup \{Q \in D_1 : Q \cap B(x,2r) \neq \emptyset\} \quad\text{and}\quad   S_2 = \bigcup \{Q \in D_1 : Q \cap B(x,r/2) \neq \emptyset\}.
  \end{align*}
  As elements of $D_1$ have diameters bounded by $16/3 \cdot 2^{-7} < r/2$, we get the containments
  \begin{align*}
    B(x,2r) \subset S_1 \subset B(x,4r) \quad \text{and}\quad
    B(x,r/2) \subset S_2 \subset B(x,r).
  \end{align*}
  We now can bound
  \begin{align*}
    \nu(B(x,2r)) \leq \nu(S_1) &\overset{\eqref{e:D1-bound}}{\leq} C_4\, \mu(S_1) \leq C_4\, \mu(B(x,4r)) \leq C_4\, C_\mu^3\, \mu(B(x,r/2)) \\
    &\leq C_4\, C_\mu^3\, \mu(S_2) \overset{\eqref{e:D1-bound}}{\leq} C_4^2\, C_\mu^3\, \nu(S_2) \leq C_4^2\, C_\mu^3\, \nu(B(x,r)). \qedhere
  \end{align*}
\end{proof}

\color{black}

\subsection{Cubes with high density}

For $0 \leq k \leq n$ and $Q \in D$, we define $\cK_Q(n,k)$ to be the collection of cubes $S \in D_n(Q)$
for which there exist at least $n-k$ distinct cubes $T \in \bigcup_{j=0}^{n-1} D_j(Q)$ such that $S \subset R_T$. We remark  that $$\cK_Q(n,k)\subset \cK_Q(n,l)\quad\text{when}\quad 0\leq k\leq l\leq n,$$ with $\#\cK_Q(n,0)=1$ and $\cK_Q(n,n)= D_n(Q)$. When $k\ll n$, the cubes $S\in\cK_Q(n,k)$ have relatively high density $\nu(S)/\mu(S)$ compared to $\nu(Q)/\mu(Q)$.

\begin{lemma} \label{l:mass-lower} If $k \geq \delta n$, then
  $\nu \left(\bigcup \cK_Q(n,k)\right)
  \geq
  \left(1- \exp \left[ - \frac{n}{8} \left( \frac{k}{n}-\delta\right)^2 \right] \right) \nu(Q)$.
\end{lemma}

\begin{proof}
  Fix $Q\in D$. Without loss of generality, we may assume that $\nu(Q) = 1$.  This will allow us to adopt a probabilistic view. Let $\P$ denote the probability measure $\nu\res Q$ and let $\E$ denote the corresponding expectation.

For $j \geq 1$, define $D_j' := \{R_T : T \in D_{j-1}(Q)\}$ and the random variable $Y_j = \sum_{S \in D_j'} {\bf 1}_S$.  By Lemma \ref{l:p-bound}, we have $\E[Y_j] \geq 1 - \delta$. From the definition of $Y_j$ and the nested nature of the $D_k$'s, it is apparent that the random variables
  \begin{align*}
    X_0 = 0, \qquad X_j = \sum_{i=1}^j (Y_i - \E[Y_i]) \quad \text{for all }j \geq 1,
  \end{align*}
  form a martingale with respect to the filtration generated by $\{D_j:j\geq 1\}$. Furthermore, $|X_j - X_{j-1}| = \left|Y_j - \E[Y_j]\right| \leq 2$ for all $j$.  Thus, we may bound
  \begin{align*}
    \P\left[ \sum_{j=1}^n Y_j < n-k \right] = \P\left[X_n < n - k - \sum_{j=1}^n \E[Y_j] \right] &\leq \P[X_n - X_0 <  \delta n - k] \\
    &\leq \exp\left[ - \frac{( \delta n - k)^2}{8n} \right],
  \end{align*}
  where  the final estimate holds by Azuma's inequality (see e.g.~\cite[Theorem 7.2.1]{alon-spencer}) provided that $\delta n-k\leq 0$. The lemma follows, because $\bigcup \cK_Q(n,k) = \{ \sum_{j=1}^n Y_j \geq n-k \}$.
\end{proof}

\begin{lemma}
  There exists a constant $C_5 \geq 1$ depending only on $C_\mu$ so that
  \begin{align}
    \# \cK_Q(n,k) \leq \left(C_5\,\frac{n}{k}\right)^k \quad \text{for all }Q \in D. \label{e:num-bound}
  \end{align}
\end{lemma}

\begin{proof}
  By Corollary \ref{l:growth}, we can index each child in $D_1(Q)$ of a cube $Q$ by a character in $\mathcal{A}=\{1,\dots,M^7\}$.
  We make the convention that $R_Q$ is indexed by 1.  We can then continue indexing all descendants via strings of characters in $\mathcal{A}$ in an obvious way so that cubes in $D_n(Q)$ are length $n$ strings.

By our indexing convention and the definition of $\cK_Q(n,k)$, we see that $\# \cK_Q(n,k)$ is no greater than the number of length $n$ strings of characters in $\mathcal{A}$ with at least $n-k$ of the characters equal to 1.  We can bound this quantity by $\binom{n}{n-k} M^{7k}$, since $\binom{n}{n-k}$ equals the number ways in which $n-k$ characters equal to 1 can be chosen
  and $M^{7k}$ bounds the number of all possible choices of characters in the other $k$ positions. Therefore,
  \begin{align*}
    \# \cK_Q(n,k)
    \leq
    \binom{n}{n-k} M^{7k}
    \leq \frac{n^k}{k!} M^{7k} \leq \left( M^7e\frac{n}{k} \right)^k,
  \end{align*}
  where we used the Taylor series of $e^x$ to write $k^k/k! < e^k$.\end{proof}

\subsection{Rectifiable curves with significant $\nu$ measure}

For this subsection, let $Q_1\in D$ be fixed. Our goal is to find a rectifiable curve $\Gamma=\Gamma(Q_1)$ such that $\nu(\Gamma\cap Q_1)>0$, quantitatively. Let $Q_0\in D_0$ denote the unique cube of side length 1 such that $Q_1\subset Q_0$. Since $X$ is locally quasiconvex, there exists a constant $C_{qc, Q_0}\geq 1$ such that any two points $x,y\in Q_0$ can be connected by a rectifiable curve $\Gamma_{x,y}$ in $X$ with $\Haus^1(\Gamma_{x,y})\lesssim_{C_{qc, Q_0}}d(x,y)$. (We do not claim (and it is not true in general) that $\Gamma_{x,y}$ is contained in $Q_0$.)

To proceed, given a cube $Q \in D(Q_0)$ and $0 \leq k \leq n$, we define an auxiliary curve $\Gamma_Q(n,k)$ as follows: for each $S \in \cK_Q(n,k)$, connect $x_S$ to $x_Q$ with a curve of length at most $C_{qc, Q_0} \diam Q$,
where $C_{qc, Q_0}$ is the local quasiconvexity constant of $X$, described in the previous paragraph.
The set $\Gamma_Q(n,k)$ is then defined to be the union of these curves.
For all $Q\in D_m(Q_0)$, we have the bound
\begin{align}
  \cH^1(\Gamma_Q(n,k)) \leq C_{qc, Q_0} \diam(Q) \cdot \# \cK_Q(n,k) \overset{\eqref{e:num-bound}}{\leq} \tfrac{16}{3} C_{qc, Q_0} \cdot 2^{-7m} \left(\frac{C_5n}{k} \right)^k. \label{e:piece-bound}
\end{align}
Recalling that $C_5$ does not depend on $\delta$,
we may finally fix $\delta > 0$ sufficiently small and $n_1 \in \N$ so that
\begin{align}
  \left( \frac{C_5}{2\delta} \right)^{2\delta} \leq 64 = 2^6 \label{e:delta-bound}
\end{align}
and such that $k_1 = 2\delta n_1$ is an integer.
We now construct a sequence $(n_i,k_i)_{i=1}^\infty$ by defining $n_j = j n_1$ and $k_j = j k_1$, and note that $n_j/k_j = (2\delta)^{-1}$, for all $j \in \mathbb{N}$.

Recall that $Q_1\in D$ is fixed and $Q_0\in D_0$ is the unique cube of side length 1 such that $Q_1\subset Q_0$. We now construct a curve $\Gamma=\Gamma(Q_1)$ that captures a significant portion of the mass of $\nu\res Q_1$. Define $\cK_0 := \{Q_1\}$ and $K_0:=\bigcup\cK_0=Q_1$.  Assuming $\cK_{j-1}$ is defined for some $j\geq 1$, we next define $\cK_{j} := \bigcup_{Q \in \cK_{j-1}} \cK_Q(n_{j},k_{j})$ and $K_j := \bigcup \cK_j$.  Note that $K_{j} \subset K_{j-1}$, and
\begin{align*}
  \# \cK_j \overset{\eqref{e:num-bound}}\leq \#\cK_{j-1} \left( \frac{C_5 n_j}{k_j} \right)^{k_j} = \# \cK_{j-1} \left( \frac{C_5}{2\delta} \right)^{k_j}.
\end{align*}
Iterating this estimate gives
\begin{align}
  \# \cK_j \leq \left( \frac{C_5}{2\delta} \right)^{k_1 + \cdots + k_j}. \label{e:cK-bound}
\end{align}
We now define
\begin{align*}
  \Gamma = \bigcup_{j=1}^\infty \bigcup_{Q \in \cK_{j-1}} \Gamma_Q(n_{j}, k_{j}) \cup \bigcap_{j=1}^\infty K_j.
\end{align*} Note that $\Gamma$ is closed. Furthermore, as $\Gamma_Q(n_{j},k_{j})$ connects $x_S$ to $x_Q$ for each cube $S \in \cK_Q(n_{j}, k_{j})$, the set $\Gamma$ is path-connected.

The proof of Theorem \ref{t-gks} is a short step from the next two lemmas.

\begin{lemma}\label{finite-length} $\Gamma=\Gamma(Q_1)$ is a rectifiable curve with $\Haus^1(\Gamma)\lesssim C_{qc, Q_0} \diam Q_1$.
\end{lemma}

\begin{proof} Fix $\ell \geq 1$ and $\eta = \tfrac{16}{3} \cdot 2^{-7(n_1 + \cdots + n_\ell)}$. For every $Q\in\mathcal{K}_\ell$, we have $\diam Q\leq \diam B_Q\leq \eta$. Hence
  \begin{align*}
    \cH^1_\eta\left( \bigcap_{j=1}^\infty K_j \right) \leq \cH^1_\eta(K_\ell)
    \overset{\eqref{e:cK-bound}}\leq \tfrac{16}{3} \cdot 2^{-7(n_1 + \cdots + n_\ell)} \left( \frac{C_5}{2\delta} \right)^{k_1 + \cdots + k_\ell}
    \overset{\eqref{e:delta-bound}}\leq \tfrac{16}{3} \cdot 2^{-(n_1 + \cdots + n_\ell)}.
  \end{align*}
  Since $n_1 + \cdots + n_\ell \to \infty$ and $\eta \to 0$ as $\ell \to \infty$, we get that $\cH^1\left( \bigcap_j K_j \right) = 0$.  Thus,
  \begin{align*}
    \cH^1(\Gamma) \leq \sum_{j=1}^\infty \sum_{Q \in \cK_{j-1}} \cH^1(\Gamma_Q(n_{j},k_{j})).
  \end{align*}
  As the only cube in $\cK_0$ is $Q_1$ and the cubes of $\cK_j$ are in $D_{n_1 + \cdots + n_j}(Q_1)$ whenever $j\geq 1$, we have (interpreting $n_1+\cdots+n_{j-1}\equiv 0$ and $k_1+\cdots+k_{j-1}\equiv 0$ when $j=1$)
  \begin{align*}
    \cH^1(\Gamma) &\overset{\eqref{e:piece-bound}}
    \leq
    \tfrac{16}{3}C_{qc, Q_0}\side Q_1 \sum_{j=1}^\infty \# \cK_{j-1}\cdot
    2^{-7(n_1 + \cdots + n_{j-1})} \left(\frac{C_5 n_{j}}{k_{j}} \right)^{k_{j}} \\ &\overset{\eqref{e:cK-bound}}{\leq}
    \tfrac{16}{3}C_{qc, Q_0}\side Q_1 \sum_{j=1}^\infty
    2^{-7(n_1 + \cdots + n_{j-1})} \left( \frac{C_5}{2\delta} \right)^{k_1 + \cdots + k_{j}}\\
    &= \tfrac{16}{3}C_{qc, Q_0}\side Q_1 \sum_{j=1}^\infty \left[2^{-7} \left( \frac{C_5}{2\delta} \right)^{2\delta} \right]^{n_1 + \cdots + n_{j}} 2^{7n_{j}} \\
    &\overset{\eqref{e:delta-bound}}{\leq}
    \tfrac{16}{3}C_{qc, Q_0}\side Q_1 \sum_{j=1}^\infty
    2^{-(n_1 + \cdots + n_{j})} 2^{7n_{j}}\lesssim {C_{qc, Q_0}}\side Q_1\lesssim{C_{qc, Q_0}}\diam Q_1.
  \end{align*}
  In the last line, we used $\sum_{j=1}^\infty 2^{-(n_1 + \cdots + n_{j})} 2^{7n_{j}}=\sum_{j=1}^\infty 2^{-\frac12 j(j+1)n_1+7jn_1}\lesssim 1$, since $n_j=jn_1$; indeed, the tail of the series is dominated by a convergent geometric series.
  Because $X$ is a complete metric space and $\Gamma\subset X$ is nonempty, closed, connected, and $\mathcal{H}^1(\Gamma) < \infty$, Lemma~\ref{l-waz-2} implies that $\Gamma$ is a rectifiable curve.
\end{proof}

\begin{lemma}\label{lots-of-measure}
  $
    \nu(\Gamma\cap Q_1) \geq \varepsilon \nu(Q_1)
  $ for some constant $\varepsilon\in(0,1)$ determined by $\delta$ and $n_1$. In particular, $\varepsilon$ is independent of $Q_1$.
\end{lemma}

\begin{proof}
  As $K_{j+1} \subset K_j$, we have by the dominated convergence theorem that
  \begin{align*}
    \nu(\Gamma\cap Q_0) \geq \nu\left( \bigcap_{j=1}^\infty K_j \right) = \lim_{j \to \infty} \nu(K_j).
  \end{align*}
  By the construction of $\cK_j$ and Lemma \ref{l:mass-lower}, we have
  \begin{align*}
    \nu(K_j) \geq \left( 1- e^{ - n_j \delta^2/8 }\right) \nu(K_{j-1}) \geq \nu(Q_1) \prod_{i=1}^j (1-e^{-n_i\delta^2/8}).
  \end{align*}
  This product converges to a nonzero number as $\sum_{i=1}^\infty e^{-n_i\delta^2/8}$ is a convergent geometric series (since $n_i = i n_1$).  This proves the lemma.
\end{proof}

\subsection{Proof of Theorem \ref{t-gks}} Let $\nu$ be the measure defined in \S\ref{ss:nu}. By Proposition \ref{prop:doubling}, $\nu$ is a doubling measure on $X$. As $\nu=\sum_{Q_0\in D_0} \nu\res Q_0$ and $D_0$ is countable, to prove that $\nu$ is 1-rectifiable, it will suffice to check that $\nu\res Q_0$ is 1-rectifiable for each $Q_0\in D_0$.

Fix $Q_0\in D_0$. By the above discussion (see Lemmas \ref{finite-length} and \ref{lots-of-measure}), there exists a rectifiable curve $\Gamma=\Gamma(Q_0)$ such that $\nu(Q_0\setminus \Gamma)\leq (1-\varepsilon)\nu(Q_0)$ for some constant $\varepsilon\in(0,1)$ independent of $Q_0$.

Suppose for induction that for some $k\geq 1$ we have found a finite family $\mathscr{C}_k$ of rectifiable curves such that $\nu(Q_0\setminus\bigcup\mathscr{C}_k)\leq (1-\frac12\varepsilon)^k\nu(Q_0)$.  Since the set $\bigcup\mathscr{C}_k$ is closed (being a finite union of closed sets), we may write $Q_0\setminus \bigcup{\mathscr{C}_k}$ as a countable union of pairwise disjoint cubes $Q_1,Q_2,\cdots\in D(Q_0)$. Once again, for each $i\geq 1$, we can find a rectifiable curve $\Gamma_i$ such that $\nu(Q_i\setminus \Gamma_i)\leq (1-\varepsilon)\nu(Q_i)$ by Lemmas \ref{finite-length} and \ref{lots-of-measure}. All together, $$\nu\left(Q_0\setminus \left( \bigcup\mathscr{C}_k\cup\bigcup_{i=1}^\infty\Gamma_i\right)\right) \leq (1-\varepsilon)\sum_{i=1}^\infty \nu(Q_i) = (1-\varepsilon)\nu\left(Q_0\setminus\bigcup\mathscr{C}_k\right).$$ Thus, by continuity from above and the induction hypothesis, we can find $j\geq 1$ sufficiently large such that $$\nu\left(Q_0\setminus \left( \bigcup\mathscr{C}_k\cup\bigcup_{i=1}^j\Gamma_i\right)\right)\leq (1-\tfrac12\varepsilon)\nu\left(Q_0\setminus\bigcup\mathscr{C}_k\right)\leq (1-\tfrac12\varepsilon)^{k+1}\nu(Q_0).$$ Hence $\mathscr{C}_{k+1}:=\mathscr{C}_k\cup\{\Gamma_1,\dots,\Gamma_j\}$ satisfies the next step of the induction.

Finally, $\mathscr{C}=\bigcup_{k=1}^\infty\mathscr{C}_k$ is a countable family of rectifiable curves and $$\nu\left(Q_0\setminus \bigcup\mathscr{C}\right) \leq \inf_{k\geq 1} \nu\left(Q_0\setminus\bigcup\mathscr{C}_k\right) \leq \inf_{k\geq 1} (1-\tfrac12\varepsilon)^k\nu(Q_0)=0.$$ Therefore, $\nu\res Q_0$ is 1-rectifiable. This completes the proof of Theorem \ref{t-gks}.

\bibliographystyle{plain}
\bibliography{biblio}

\end{document}